 \documentclass[final,1p,times,numbers]{elsarticle}
\usepackage{amssymb}
\usepackage{amsfonts}
\usepackage{amsmath,amsthm}
\usepackage{tikz-cd}
\usetikzlibrary{shapes}
\usetikzlibrary{arrows}
\usepackage{tikz}
\usepackage{cancel}
\usepackage{multicol}
\usepackage{listings}
\usepackage{caption}
\newtheorem{definition}{Definition}
\newtheorem{proposition}{Proposition}

\newtheorem{example}{Example}
\newtheorem{theorem}{Theorem}
\newtheorem{corollary}{Corollary}
\newtheorem{lemma}{Lemma}

\newtheorem{notation}{Notation}
\newtheorem{observation}{Observation}

\newtheorem{note}{Note}
\newcommand{\CommaBin}{\mathbin{,}}
\newcommand{\ra}{\rightarrow}

\newcommand{\bit}{\begin{itemize}}
\newcommand{\eit}{\end{itemize}}
\newcommand{\ben}{\begin{enumerate}}
\newcommand{\een}{\end{enumerate}}
\newcommand{\bds}{\begin{description}}
\newcommand{\eds}{\end{description}}

\newcounter{romc}

\newcounter{alphc}

\newcommand{\blr}{\begin{list}{~(\roman{romc})~} {\usecounter{romc}
        \setlength{\topsep}{0pt} \setlength{\itemsep}{0pt}}}
\newcommand{\elr}{\end{list}}
\newcommand{\bla}{\begin{list}{~(\alph{alphc})~} {\usecounter{alphc}
        \setlength{\topsep}{0pt} \setlength{\itemsep}{0pt}}}
\newcommand{\ela}{\end{list}}

\begin{document}

\begin{frontmatter}

\title{Topological Representation of Double Boolean Algebras\tnoteref{mytitlenote}}
\tnotetext[mytitlenote]{Declarations of interest: none}
%\tnotetext[mytitlenote]{}
 \author[label1]{Prosenjit Howlader\corref{mycorrespondingauthor}}
 \cortext[mycorrespondingauthor]{Corresponding author}
\ead{prosenjithowlader@gmail.com}
 \author[label1]{Mohua Banerjee}
  \ead{mohua@iitk.ac.in}
 \address[label1]{Department of Mathematics and Statistics, Indian Institute of Technology  Kanpur,
Kanpur 208016, Uttar Pradesh, India }

\begin{abstract}
In formal concept analysis, the collection of  protoconcepts of any context forms a double Boolean algebra (dBa) which is fully contextual. Semiconcepts of a context form a  pure dBa. The present article is a study on topological representation results for dBas, and in particular, those for fully contextual and  pure dBas. The representation is  in terms of object oriented protoconcepts and semiconcepts of a context. 
%, notions that were defined by the authors earlier for introducing negation in the study of object oriented concepts.  In this work,  a 
A context on topological spaces (CTS) is considered,  and the focus is on   a special kind of CTS in which  the relation defining  the context as well as the converse of the relation are  continuous with respect to the topologies. Such CTS are denoted as ``CTSCR''. %The CTSCR as defined here is different from the notion of the same name used by Hartung. 
It is observed that  clopen object oriented protoconcepts of a CTSCR form a fully contextual dBa, while clopen object oriented semiconcepts form a pure dBa. Every dBa is shown to be quasi-embeddable into the dBa of clopen object oriented protoconcepts of a particular CTSCR.  
The quasi-embedding turns into an embedding in case of a 
%Using the quasi-embedding, it is  proved that every 
contextual dBa, and into an isomorphism, when the dBa is fully contextual. 
%and into an isomorphism for a pure dBa. 
% is embeddable into the algebra of clopen object oriented protoconcepts of some CTSCR. 
For    pure dBas, one obtains an isomorphism with the algebra of clopen object oriented semiconcepts of the CTSCR. Representation of finite dBas and Boolean algebras is also addressed in the process. Abstraction of properties of this CTSCR leads to the definition of   ``Stone contexts''.  Stone contexts and CTSCR-homeomorphisms are seen to form a category, denoted as  $\textbf{Scxt}$. 
Furthermore, correspondences are observed between dBa isomorphisms and CTSCR-homeomorphisms.
This motivates a  study of categorical duality of dBas, constituting the second part of the article.
Pure dBas  and fully contextual dBas  along with dBa isomorphisms form categories, denoted as $\textbf{PDBA}$  and  $\textbf{FCDBA}$ respectively. It is established that  $\textbf{PDBA}$ is equivalent to $\textbf{FCDBA}$, while $\textbf{FCDBA}$ and $\textbf{PDBA}$ are dually equivalent to $\textbf{Scxt}$.
%Correspondences are observed between pure dBa isomorphism and special  CTSCR-embeddings that are termed as ``SC-embeddings". 

%   It is shown that  $\textbf{PDBA}$ and $\textbf{Scxt}$ are related through  a functor that is contravariant, essential surjective and an embedding. 
%Furthermore, it is shown that the subcategory of $\textbf{PDBA}$  formed by  pure dBas and  dBa isomorphisms and that of $\textbf{Scxt}$ formed by Stone contexts and CTSCR-homeomorphisms, are  
%dually equivalent.

\end{abstract}


\begin{keyword}
Formal concept analysis, protoconcept, semiconcept, object oriented concept, double Boolean algebra, 
%Boolean algebra, 
continuous relation, 
%CTSCR, 
categorical duality. 
\MSC[2010] 06B15 \sep 06D50 \sep 06E15\sep  06E75
\end{keyword}

\end{frontmatter}


\section{Introduction}
\label{sec:Introduction}
In the field of lattice theory, topological representation and categorical duality  have been a subject of study for many years.  Well-known instances of  topological representation and duality results  are those for Boolean algebras by Stone \cite{StoneB,stone1937applications}, and  for distributive lattices  by Stone \cite{stone1938topological} and Priestley  \cite{priestley1970representation}.    Duality  for general lattices was investigated by Hartonas and  Dunn \cite{hartonas1997stone}, while that for bounded lattices was studied by  Urquhart \cite{urquhart1978topological} and Hartung \cite{Hartung1992}. \cite{Hartung1992} is of particular relevance to the present work, as {\it formal concept analysis} (FCA) \cite{ganter2012formal} was made use of to obtain the   results. This article addresses topological representation  results for a class of algebraic structures that arise in FCA, namely the {\it double Boolean algebras}. Furthermore, categorical duality is investigated for the classes of {\it fully contextual} and {\it pure} double Boolean algebras.

FCA was introduced by Wille \cite{wille1982restructuring} and has given rise to a rich body of work, both in theory and applications. In particular, a lot of theoretical development has been made in the direction of algebraic and category-theoretic studies related to FCA  (see e.g. \cite{wille,hitzler2004cartesian,hitzler2006categorical,zhang2006approximable,kwuida2007prime,GUO201929,howlader3}). 
The central objects in FCA are {\it contexts} and {\it concepts}. A {\it context} is a triple $(G,M,R)$, where $G,M$ are sets of {\it objects} and {\it properties} respectively, and $R$ is a relation between them with $gRm$ signifying that object $g$ has the property $m$. A {\it concept} of a context is a pair $(A,B), A \subseteq G, B \subseteq M$, such that $B$ contains exactly the properties that all objects of $A$ have, while $A$ contains exactly the objects having all the properties in $B$. The set of all concepts of a context forms a complete lattice, called the {\it concept lattice} of the context. Moreover, any complete lattice is isomorphic to a concept lattice of some context. Wille initiated the study of the  negation of a  concept  \cite{wille1989knowledge,wille}, to make formal concept analysis more useful for representation, processing, and acquisition of conceptual knowledge. 
If set-complement is used to define the negation of a concept, one encounters the  problem of closure. So the notion of a concept was generalized to that of a {\it semiconcept} and a {\it protoconcept} \cite{wille}. The set of all protoconcepts fails to form a lattice. However, it  leads to the algebraic structure of a double Boolean algebra (dBa)  \cite{wille}. In particular, protoconcepts  form fully contextual dBas \cite{vormbrock2005semiconcept}. In comparison, the set of all semiconcepts forms a subalgebra of the algebra of protoconcepts.  Abstraction of properties of this subalgebra yields a  pure dBa \cite{wille}. In this work, we obtain topological representation results for dBas in general, as well as for fully contextual and pure dBas.

A concept lattice is also dually isomorphic to the lattice formed by the set of all {\it object oriented concepts}, the latter defined in the context of {\it rough set theory}. 
Proposed by Pawlak  \cite{pawlak2012rough}, rough set theory  is a well-established mathematical tool  to handle incompleteness in data.  The  theory hinges on the notions of {\it approximation space} and {\it lower} and {\it upper approximation operators} defined on the space. The Pawlakian  definition of approximation space has been generalized; a {\it generalized approximation space} \cite{yao1996generalization} consists of a set $W$ and a binary relation $R$ on $W$. On such a space, different definitions of lower and upper approximation operators are found in literature. In the ones used here, the lower approximation of any subset $B$ of $W$ collects every  element of $W$ such that all elements  $R$-related to it lie within $B$, while the  upper approximation of $B$ contains all elements of $W$  that are $R$-related to at least one element of $B$. Many comparative studies between rough set theory and FCA have been made, e.g. in \cite{RCA,CACLkeyun,saquer2001concept,duntsch2002modal,yao2004comparative,yao2004concept,RSAFCAyao,aclmc,FCAARDT,howlader2018algebras,howlader2020}. The   approximation operators in a generalized approximation space have been imported into FCA and named as  {\it necessity} and {\it possibility operators} \cite{duntsch2002modal}. Using these  operators,  D\"{u}ntsch and Gediga \cite{duntsch2002modal} defined {\it property oriented concepts}, and Yao  \cite{yao2004concept} defined {\it object oriented concepts}. 
Besides property oriented and object oriented concepts, various other types of concepts and related notions have been defined in rough set theory and  studied from the algebraic  and categorical points of view (see e.g. \cite{lei2009rough,yang2009rough,GUO2014885,howlader2018algebras,howlader2020}).  
The present authors introduced negation in the study of  object oriented concepts of a context, in the lines of Wille's study on negation of concepts. Notions of {\it object oriented semiconcepts} and {\it object oriented protoconcepts} of a context were  defined in \cite{howlader2018algebras,howlader2020}. It was shown that the algebra of protoconcepts  is  isomorphic to that of object oriented protoconcepts \cite{howlader2020}, and the algebra of  semiconcepts is dually isomorphic to that of object oriented semiconcepts \cite{howlader2018algebras}. 
The representation results proved in this work involve object oriented protoconcepts and object oriented semiconcepts of certain special contexts.

Wille \cite{wille} constructed a {\it standard context} for each dBa $\textbf{D}$, denoted as $\mathbb{K}(\textbf{D})$. The context consists of the sets of all primary filters and ideals of $\textbf{D}$, and a relation $\Delta$  such that $F\Delta I$ if and only if $F\cap I\neq\emptyset$ for any primary filter $F$ and primary ideal $I$. It has been proved that every dBa $\textbf{D}$ is  quasi-embedded into the algebra of protoconcepts of $\mathbb{K}(\textbf{D})$. The quasi-embedding becomes an embedding in case of   pure dBas \cite{BALBIANI2012260}. The special case of the representation result for finite dBas was also given by Wille in \cite{wille}.
%Vormbrock \cite{vormbrock2005semiconcept}  characterized the class of protoconcepts algebra as particular subclass of the class of complete fully contextual dBas, where as he proved that the class of semiconcepts algebras is equivalent to a particular subclass of the class of complete dBas. However, all dBas are clearly not complete. 
In this work, we equip   the sets of all primary filters and ideals of $\textbf{D}$ with certain topologies. The resulting structure, denoted as $\mathbb{K}_{pr}^{T}(\textbf{D})$, is an instance of a {\it context on topological spaces} (CTS), that consists of a pair of  topological spaces and a relation between the domains. In fact, $\mathbb{K}_{pr}^{T}(\textbf{D})$  is proved to be an instance of a special kind of 
   CTS, denoted as  {\it CTSCR},   in which   the relation and its converse are both {\it continuous}.  

Continuity (or hemicontinuity) of a relation was introduced by  Berge \cite{berge1997topological}.
We work with a more general definition of continuity that is  considered  in \cite{guide2006infinite}. 
A relation  $R$ is said to be continuous  when the {\it upper} and {\it lower inverses} of any open set under $R$ are both open; we  observe that the upper and lower inverses of a set are just the images of the set under the  necessity and possibility operators  (respectively) in FCA. Our approach for proving  the  representation  results presented in this paper is based on the approach adopted by  Hartung \cite{Hartung1992} to deal with bounded lattices, and the above observation.

That $\mathbb{K}_{pr}^{T}(\textbf{D})$  is a CTSCR, is proved by using the prime ideal theorem \cite{kwuida2007prime,howlader3} for dBas. We consider  {\it clopen} object oriented semiconcepts and  protoconcepts of a CTS, namely where the component sets of the object oriented semiconcepts and  protoconcepts are both clopen (that is, closed and open)  in the respective topologies. It is established  that the set of clopen object oriented protoconcepts (semiconcepts) of a CTSCR forms a fully contextual  dBa (pure dBa). The following representation results are then obtained. Any dBa $\textbf{D}$ is quasi-embeddable into the algebra of clopen object oriented protoconcepts of $\mathbb{K}_{pr}^{T}(\textbf{D})$. In case $\textbf{D}$ is contextual, the quasi-embedding is an embedding. If $\textbf{D}$ is fully contextual, the quasi-embedding turns into an isomorphism.  On the other hand, it is shown that the largest pure subalgebra $\textbf{D}_{p}$ of any dBa $\textbf{D}$ is isomorphic to  the algebra of clopen object oriented semiconcepts of $\mathbb{K}_{pr}^{T}(\textbf{D})$. This results in an isomorphism theorem for pure dBas, since $\textbf{D}_{p}=\textbf{D}$ if $\textbf{D}$  is  pure.  The representation theorems for fully contextual and pure dBas yield  an isomorphism theorem for Boolean algebras as well.   It is observed that a representation result for finite dBas can be  obtained in terms of object oriented protoconcepts and semiconcepts; we show that it is also obtained as a special case from the above-mentioned representation result for dBas. 

The second part of the  paper focusses on categorical  duality  results for dBas. For a dBa $\textbf{D}$, the CTSCR $\mathbb{K}_{pr}^{T}(\textbf{D})$  is observed to have some special properties, which, on abstraction, lead to the definition of a {\it Stone context}. Stone contexts  and CTSCR-homeomorphisms form a category, denoted as $\textbf{Scxt}$.
%Correspondences are established between dBa isomorphisms and CTSCR-homeomorphisms.
%Furthermore, correspondences between surjective dBa homomorphisms and embeddings between CTSCRs are established.
%It is shown  that for each surjective dBa homomorphism $f$ from a  dBa $\textbf{D}$ to a dBa $\textbf{M}$, there is  a CTSCR-embedding 
%$(\alpha_{f},\beta_{f})$ from  $\mathbb{K}_{pr}^{T}(\textbf{M})$  to  $\mathbb{K}_{pr}^{T}(\textbf{D})$. 
%On abstraction of properties of the  CTSCR-embedding $(\alpha_{f},\beta_{f})$, where $\textbf{D}$ and $\textbf{M}$ are pure dBas, the notion of an {\it SC-embedding} is introduced. 
Fully contextual and pure dBas along with dBa isomorphisms also form  categories, denoted as $\textbf{FCDBA}$ and $\textbf{PDBA}$ respectively. It is shown that $\textbf{FCDBA}$  is equivalent to $\textbf{PDBA}$, whereas  $\textbf{PDBA}$ is dually equivalent to $\textbf{Scxt}$. As a consequence, one obtains a dual equivalence  between $\textbf{FCDBA}$ and $\textbf{Scxt}$.
%It is shown that there is a  functor $F$  from \textbf{PDBA} to \textbf{Scxt} that is contravariant, essential surjective and an embedding. Finally, it is observed that the restriction  of $F$ to  \textbf{IPDBA},  a subcategory of the category \textbf{PDBA} comprising pure dBas and dBa isomorphisms, is a dual equivalence into the subcategory 
 
%\textbf{HScxt} of \textbf{Scxt}  that comprises Stone contexts and CTSCR-homeomorphisms. 
The paper has been arranged as follows. Section \ref{dba} gives the preliminaries that are required in this work.  
 Some  algebraic properties of dBas have been obtained  in Section \ref{AIdBa} -- these are  used later  in Sections \ref{RT}  and  \ref{CE}. In Section \ref{TC}, CTS, clopen object oriented semiconcepts and protoconcepts, and CTSCR are defined and studied in relation to dBas. The representation theorems for dBas are proved in Section \ref{RT}. Section \ref{CE} presents the study on categorical  duality of fully contextual and pure dBas. Section  \ref{conclusion} concludes the  work.
 
 In our presentation, the symbols $\forall$,  $\Rightarrow$, $\Leftrightarrow$, {\it and}, {\it or} and $not$ will be used with the usual meanings in the metalanguage.

\section{Preliminaries}
\label{dba}

In the following subsections, we present preliminaries related to dBas, object oriented concepts, semiconcepts and protoconcepts. % that we require in our work.
Our primary references are \cite{ganter2012formal,wille,duntsch2002modal,yao2004concept,yao2004comparative,howlader2018algebras,howlader2020,howlader3}.
%The section is divided  into five subsections.
 
\subsection{\rm{\textbf{Concept, semiconcept and protoconcept of a context}}}

Let us recall the definitions and some properties of contexts, concepts, semiconcepts and protoconcepts. The following are taken from  \cite{ganter2012formal}.

\begin{definition}
{\rm A {\it context} is a triple $\mathbb{K}:=(G, M, R)$, where $G$ is a set of {\it objects}, $M$ a set of {\it properties}, and $R\subseteq G\times M$. \\
The {\it complement}  of a {\rm context} $\mathbb{K}:=(G,M,R)$ is the context $\mathbb{K}^{c}:=(G,M,-R)$, where $-R:=(G\times M) \setminus R$.\\
For any $A\subseteq G, B\subseteq M$, consider the sets \\
$A^{\prime}:=\{m\in M:\forall g\in G(g\in A\implies gRm)\}$ and
$ B^{\prime}:=\{g\in G:\forall m\in M(m\in B\implies gRm)\}$.\\
$(A,B)$ is a {\it concept} of $\mathbb{K}$, when $A^{\prime}=B$ and $B^{\prime}=A$. 
The set of all concepts is denoted by $\mathfrak{B}(\mathbb{K})$.\\
A partial order relation $\leq$ is given on $\mathfrak{B}(\mathbb{K})$ as follows. For concepts $(A_{1},B_{1})$ and  $(A_{2},B_{2})$, \\$(A_{1},B_{1})\leq (A_{2},B_{2})$ if and only if $A_{1}\subseteq A_{2}$
(equivalently $B_{2}\subseteq B_{1})$.}
\end{definition} 

\begin{definition}
{\rm Let $\mathbb{K}:=(G,M, R)$ be a context and  $H\subseteq G$, $N\subseteq M$. $\mathbb{S}:=(H,N,R\cap(H\times N))$ is called  a {\it subcontext} of $\mathbb{K}$.}
\end{definition}

\begin{definition}
{\rm Let $\mathbb{K}_{1}:=(G_1,M_1,R_1)$ and $\mathbb{K}_{2}:=(G_2,M_2,R_2)$ be two contexts. A {\it context homomorphism}  $f:\mathbb{K}_{1} \rightarrow \mathbb{K}_{2}$ is a pair of maps $(\alpha,\beta)$, where $\alpha:G_1 \ra G_2$, $\beta: M_{1} \ra M_{2}$ are such that $gR_1 m\iff\alpha(g)R_{2}\beta(m),~\mbox{for all}~ g\in G_{1}~\mbox{and}~ m\in M_{1}.$\\
If $\alpha$ and $\beta$  are injective then $f$ is called a {\it context embedding}.   $f$ is called a {\it context isomorphism} if $\alpha$, $\beta$ are bijective -- in that case one says that $\mathbb{K}_{1}$ is {\it isomorphic} to $\mathbb{K}_{2}$.}
\end{definition}
\noindent For the context $\mathbb{K}:=(G,M,R)$ and the identity maps $id_{G},id_{M}$ on $G,M$ respectively, $id_{\mathbb{K}}:\mathbb{K} \rightarrow \mathbb{K}$ denotes the context isomorphism $(id_{G},id_{M})$. \\
Composition of context homomorphisms is  defined component wise:

\begin{center}
\begin{tikzcd}[row sep =large, column sep = large]
\mathbb{K}^{T}_{1}\arrow[r,"f_{1}:=(\alpha\CommaBin\beta)"] \arrow[dr,dashrightarrow,"f_{2}\circ f_{1}=(\delta\circ\alpha\CommaBin\gamma\circ\beta)"']& \mathbb{K}^{T}_{2}\arrow[d,"f_{2}:=(\delta\CommaBin\gamma)"]\\
& \mathbb{K}^{T}_{3}
\end{tikzcd}
\end{center}

\begin{proposition}
\label{contx inv iso}
{\rm For a context isomorphism $f:=(\alpha,\beta)$ from $\mathbb{K}_{1}$ to $\mathbb{K}_{2}$,   $g:=(\alpha^{-1},\beta^{-1})$ is a context isomorphism from $\mathbb{K}_{2}$ to $\mathbb{K}_{1}$ such that $f\circ g=id_{\mathbb{K}_{2}}$ and $g\circ f=id_{\mathbb{K}_{1}}$.}
\end{proposition}
\noindent $g$ is called the {\it inverse of} $f$.

\begin{notation} {\rm 
For a relation $R\subseteq G\times M$,  $R^{-1}$ denotes the converse of $R$, that is $R^{-1}\subseteq M\times G$ and $yR^{-1}x$ if and only if $xRy$. \\ 
$R(x):=\{y\in M~:~xRy\}$,  and $R^{-1}(y):=\{x\in G~:~xRy\}$, for any $x\in G$, $y\in M$. \\
$X^c$ denotes the complement of a subset $X$ of $G$ (or $M$) and $\mathcal{P}(X)$, the power set of any set $X$.\\
$\mathcal{B}(\mathbb{K})$ denotes the set of all concepts of the context $\mathbb{K}$. For a concept $(A,B)$, $A:=ext((A,B))$ is its {\it  extent},  its {\it intent}, $int((A,B))$, is $B$.
}\end{notation}

While attempting to define the negation of a formal concept, it was noticed that  the set-complement   could not be used as it resulted in a problem of closure. 
%as it was noticed that there is  if is used to define it. 
So generalizations of the notion of a concept, namely  {\it semiconcepts}  and {\it protoconcepts} were introduced \cite{wille}.
\begin{definition}
{\rm For  a context  $\mathbb{K}:=(G,M,R)$  and  $A\subseteq G, B\subseteq M$, the pair $(A,B)$ is called a {\it semiconcept}  of $\mathbb{K}$ if and only if $A^{\prime}=B$ or $B^{\prime}=A$. $(A,B)$ is a {\it protoconcept} of $\mathbb{K}$ if and only if $A^{\prime\prime}=B^{\prime}~ (\mbox{equivalently}~ A^{\prime}=B^{\prime\prime})$.
}\end{definition}
\begin{notation} {\rm $\mathfrak{H}(\mathbb{K})$ denotes the set of all semiconcepts   and the set of all protoconcepts  is denoted by $\mathfrak{P}(\mathbb{K})$.} \end{notation} 

\begin{observation} 
{\rm $\mathfrak{H}(\mathbb{K})\subseteq\mathfrak{P}(\mathbb{K})$.}
 \end{observation}

\noindent  On $\mathfrak{P}(\mathbb{K})$,  operations $\sqcap, \sqcup, \neg, \lrcorner, \top$ and $\bot$ are defined as follows. For any protoconcepts  $(A_{1},B_{1}),\\(A_{2},B_{2}), (A,B)$,\\$(A_{1},B_{1})\sqcap(A_{2},B_{2}) :=(A_{1}\cap A_{2}, (A_{1}\cap A_{2})^{\prime}),$
$(A_{1},B_{1})\sqcup(A_{2},B_{2}) :=((B_{1}\cap B_{2})^{\prime}, B_{1}\cap B_{2}),$\\
$\neg(A,B) :=( A^{c}, A^{c\prime}),~\lrcorner(A,B) :=( B^{c\prime}, B^{c}),~
\top :=(G,\emptyset)$ and $\bot :=(\emptyset,M)$.

\begin{notation} 
\label{semiproto} {\rm 
$\underline{\mathfrak{P}}(\mathbb{K})$ denotes the 
 {\it algebra of protoconcepts} of a context $\mathbb{K}$, which is the abstract algebra $(\mathfrak{P}(\mathbb{K}),\sqcup,\sqcap,\neg,\lrcorner,\top,\bot)$ of type $(2,2,1,1,0,0)$  formed  by $\mathfrak{P}(\mathbb{K})$. \\
$\underline{\mathfrak{H}}(\mathbb{K})$ denotes  the {\it algebra of semiconcepts} of  $\mathbb{K}$, which is the algebra $(\mathfrak{H}(\mathbb{K}),\sqcup,\sqcap,\neg,\lrcorner,\top,\bot)$. \vskip 2pt
\noindent $\underline{\mathfrak{H}}(\mathbb{K})$ is a subalgebra of $\underline{\mathfrak{P}}(\mathbb{K})$. } \end{notation}

\subsection{\rm{\textbf{Double Boolean algebras}}}

The structure of a dBa was proposed \cite{wille} as an abstraction of  the algebra of protoconcepts of a context. A pure dBa is the result of abstraction of the algebra of semiconcepts.

\begin{definition}
\label{DBA}
{\rm \cite{wille} An abstract algebra $  \textbf{D}:= (D,\sqcup, \sqcap, \neg,\lrcorner,\top,\bot)$ satisfying the following properties is called a  {\it double Boolean algebra} (dBa). For any $x,y,z \in D$,
\begin{multicols}{2}
	\item[(1a)] $ (x \sqcap x ) \sqcap  y = x \sqcap  y $
	\item[(2a)] $ x\sqcap y = y\sqcap  x $
	\item[(3a)] $ x \sqcap ( y \sqcap  z) = (x \sqcap  y) \sqcap  z $
	\item[(4a)] $\neg (x \sqcap  x) = \neg  x $
	\item[(5a)] $ x  \sqcap (x \sqcup y)=x \sqcap  x $
	\item [(6a)]$x \sqcap  (y \vee z ) = (x\sqcap  y)\vee (x \sqcap  z)$
	\item[(7a)] $ x \sqcap (x\vee y)= x \sqcap  x $
	\item [(8a)]$ \neg \neg (x \sqcap  y)= x \sqcap  y$
	\item[(9a)]$ x  \sqcap \neg  x= \bot$
	\item[(10a)] $\neg \bot = \top \sqcap   \top $
	\item[(11a)] $ \neg \top = \bot $
	
	\item[(1b)] $ (x \sqcup x)\sqcup  y = x \sqcup y $
	\item [(2b)]$ x \sqcup   y = y\sqcup   x $
	\item [(3b)]$ x \sqcup (y \sqcup  z) = (x \sqcup  y)\sqcup  z $
	\item [(4b)]$ \lrcorner(x \sqcup   x )= \lrcorner x $
	\item[(5b)] $ x \sqcup  (x \sqcap y) = x \sqcup   x$
	\item[(6b)]$ x \sqcup  (y \wedge z) = (x \sqcup  y) \wedge  (x \sqcup  z) $
	\item[(7b)] $ x\sqcup  (x \wedge  y) =x \sqcup   x $
	\item[(8b)] $ \lrcorner\lrcorner(x \sqcup  y) = x\sqcup  y $
	\item [(9b)]$ x \sqcup \lrcorner x = \top $
	\item [(10b)]$ \lrcorner\top =\bot \sqcup  \bot $
	\item [(11b)]$ \lrcorner\bot =\top $
	
\end{multicols}
\begin{itemize}
\item[(12)] $ (x \sqcap  x) \sqcup (x \sqcap x) = (x \sqcup x) \sqcap (x \sqcup x), $
\end{itemize}
where $ x\vee y := \neg(\neg x \sqcap\neg y)$ and 
$ x \wedge y :=\lrcorner(\lrcorner x \sqcup \lrcorner y)$. \\
A quasi-order relation $\sqsubseteq$ on $D$ is given as follows. For any $x,y \in D$,\\ $x \sqsubseteq y \iff x\sqcap y=x\sqcap x~\mbox{and}~x\sqcup y=y\sqcup y$.
 \vskip 2pt

}
\end{definition}
\begin{definition}{\rm \cite{wille,vormbrock2005semiconcept} \label{contextdefn}
Let $\textbf{D}$ be a dBa.
\begin{enumerate}[{(a)}]
 \item $\textbf{D}$ is  {\it contextual} if and only if  the  quasi-order $\sqsubseteq$ on $\textbf{D}$ is a partial order.
\item $\textbf{D}$ is  {\it fully contextual} if and only if it is contextual and, for each $y\in D_{\sqcap}$ and $x\in D_{\sqcup}$ with $y_{\sqcup}=x_{\sqcap}$, there is a unique $z\in D$ with $z_{\sqcap}=y$ and $z_{\sqcup}=x$.
\item $\textbf{D}$ is called {\it pure} if, for all $x\in D$, either $x\sqcap x=x$ or $x\sqcup x=x$.
\end{enumerate}
 Contextual dBas are also referred to as {\it regular} dBas in literature \cite{breckner2019topological}.}
\end{definition} 
\begin{theorem}
{\rm \cite{wille,vormbrock2005semiconcept} $\underline{\mathfrak{P}}(\mathbb{K})$ is a fully contextual dBa, while 
$\underline{\mathfrak{H}}(\mathbb{K})$ is a  pure dBa.
} 
\end{theorem}

\vskip 3pt \noindent In the following, let $\textbf{D}:= (D,\sqcup,\sqcap,\neg,\lrcorner,\top,\bot)$ be a dBa.
\begin{notation} {\rm 
%For any dBa $\textbf{ D}:= (D,\sqcup,\sqcap,\neg,\lrcorner,\top,\bot)$, 
$D_{\sqcap}:=\{x\in D~:~x\sqcap x=x\}$,  $D_{\sqcup}:=\{x\in D~:~x\sqcup x=x\}$, $D_{p}:=D_{\sqcap}\cup D_{\sqcup}$.\\
For $x\in D$, $x_{\sqcap}:=x\sqcap x$ and $x_{\sqcup}:=x\sqcup x$. 
 } \end{notation}
 %Let $\textbf{ D}:= (D,\sqcup,\sqcap,\neg,\lrcorner,\top,\bot)$ be a dBa. Then we have 

 \begin{proposition} 
\label{pro1}
{\rm \cite{vormbrock}
%Let $\textbf{D}$ be a dBa.
\begin{enumerate}[{(i)}]
\item $\textbf{D}_{\sqcap}:=(D_{\sqcap},\sqcap,\vee,\neg,\bot,\neg\bot)$ is a Boolean algebra. The partial  order relation in $\textbf{D}_{\sqcap}$ is the restriction of the quasi-order $\sqsubseteq$ on $D$ to $D_{\sqcap}$ and is denoted by $\sqsubseteq_{\sqcap}$.
\item  $\textbf{D}_{\sqcup}:=(D_{\sqcup},\sqcup,\wedge,\lrcorner,\top,\lrcorner\top)$ is a  Boolean algebra. For $\textbf{D}_{\sqcup}$, the partial  order relation   is the restriction of $\sqsubseteq$ to $D_{\sqcup}$ and  is denoted by $\sqsubseteq_{\sqcup}$.
\item For any $x,y\in D$, $x\sqsubseteq y$ if and only if $x\sqcap x\sqsubseteq y\sqcap y$ and $x\sqcup x\sqsubseteq y\sqcup y$, that is, $x_{\sqcap}\sqsubseteq_{\sqcap} y_{\sqcap}$ and $x_{\sqcup}\sqsubseteq_{\sqcup} y_{\sqcup}$.
\end{enumerate}}

\end{proposition}

\begin{proposition}
\label{order pure}
{\rm \cite{BALBIANI2012260}   If $\textbf{D}$ is a pure dBa, the quasi-order $\sqsubseteq$ on $D$  becomes a partial order, that is, every pure dBa is contextual.
}
\end{proposition}

\noindent We state below some results giving further properties of dBas. 
In \cite{kwuida2007prime}, Kwuida obtained the following.
\begin{proposition}
\label{pro1.5} {\rm \cite{kwuida2007prime} 
%Let $\textbf{D}$ be a dBa. Then 
Let $x,y,a\in D$.
 \begin{enumerate}[{(i)}]
 \item $x\sqcap\bot=\bot$ and $x\sqcup\bot=x\sqcup x$, that is $\bot\sqsubseteq x$.
\item $x\sqcup \top=\top$ and $x\sqcap \top=x\sqcap x$, that is $x\sqsubseteq \top$.
\item $x=y$ implies that $x\sqsubseteq y$ and $y\sqsubseteq x$.
\item $x\sqsubseteq y$ and $y \sqsubseteq x$ if and only if $x\sqcap x=y\sqcap y$ and $x\sqcup x = y\sqcup y$.
\item $x\sqcap y\sqsubseteq x$ and $x\sqcap y\sqsubseteq y.$ Dually $y\sqsubseteq x\sqcup y$ and $x\sqsubseteq x\sqcup y$.
\item $x\sqsubseteq y$ implies $x\sqcap a\sqsubseteq y\sqcap a$ and $x\sqcup a\sqsubseteq y\sqcup a$. 
\end{enumerate}}
\end{proposition}

\noindent In  \cite{howlader3}, one observed the following.
\begin{proposition}
	\label{pro2}
	{\rm 
	%Let $\textbf{D}$ be a dBa. Then 
	Let $ x,y\in D$.
	\begin{enumerate}[{(i)}]
	\item $\neg x=(\neg x)_{\sqcap}\in D_{\sqcap}$ and $\lrcorner x=(\lrcorner x)_{\sqcup}\in D_{\sqcup}$.
\item $x\sqsubseteq y$  if and only if $ \neg y\sqsubseteq \neg x $ and $ \lrcorner y\sqsubseteq \lrcorner x $.
	\item $\neg\neg x=x\sqcap x$ and $\lrcorner\lrcorner x=x\sqcup x$.
	\item $x\sqcap x$, $x\vee y\in D_{\sqcap}$ and  $x\sqcup x$, $x\wedge y\in D_{\sqcup}$.
	\item $\neg(x\vee y)=\neg x\sqcap\neg y$ and $\neg(x\sqcap y)=\neg x\vee\neg y$.
	\item $\lrcorner(x\wedge y)=\lrcorner x\sqcup\lrcorner y$ and $\lrcorner(x\sqcup y)=\lrcorner x\wedge \lrcorner y$.
     \item $  x\sqsubseteq\lrcorner y$ if and only if $ y\sqsubseteq \lrcorner x $.
	\item $ \neg x\sqsubseteq y$ if and only if  $\neg y\sqsubseteq x $.
	\item $\neg\neg\neg x=\neg x$.
    \end{enumerate}}
\end{proposition}
Using the properties of dBas given above, one can show
 \begin{proposition}
 \label{puresub}
{\rm  \cite{wille} $\textbf{D}_{p}:=(D_{p},\sqcup,\sqcap,\neg,\lrcorner,\top,\bot)$ is the largest pure subalgebra of $\textbf{D}$. Moreover, if $\textbf{D}$ is pure, $\textbf{D}_{p}=\textbf{D}$. }
 \end{proposition}

\noindent For the set $\mathfrak{P}(\mathbb{K})$ of all protoconcepts, Kwuida \cite{kwuida2007prime}  showed that\\ 
  \hspace*{1.5cm}  $(x\sqcup y)\sqcap (x\sqcup \lrcorner y)  \leq x\sqcup x$ and $x\sqcap x \leq (x\sqcap y)\sqcup (x\sqcap\neg y)$, for any $x,y\in \mathfrak{P}(\mathbb{K})$. \\He then redefined Wille's dBa by adding these two properties  as  axioms to  Definition \ref{DBA}.
% $x\sqcap x \sqsubseteq (x\sqcap y)\sqcup (x\sqcap\neg y),$ 
%and  $(x\sqcup y)\sqcap (x\sqcup \lrcorner y)  \sqsubseteq x\sqcup x$. 
These new algebras were also termed as `dBas'   by Kwuida. In \cite{howlader3}, it has been  established that Kwuida's class of dBas is equivalent to that defined by Wille. The result (Corollary \ref{cor1} below)  is obtained by proving   a number of intermediate properties of a dBa  that are given in the following theorem.  
%More precisely, we show that (a) and (b) are derivable from  Definition \ref{DBA} using properties given in Theorem \ref{algebraic equation} .
\begin{theorem}
\label{algebraic equation}
{\rm \cite{howlader3} 
%Let $\textbf{D}$ be a dBa. 
For all $x,y\in D$,
\begin{enumerate}[{(i)}]
\begin{multicols}{2}
\item $x\sqcap\neg (x\sqcup y)=\bot$.
\item $\neg(x\sqcup y)=\neg(x\sqcup y)\sqcap \neg x$.
\item $x\sqcap y=x\sqcap\neg (x\sqcap \neg y)$.
\item $x\sqcup(y\sqcap \neg x)=x\sqcup (y\sqcap y)$.
\item $(x\sqcap y)\sqcup (x\sqcap \neg y)=(x\sqcap x)\sqcup (x\sqcap x)$.
\item $x\sqcup\lrcorner (x\sqcap y)=\top$.
\item $\lrcorner(x\sqcap y)=\lrcorner(x\sqcap y)\sqcup \lrcorner x$.
\item $x\sqcup y=x\sqcup\lrcorner (x\sqcup \lrcorner y)$.
\item $x\sqcap(y\sqcup \lrcorner x)=x\sqcap (y\sqcup y)$.
\item $(x\sqcup y)\sqcap (x\sqcup \lrcorner y)=(x\sqcup x)\sqcap (x\sqcup x)$.
\end{multicols}
\end{enumerate}}
\end{theorem}

\begin{corollary}
\label{cor1}
{\rm \cite{howlader3}
%Let $\textbf{D}$ be a dBa. Then 
For all $x,y\in D$, the following hold.
\begin{enumerate}[{(i)}]
\item $ (x\sqcup y)\sqcap (x\sqcup \lrcorner y)\sqsubseteq x\sqcup x.$
\item $ x\sqcap x\sqsubseteq (x\sqcap y)\sqcup (x\sqcap \neg y)$.
\end{enumerate}}
\end{corollary}

We shall also require the notion of a {\it dBa homomorphism}.
\begin{definition}\label{DBAHOM}
{\rm Let $\textbf{D}:=(D,\sqcup,\sqcap,\neg,\lrcorner,\top_{\textbf{D}},\bot_{\textbf{D}})$ and $\textbf{M}:=(M,\sqcup,\sqcap,\neg,\lrcorner,\top_{\textbf{M}},\bot_{\textbf{M}})$ be two dBas. A map  $h:D\rightarrow M$ is called a {\it dBa  homomorphism} from $\textbf{D}$ to $\textbf{M}$, if  the following hold for all $a,b\in D$:  $h(a\sqcap b)=h(a)\sqcap h(b)$, $h(a\sqcup b)=h(a)\sqcup h(b)$, $h(\neg a)=\neg h(a)$, $h(\lrcorner a)=\lrcorner h(a)$, and \\$h(\top_{\textbf{D}})=\top_{\textbf{M}}$, $h(\bot_{\textbf{D}})=\bot_{\textbf{M}}$.\\
The dBa homomorphism $h$ is called {\it quasi-injective}, when $x\sqsubseteq y$ if and only if $h(x)\sqsubseteq h(y)$, for all $x,y\in D$.  $\textbf{D}$ is then said to be {\it quasi-embedded into} $\textbf{M}$. \\ If $h$ is an injective dBa homomorphism,  $\textbf{D}$ is said to be {\it embedded into} $\textbf{M}$.\\ 
A   dBa homomorphism from $\textbf{D}$ to $\textbf{M}$  that is quasi-injective and surjective  is called  a {\it dBa  quasi-isomorphism}, and  $\textbf{D}$ is said to be {\it quasi-isomorphic} to $\textbf{M}$.\\
If there is a dBa homomorphism from $\textbf{D}$ to $\textbf{M}$ that is bijective then the map is called a {\it  dBa isomorphism}, and the dBas $\textbf{D},\textbf{M}$ are said to be {\it isomorphic} to each other.
}
\end{definition}
\noindent Note that composition of two dBa homomorphisms is also a dBa homomorphism.

%\begin{definition}**
%{\rm Let $\textbf{D}$ and $\textbf{M}$ be two dBas. We say $\textbf{D}$ is quasi-embedded into $\textbf{M}$ if there is a quasi-injective dBa homomorphism $h$ from $\textbf{D}$ into $\textbf{M}$ and $h$ is called a quasi dBa embedding. If $h$ is an injective dBa homomorphism, we say $\textbf{D}$ is embedded into $\textbf{M}$ and in this cases $h$ is called a dBa embedding }
%\end{definition}

\subsection{\rm{\textbf{Filters, ideals and the prime ideal theorem for dBas}}}
\label{filterideal}
 Kwuida proved the prime ideal theorem for his class of dBas by generalizing the notion of  a prime filter (ideal) of Boolean algebras to define  a {\it primary} filter (ideal) for dBas \cite{kwuida2007prime}. Due to Corollary \ref{cor1}, these notions and the prime ideal  theorem hold for Wille's dBas. Let us give the  basic definitions and results in this regard. Again, let  $\textbf{D}:=(D,\sqcup,\sqcap,\neg,\lrcorner,\top,\bot)$ be a dBa. 
\begin{definition}
{\rm A   {\it filter} in  $\textbf{D}$ is a subset $F$ of $D$ such that $x\sqcap y\in F$ for all $x,y \in F$, and for all $z\in D$ and $ x \in F, x\sqsubseteq z$ implies that $z\in F$. An {\it ideal} in a dBa is defined dually.\\
The {\it filter (ideal) generated by a set} $X$, denoted as $F(X)(I(X))$, is the smallest filter (ideal) in \textbf{D} containing $X$.\\
A  {\it base} for the filter $F$ is a subset $F_{0}(\subseteq D)$ such that $F=\{y\in D : x\sqsubseteq y ~\mbox{for some}~ x\in F_{0}\}$. A base for an ideal is dually defined.\\
A filter $F$ (ideal $I$) is  {\it proper} if and only if  $F\neq D$ ($I\neq D$). \\
 A {\it primary} filter  $F$ (ideal $I$) is a proper filter (ideal)  such that $x\in F~ \mbox{or} ~\neg x \in F ~(x\in I ~\mbox{or}~ \lrcorner x\in I)$, for all $x\in D$.
 %\vskip 2pt \noindent
}
\end{definition} 
\noindent It can be shown that $F(X) =\{x\in D~:~ \sqcap_{i=1}^{n}a_{i}\sqsubseteq x~\mbox{for some}~ a_{i}\in X~\mbox{and}~i=1,\ldots,n \}$, and dually,  $I(X) =\{x\in D~:~ x \sqsubseteq\sqcup_{i=1}^{n}a_{i} ~\mbox{for some}~ a_{i}\in X ~\mbox{and}~i=1,\ldots,n\}$.

 \begin{theorem}[\textbf{Prime ideal theorem for dBas}]
\label{PITDB}
{\rm Let $F$ be a filter and $I$  an ideal  in $\textbf{D}$  such that $F\cap I=\emptyset$. Then there exists a primary filter $G$ and a primary ideal $J$ of $\textbf{D}$ satisfying $F\subseteq G,~I\subseteq J$ and $G\cap J=\emptyset$.}
\end{theorem}

\begin{notation}
\label{Fx and Ix} {\rm 
$\mathcal{F}_{pr}(\textbf{D})$ denotes the set of all  primary filters and  $\mathcal{I}_{pr}(\textbf{D})$, the set of all primary ideals.\\
%Let $\textbf{D}:=(D,\sqcup,\sqcap,\neg,\lrcorner,\top,\bot)$ be a dBa.\\
$\mathcal{F}_{p}(\textbf{D}):=\{F \subseteq D : F ~\mbox{is a  filter of}~\textbf{D}~\mbox{and}~F\cap D_{\sqcap} ~\mbox{is a prime filter in}~ \textbf{D}_{\sqcap}\}$, while \\
$\mathcal{I}_{p}(\textbf{D}):=\{I \subseteq D : I ~\mbox{is an ideal of}~\textbf{D}~\mbox{and}~I\cap D_{\sqcup}~\mbox{ is a prime ideal in}~ \textbf{D}_{\sqcup}\}$. \\
$F_{x}:=\{F\in\mathcal{F}_{p}(\textbf{D})~:~x\in F\}$  and $I_{x}:=\{I\in\mathcal{I}_{p}(\textbf{D})~:~x\in I\}$, for any $x \in D$.\\
$\mathbb{K}(\textbf{D}):=(\mathcal{F}_{p}(\textbf{D}),\mathcal{I}_{p}(\textbf{D}),\Delta)$ is defined as the {\it standard context}, where  $F\Delta I$ if and only if $F\cap I\neq\emptyset$, for any $F\in \mathcal{F}_{p}(\textbf{D}),~I\in \mathcal{I}_{p}(\textbf{D})$. 
}
\end{notation}
\begin{lemma}{\rm \cite{wille}
\label{derivation} 
%Let $\textbf{D}$ be a dBa and $\mathbb{K}(\textbf{D}):=(\mathcal{F}_{p}(\textbf{D}),\mathcal{I}_{p}(\textbf{D}),\Delta)$ then 
Let $x\in \textbf{D}$. 
\begin{enumerate}[{(i)}]
 \item  $F_{x}^{\prime}=I_{x_{\sqcap\sqcup}}$ and $I_{x}^{\prime}=F_{x_{\sqcup\sqcap}}$.
\item  $(F_{x})^{c}=F_{\neg x}$ and $(I_{x})^{c}=I_{\lrcorner x}$.
\item $I_{x}\cap I_{y}=I_{x\sqcup y}$ and $I_{x_{\sqcup} }= I_{x}$.
\item $F_{x}\cap F_{y}= F_{x\sqcap y}$ and $F_{x_{\sqcap }}=F_{x}$.
\end{enumerate}
}\end{lemma}

\begin{lemma}{\rm 
\label{lema1}
  Let $F$ be a filter  and $I$ be an ideal of  $\textbf{D}$. 
\begin{enumerate}[{(i)}]
\item $F\cap D_{\sqcap}$ and $F\cap D_{\sqcup}$ are filters of the Boolean algebras $\textbf{D}_{\sqcap}$, $\textbf{D}_{\sqcup}$ respectively.
\item Each filter $F_{0}$ of the Boolean algebra $\textbf{D}_{\sqcap}$ is the base of some filter $F$ of $\textbf{D}$ such that $F_{0}=F\cap D_{\sqcap}$. Moreover if $F_{0}$ is a prime filter of $\textbf{D}_{\sqcap}$, $F\in\mathcal{F}_{p}(\textbf{D})$. 
\item $I\cap D_{\sqcap}$ and $I\cap D_{\sqcup}$ are ideals of the Boolean algebras $\textbf{D}_{\sqcap}$, $\textbf{D}_{\sqcup}$ respectively.
\item Each ideal $I_{0}$ of the Boolean algebra $\textbf{D}_{\sqcup}$ is the base of some ideal $I$ of $\textbf{D}$ such that $I_{0}=I\cap D_{\sqcup}$. Moreover if $I_{0}$  is prime, $I\in \mathcal{I}_{p}(\textbf{D})$. 
\end{enumerate}}
\end{lemma}
\noindent  (i) and (ii) of Lemma \ref{lema1} have been proved in \cite{wille}. (iii) and (iv) for ideals can be proved dually. 
\begin{observation}
{ \rm For each filter $F$ and ideal $I$ of $\textbf{D}$, $F\cap D_{\sqcap}$ is a base of $F$, $I\cap D_{\sqcup}$ is a base of $I$.}
\end{observation}

In \cite{howlader3}, one observed that  

\begin{proposition}
\label{comparison of two ideal}
{\rm \cite{howlader3}
 \noindent \begin{enumerate}[{(i)}]
\item $\mathcal{F}_{pr}(\textbf{D})$=$\mathcal{F}_{p}(\textbf{D}).$
\item $\mathcal{I}_{pr}(\textbf{D})$=$\mathcal{I}_{p}(\textbf{D})$.
\end{enumerate}}
\end{proposition}

\noindent Proposition \ref{comparison of two ideal} implies that  there  is a one-one and onto correspondence between the set of primary filters (ideals) of $\textbf{D}$ and the set of prime filters (ideals) of $\textbf{D}_{\sqcap} (\textbf{D}_{\sqcup}$).

\subsection{{\rm \textbf{Object oriented concept,  semiconcept and  protoconcept of a context}}}

% {\it Property oriented concepts}  were introduced using modal style operators, in \cite{duntsch2002modal} by D\"{u}ntsch and Gediga, who  pointed out limitations of Formal Concept Analysis as a tool for qualitative data analysis. 
 In the following, let $\mathbb{K}:=(G,M,R)$ be a context, and $A\subseteq G$, $B\subseteq M$. Operators $\lozenge, \square, \blacklozenge, \blacksquare$ are introduced on the power sets of $G,M$ as follows.\vskip 2pt \noindent 
 \hspace*{1.5cm} 
$B_{R}^{\lozenge}:=\{x\in G:R(x)\cap B\neq\emptyset\},~~~~ ~~~~~B_{R}^{\square}:=\{x\in G:R(x)\subseteq B\}$,
\vskip 3pt \noindent
 \hspace*{1.5cm}  $A_{R^{-1}}^{\blacklozenge}:=\{y\in M:R^{-1}(y)\cap A\neq\emptyset\},~~~ A_{R^{-1}}^{\blacksquare}:=\{y\in M:R^{-1}(y)\subseteq A\}$.

\vskip 2pt \noindent  If  the relation involved is clear from the context,  we shall omit the subscripts and denote $B_{R}^{\lozenge}$ by $B^{\lozenge}$, $B_{R}^{\square}$ by $B^{\square}$, and similarly for the case of $A$. 
\vskip 2pt \noindent 
Let us recall the notions of closure and interior operators.

\begin{definition}
{\rm \cite{davey2002introduction} An operator $C$ on the power set $\mathcal{P}(X)$ of a set $X$ is called a {\it closure operator} on $X$, if for all $A,B\in\mathcal{P}(X)$,
\begin{enumerate}
\item[C1] $A\subseteq C(A)$,
\item[C2] $A\subseteq B$ implies $C(A)\subseteq C(B)$, and
\item[C3] $C(C(A))=C(A)$.
\end{enumerate}
\noindent $A\in \mathcal{P}(X)$ is  called {\it closed} if and only if $C(A)=A$. \\
An {\it interior operator}  $I$ on the set $X$  is defined  dually. $A\in \mathcal{P}(X)$ is  {\it open} if and only if  $I(A)=A$.
}
\end{definition} 

\noindent We next list some properties of the operators $\square,\lozenge,\blacklozenge, \blacksquare$. 

\begin{theorem}{\rm \cite{ duntsch2002modal,yao2004comparative}
\label{property of box}
Let $A,A_{1},A_{2}\subseteq G$ and  $B,B_{1},B_{2}\subseteq M$.
\begin{enumerate}[{(i)}]
\item $A_{1}\subseteq A_{2}$ implies that $A_{1}^{\blacksquare}\subseteq A_{2}^{\blacksquare}$ and $A_{1}^{\blacklozenge}\subseteq A_{2}^{\blacklozenge}$.
\item $B_{1}\subseteq B_{2}$ implies that $B_{1}^{\square}\subseteq B_{2}^{\square}$ and $B_{1}^{\lozenge}\subseteq B_{2}^{\lozenge}$.
\item $(B_{1}\cap B_{2})^{\square}=B_{1}^{\square}\cap B_{2}^{\square}$ and $(B_{1}\cup B_{2})^{\lozenge}=B_{1}^{\lozenge}\cup B_{2}^{\lozenge}$ .
\item $(A_{1}\cap A_{2})^{\blacksquare}=A_{1}^{\blacksquare}\cap A_{2}^{\blacksquare}$ and $(A_{1}\cup A_{2})^{\blacklozenge}=A_{1}^{\blacklozenge}\cup A_{2}^{\blacklozenge}$ .
\item $B^{\square}=B^{c\lozenge c}$  and $A^{\blacksquare}=A^{c\blacklozenge c}$ .
\item $A^{\blacksquare}_{R}=A^{c\prime}_{-R}; B^{\square}_{R}= B^{c\prime}_{-R}$ and $A^{\blacklozenge}_{R}=A^{\prime c}_{-R}; B^{\lozenge}_{R}= B^{\prime c}_{-R} $ .
\item $A^{\blacksquare\lozenge\blacksquare}=A^{\blacksquare}$ and $B^{\square\lozenge\square}=B^{\square}$.
\item $A^{\blacklozenge\square\blacklozenge}=A^{\blacklozenge}$ and $B^{\lozenge\square\lozenge}=B^{\lozenge}$.
\item $\blacksquare\lozenge$ is interior operator on $G$ and $\lozenge\blacksquare$ is closure operator on $M$.
\end{enumerate}}
\end{theorem}
%Let $\mathbb{K}:=(G,M,R)$ be a context, $A\subseteq G, B\subseteq M$. 
%\begin{definition}\cite{duntsch2002modal}
%{\rm  $(A,B)$ is a  {\it property oriented concept}  of  $\mathbb{K}$ if it satisfies the conditions $A^{\blacklozenge}=B$ and $B^{\square}=A$. } 
%\end{definition}
 
In 2004, Yao defined {\it object oriented concepts}  \cite{yao2004comparative}.
\begin{definition} \label{objorc}
{\rm \cite{yao2004comparative} $(A,B)$  is an {\it object oriented concept}  of the  context $\mathbb{K}$ if  $A^{\blacksquare}=B$ and $B^{\lozenge}=A$. \\
A partial order  $\leq$  is  given by the following relation defined on the set of all object oriented concepts.  For any object oriented concepts 
$(A_{1},B_{1}),(A_{2},B_{2})$, \\
\hspace*{1.5cm} $(A_{1},B_{1})\leq (A_{2},B_{2})~\mbox{if and only if}~ A_{1}\subseteq A_{2}$ (equivalently, $B_{1}\subseteq B_{2}$).
}
\end{definition}

%\noindent A comparative study of Wille's concepts, property and object oriented concepts, and concept lattices is done extensively by Yao in \cite{yao2004comparative}.
%\vskip 2pt
 
In \cite{howlader2018algebras},  {\it object oriented semiconcepts} were  introduced in order to bring  the notion of negation into the study. 

\begin{definition}
{\rm \cite{howlader2018algebras} $(A, B)$ is an {\it object oriented semiconcept} of $\mathbb{K}$ if $A^{\blacksquare}=B$ or $B^{\lozenge}=A$.
}\end{definition}

\begin{notation} {\rm 
$RO-L(\mathbb{K})$ denotes the set of all object oriented concepts of $\mathbb{K}$, while the set of all object oriented semiconcepts is denoted by $\mathfrak{S}(\mathbb{K})$.
}
\end{notation}
\noindent In \cite{howlader2018algebras}, one observes the following. 
%\\**itemize wherever you have done so manually within proposition or other environments**.
\begin{proposition}\label{obssemi}
{\rm  \noindent 
\begin{enumerate}[{(i)}] \item $(A, B)\in \mathfrak{S}(\mathbb{K}) $  if and only if either $(A, B)=(A, A^{\blacksquare})$ or $(A, B)=(B^{\lozenge}, B).$ 
\item $RO-L(\mathbb{K})\subseteq \mathfrak{S}(\mathbb{K}).$
\item $(A, B)$ is a semiconcept of $\mathbb{K}$ if and only if $(A^{c}, B)$ is an object oriented  semiconcept of the context $\mathbb{K}^{c}$.
\end{enumerate}
}\end{proposition}

\noindent Operations $\sqcap, \sqcup, \lrcorner, \neg,\top,\bot$ are defined in $\mathfrak{S}(\mathbb{K})$ as follows.  
Let $(A_{1}, B_{1}), (A_{2}, B_{2}), (A,B)$ be any object oriented semiconcepts in $\mathfrak{S}(\mathbb{K})$.\\
$(A, B)\sqcap (C, D):=(A\cup C, (A\cup C)^{\blacksquare})$,
$(A, B)\sqcup (C, D):=((B\cap D)^{\lozenge}, B\cap D)$,\\
$\lrcorner(A, B):=(B^{c\lozenge}, B^c),~\neg(A, B):=(A^c, A^{c\blacksquare}),~\top :=(\emptyset, \emptyset),~\bot :=(G, M).$

\begin{notation}{\rm $\mathcal{S}(\mathbb{K})$ denotes the algebra formed by $\mathfrak{S}(\mathbb{K})$ with respect to the above operations.
}\end{notation}

\noindent Recall the algebra of semiconcepts $\underline{\mathfrak{H}}(\mathbb{K})$ (cf. Notation \ref{semiproto}).
\begin{theorem} \label{dualsemiconcept}
{\rm \cite{howlader2020} 
\noindent \begin{enumerate}[{(i)}]
 \item $\mathcal{S}(\mathbb{K}):=(\mathfrak{S}(\mathbb{K}),\sqcup,\sqcap,\neg,\lrcorner,\top,\bot)$ is a pure dBa.
 \item $\underline{\mathfrak{H}}(\mathbb{K})$ is  isomorphic to $\mathcal{S}(\mathbb{K}^{c})$.
\end{enumerate}
}\end{theorem}
\noindent The  dBa isomorphism in Theorem \ref{dualsemiconcept}(ii) above   is obtained using  Proposition \ref{obssemi}(iii). As mentioned in Section \ref{sec:Introduction}, it was shown in \cite{howlader2018algebras} that $\underline{\mathfrak{H}}(\mathbb{K})$ is {\it dually} isomorphic to the algebra of object oriented semiconcepts. This is because the operations on the set of object oriented semiconcepts considered in \cite{howlader2018algebras} are dual to those defined above. As we would like to prove representation results and such results  are usually obtained in terms of isomorphisms (and not dual isomorphisms), we consider the operations in the form given above.
%In this work, the operations are taken as 
%*Pure dBa is obtained by abstructing certain properties of the algebra of semiconcepts. As we want to prove the representation theorem for pure dBa in terms of object oriented semiconcepts, we consider the operation on the set of object oriented semiconcepts so that algebra of object oriented semiconcepts is isomorphic to the algebra of semiconcepts. So the operations we consider here are dual to that defined in \cite{howlader2018algebras}.
\vskip 2pt
With reference to $\mathfrak{S}(\mathbb{K})$ and the quasi-order $\sqsubseteq$ given in Definition \ref{DBA}, we get the following.
\begin{proposition}
\label{order object-semi}
{\rm  For any $(A,B),(C,D)\in \mathfrak{S}(\mathbb{K})$, $(A,B)\sqsubseteq (C,D)$  if and only if $C\subseteq A $ and $D\subseteq B$. }
\end{proposition}

%\begin{proof}
%Proof follows immediately from the partial order in a pure dBa (cf. Proposition \ref{order pure})**. 
%\end{proof}
%\vskip 2pt
We next define and give some properties of  {\it object oriented protoconcepts} which were introduced in  \cite{howlader2020}. 
%As before, $\mathbb{K}:=(G,M,R)$ is a context, $A\subseteq G, B\subseteq M$. 

\begin{definition}
\label{proto}
{\rm \cite{howlader2020} $(A, B)$ is an {\it object oriented protoconcept} of $\mathbb{K}$ if $A^{\blacksquare\lozenge}=B^{\lozenge}$. 
}\end{definition}

\begin{notation}
{\rm $\mathfrak{R}(\mathbb{K})$  denotes the set of all object oriented protoconcepts. 
}\end{notation}
\noindent  Observe  that an object oriented semiconcept is an object oriented protoconcept, that is $\mathfrak{S}(\mathbb{K})\subseteq\mathfrak{R}(\mathbb{K})$. Moreover, the following yields an equivalent definition of object oriented protoconcepts.
\begin{observation}
\label{equivalent def of proto}
{\rm $A^{\blacksquare\lozenge}=B^{\lozenge}$ if and only if $A^{\blacksquare}=B^{\lozenge\blacksquare}$.  
 }
 \end{observation}
 
\noindent 
A characterization of  object oriented protoconcepts  of $\mathbb{K}$ was established in \cite{howlader2020},  using a notion of `approximation' by object oriented concepts.

\vskip 2pt
\noindent  Similar to  Propositions \ref{obssemi}(iii) and \ref{order object-semi}, one obtains
\begin{proposition}
\label{relation box and prime}
{\rm $(A,B)$ is a protoconcept of $\mathbb{K}$  if and only if $(A^{c},B)$ is an object oriented protoconcept of $\mathbb{K}^{c}$.}
\end{proposition}

\noindent $\mathfrak{R}(\mathbb{K})$ is closed with respect to the operations $\sqcup,\sqcap,\neg,\lrcorner,\top,\bot$ defined on  $\mathfrak{S}(\mathbb{K})$ and we have

\begin{notation}{\rm $\underline{\mathfrak{R}}(\mathbb{K})$ denotes the algebra formed by $\mathfrak{R}(\mathbb{K})$ with respect to these operations.
}\end{notation}

\begin{proposition}
\label{largest sub algebra}
{\rm
\noindent\begin{enumerate}[{(i)}]
    \item $\mathcal{S}(\mathbb{K})$ is a subalgebra of $\underline{\mathfrak{R}}(\mathbb{K})$.
    \item $\mathcal{S}(\mathbb{K})=\underline{\mathfrak{R}}(\mathbb{K})_{p}$.
\end{enumerate}}
\end{proposition}
\noindent (ii) of the proposition can be easily proved using definitions of object oriented semiconcepts and the set $\mathfrak{R}(\mathbb{K})_{p}$.
Now recall the algebra of protoconcepts $\underline{\mathfrak{P}}(\mathbb{K})$ (cf. Notation \ref{semiproto}). One obtains
\begin{theorem}
\label{object-proto and proto}
{\rm 
\noindent \begin{enumerate}[{(i)}]
\item $\underline{\mathfrak{R}}(\mathbb{K}):=(\mathfrak{R}(\mathbb{K}),\sqcup,\sqcap,\neg,\lrcorner,\top,\bot)$ is  a dBa.
\item $\underline{\mathfrak{P}}(\mathbb{K})$   is isomorphic to $\underline{\mathfrak{R}}(\mathbb{K}^{c})$.
\end{enumerate}}
\end{theorem}
\noindent Theorem \ref{object-proto and proto}(ii) is obtained by using Proposition \ref{relation box and prime}. From  Theorem \ref{object-proto and proto}(ii) we get 

\begin{corollary}
{\rm $\underline{\mathfrak{R}}(\mathbb{K})$ is a fully contextual dBa.}
\end{corollary} 
Again we have
\begin{proposition}
\label{order object-proto}
{\rm  For any $(A,B),(C,D)\in \mathfrak{R}(\mathbb{K})$, $(A,B)\sqsubseteq (C,D)$  if and only if $C\subseteq A $ and $D\subseteq B$. %Moreover, $\sqsubseteq$ is a partial order on $\mathfrak{R}(\mathbb{K})$. 
}
\end{proposition}
%\noindent The dBa  $\underline{\mathfrak{R}}(\mathbb{K})$  is called the {\it algebra of object oriented  protoconcepts}.

A detailed example was given in [18] to motivate and illustrate both the notions of object oriented semiconcepts and object oriented protoconcepts. We re-present the example here to demonstrate that a fully contextual dBa need not be pure and vice versa.

\begin{example}
 \label{motivation}
{\rm Let $G:=\{q_{1},q_{2},q_{3},q_{4},q_{5},q_{6}\}$ be a  set of objects and consider a set of properties $S:=\{s_{1},s_{2},s_{3},s_{4},s_{5},s_{6},s_{7},s_{8},s_{9},s_{10},s_{11}\}$. 
%that are required to solve problems in $G$. 
Table \ref{context} below represents the context $\mathbb{K}:=(G, S,\Gamma)$, where  $(i,j)$-th cell containing a cross indicates  that  the object $q_{i}$ is related to the property $s_{j}$ by  $\Gamma$.
  \begin{table}[ht]
\begin{center}
\caption {Context $\mathbb{K}$} \label{context}
\begin{tabular}{|l|l|l|l|l|l|l|l|l|l|l|l|}
\hline
   & $s_{1}$ & $s_{2}$ & $s_{3}$ & $s_{4}$ & $s_{5}$ & $s_{6}$ & $s_{7}$ & $s_{8}$ & $s_{9}$& $s_{10}$& $s_{11}$ \\
   \hline
   $q_{1}$ &$\times$& &$\times$& &$\times$& &$\times$& &  &$\times$ &  \\
   \hline
   $q_{2}$ & &$\times$ & $\times$& $\times$& &$\times$ & &$\times$ &  & & \\
   \hline
   $q_{3}$ & $\times$&$\times$ & &$\times$ & &$\times$ & & & $\times$ & & \\
   \hline
   $q_{4}$ &$\times$ & & $\times$& & $\times$& & & &  & & \\
   \hline
   $q_{5}$ & & & & & $\times$& & &$\times$ & $\times$ & & $\times$\\
   \hline
   $q_{6}$& & & & & & & & & & &$\times$\\
   \hline
\end{tabular}
\end{center}
\end{table}\\}
 
 \end{example}
  Consider   $A_{1}:=\{q_{1},q_{2},q_{4},q_{6}\}$ and $B_{1}:=\{s_{3}\}$. Then $A_{1}^{\blacksquare\lozenge}=B_{1}^{\lozenge}$, so that $(A_{1},B_{1})$ is an object oriented protoconcept of $\mathbb{K}$. Observe that $B_{1}^{\lozenge}=\{q_{1},q_{2},q_{4}\}$ and so $B_{1}^{\lozenge}\neq A_{1}$; $A_{1}^{\blacksquare}=\{s_{3},s_{7},s_{10}\}$, so $A_{1}^{\blacksquare}\neq B_{1}$. This means  $(A_{1},B_{1})\in \mathfrak{R}(\mathbb{K})$ but $(A_{1},B_{1})\notin \mathfrak{S}(\mathbb{K})$. Now we consider the semiconcepts $(A_{1}, A_{1}^{\blacksquare})$ and $(B_{1}^{\lozenge}, B_{1})$. Since $(A_{1}, B_{1})$ is an object oriented protoconcept, the following equations hold: $(A_{1}, A_{1}^{\blacksquare})_{\sqcup}=(A_{1}^{\blacksquare\lozenge},A_{1}^{\blacksquare})=(B_{1}^{\lozenge},B_{1}^{\lozenge\blacksquare})=(B_{1}^{\lozenge},B_{1})_{\sqcap}$. Let us assume that there exists a semiconcept $(A,B)\in  \mathfrak{S}(\mathbb{K}) $ such that $(A,B)_{\sqcap}=(A,A^{\blacksquare})=(A_{1}, A_{1}^{\blacksquare})$ and $(A,B)_{\sqcup}=(B^{\lozenge}, B)=(B_{1}^{\lozenge}, B_{1})$, which implies that  $A=A_{1}$ and $B=B_{1}$ --  a contradiction. Therefore the subalgebra $\mathcal{S}(\mathbb{K})$ of $\underline{\mathfrak{R}}(\mathbb{K})$ is a pure dBa but not fully contextual. On the other hand, the fully contextual dBa  $\underline{\mathfrak{R}}(\mathbb{K})$  is not  pure, as $(A_{1}, B_{1})_{\sqcap}\neq (A_{1}, B_{1})$ and $(A_{1}, B_{1})_{\sqcup}\neq (A_{1}, B_{1})$.
%Now we recall the CTS $\mathbb{K}^{T}_{+}$ given in Example \ref{exmple of topocxt}.

\begin{observation}
\noindent
{\rm 
%\begin{enumerate}[{(i)}]
%\item Each context $\mathbb{K}$  produces   a pure dBa $\mathcal{S}(\mathbb{K})$ as well as a fully contextual dBa $\underline{\mathfrak{R}}(\mathbb{K})$.
The class $\textbf{K}$  of fully contextual dBas does not form  a variety \cite{MR648287}:  Example \ref{motivation} gives  $\underline{\mathfrak{R}}(\mathbb{K})\in \textbf{K}$ such that the subalgebra  $\mathcal{S}(\mathbb{K})\notin \textbf{K}$.
%\end{enumerate}
}
\end{observation}

\subsection{\rm{\textbf{Representation results}}}\label{repres}

%A summary of existing representation results related to dBas was discussed in \cite{howlader3}. We briefly re-present it here.
%Recall that  the dBa is an abstraction of the algebra of protoconcepts and the pure dBa that of the algebra of semiconcepts. Moreover, the algebra of protoconcepts is a fully contextual dBa.  
Let us recollect in this section, existing representation results related to arbitrary dBas as well as contextual, fully contextual and pure dBas. In the process, we also indicate the kind of representation results that are obtained in this work.
%Two  natural questions arise:\\
%{\bf Q1}: Is any fully contextual dBa (pure dBa)  isomorphic to the algebra of protoconcepts (semiconcepts) of some context?\\
%{\bf Q2}: Can any (contextual) dBa    be (quasi-)embedded into the algebra of protoconcepts   of some context? 
%\vskip 2pt
%\noindent {\bf Q2} is answered affirmatively  in \cite{wille} for dBas. To get to the result, 
%Some notations and definitions \cite{wille} required to state the results related to Q2 are as follows. 
Recall Definition \ref{DBAHOM},
%\vskip 2pt
%\begin{definition}
%\label{DBAHOM}
%{\rm Let $\textbf{D}$ and $\textbf{M}$ be two dBas. A dBa homomorphism $h:\textbf{M}\rightarrow \textbf{D}$ is called {\it quasi-injective}, when $x\sqsubseteq y$ if and only if $h(x)\sqsubseteq h(y)$, for all $x,y\in M$. A  quasi-injective and surjective  dBa homomorphism is called  a dBa {\it quasi-isomorphism}.
%}\end{definition}
%** Already defined in Notation 5 -- just add in the previous sentence: ".
%For every dBa $\textbf{D}$ recall the context $\mathbb{K}(\textbf{D})$ ** and the sets  $F_x, I_x$, $x\in D$ defined in Notation  \ref{Fx and Ix}**. \\
%For every dBa $\textbf{D}$, 
the standard context $\mathbb{K}(\textbf{D}):=(\mathcal{F}_{p}(\textbf{D}),\mathcal{I}_{p}(\textbf{D}),\Delta)$ for any dBa $\textbf{D}$, and 
%is defined in \cite{wille}, where for all $F\in \mathcal{F}_{p}(\textbf{D}), I\in \mathcal{I}_{p}(\textbf{D})$, $F\Delta I$ if and only if $F\cap I\neq\emptyset$. 
 the sets   $F_x, I_x$ defined  in Notation \ref{Fx and Ix}.
%Now for all $x\in A$  two sets are define as follows :
%$F_{x}=\{F\in \mathcal{F}_{p}(\textbf{A}): x\in F\}$ and
%$I_{x}=\{I\in\mathcal{I}_{p}(\textbf{A}):x\in I\}.$
The following theorem is proved by Wille in  \cite{wille}.

\begin{theorem} 
\label{protoembedding}
{\rm  \cite{wille} The map $h: D \rightarrow \mathfrak{P}(\mathbb{K}(\textbf{D}))$ defined by $h(x):=(F_{x},I_{x})$ for all $x\in D$, is a quasi-injective dBa homomorphism from \textbf{D} to $\underline{\mathfrak{P}}(\mathbb{K}(\textbf{D}))$.
}\end{theorem}

\vskip 2pt
\noindent Using Theorems \ref{protoembedding} and \ref{object-proto and proto} and Lemma \ref{derivation}(ii),  one gets a representation result for dBas in terms of object oriented protoconcepts also.

\begin{theorem}
\label{object-proto}
{\rm \cite{howlader2020} For a  dBa $\textbf{D}$, the map $h: D \rightarrow \mathfrak{R}(\mathbb{K}^{c}(\textbf{D}))$ defined by $h(x):=(F_{\neg x},I_{x})$ for any $x \in D$, is a quasi-injective dBa homomorphism from \textbf{D} to $\underline{\mathfrak{R}}(\mathbb{K}^{c}(\textbf{D}))$.
}\end{theorem}

\noindent In this work we show that, in fact, the map $h$ of Theorem \ref{object-proto} is a quasi-embedding into a subalgebra of $\underline{\mathfrak{R}}(\mathbb{K}^{c}(\textbf{D}))$. Moreover,  if $\textbf{D}$ is a  contextual dBa, this quasi-embedding  turns into an embedding.

\vskip 2pt 
Let $\textbf{D}$ be a finite dBa. Wille \cite{wille} has a representation result for this special case.  It may be observed that by Lemma \ref{lema1} and  Proposition \ref{comparison of two ideal}, the elements of $\mathcal{F}_{pr}(\textbf{D})$ in this case are just  the primary filters whose bases are  the principal filters of the Boolean algebra $\textbf{D}_{\sqcap}$ generated by its atoms, while the elements of $\mathcal{I}_{pr}(\textbf{D})$ are the primary ideals whose bases are the principal ideals of the Boolean algebra $\textbf{D}_{\sqcup}$ generated by its coatoms. In other words,  $F\in \mathcal{F}_{pr}(\textbf{D})$ if and only if $F=\{x\in D~:~ a\sqsubseteq x\}$ for some atom $a\in \mathcal{A}(\textbf{D}_{\sqcap})$, the set of all atoms of the Boolean algebra $\textbf{D}_{\sqcap}$.  $I\in \mathcal{I}_{pr}(\textbf{D})$  if and only if $I=\{x\in D~:~ x\sqsubseteq b\}$ for some $b\in \mathcal{C}(\textbf{D}_{\sqcup})$, the set of all coatoms of the Boolean algebra $\textbf{D}_{\sqcup}$. The  result  proved in \cite{wille} is as follows.

\begin{theorem}
\label{finite cas rep}
{\rm \cite{wille} Let $\textbf{D}$ be a finite dBa and $\mathbb{K}_{\mathcal{AC}}:=(\mathcal{A}(\textbf{D}_{\sqcap}),\mathcal{C}(\textbf{D}_{\sqcup}), \sqsubseteq)$. Then the map $h: D \rightarrow \mathfrak{P}(\mathbb{K}_{\mathcal{AC}})$ defined by $h(x):=(\{a\in \mathcal{A}(\textbf{D}_{\sqcap})~:~ a\sqsubseteq x\}, \{b\in \mathcal{C}(\textbf{D}_{\sqcup})~:~ x\sqsubseteq b\} )$ for all $x\in D$, is a quasi-injective dBa homomorphism from $\textbf{D}$ to $\underline{\mathfrak{P}}(\mathbb{K}_{\mathcal{AC}})$. Moreover, $\textbf{D}_{p}$ is isomorphic to $\underline{\mathfrak{H}}(\mathbb{K}_{\mathcal{AC}})$. }
\end{theorem}
\noindent From Theorems  \ref{finite cas rep},  \ref{object-proto and proto}, and  \ref{dualsemiconcept}, we have a representation theorem for finite dBas in terms of object oriented protoconcepts and object oriented semiconcepts.  

\begin{corollary}
\label{finite rep object}
{\rm  If $\textbf{D}$ is a finite dBa, the map $h: D \rightarrow \mathfrak{R}(\mathbb{K}^{c}_{\mathcal{AC}})$ defined by $h(x):=(\{a\in \mathcal{A}(\textbf{D}_{\sqcap})~:~ a\not{\sqsubseteq} x\}, \{b\in \mathcal{C}(\textbf{D}_{\sqcup})~:~ x\sqsubseteq b\} )$ for all $x\in D$, is a quasi-injective dBa homomorphism from $\textbf{D}$ to $\underline{\mathfrak{R}}(\mathbb{K}^{c}_{\mathcal{AC}})$. Moreover, $\textbf{D}_{p}$ is isomorphic to $\underline{\mathcal{S}}(\mathbb{K}^{c}_{\mathcal{AC}})$. }
\end{corollary}
\noindent We shall show that the representation result for dBas obtained in this work yields the above  as a special case when $\textbf{D}$ is finite.
%** Add result for finite dBas here from [40].**
%{\bf Q2} was addressed 
\vskip 2pt For a pure dBa,   the following is established in \cite{BALBIANI2012260}.
\begin{theorem} 
\label{semiconceptembedding}
{\rm \cite{BALBIANI2012260} If $\textbf{D}:=(D,\sqcup,\sqcap,\neg,\lrcorner,\top,\bot)$ is a pure dBa, the map $h:D \rightarrow \mathfrak{H}(\mathbb{K}(\textbf{D}))$ defined by $h(x):=(F_{x},I_{x})$ for all $x\in D$, is  an injective dBa homomorphism from \textbf{D} to $\underline{\mathfrak{H}}(\mathbb{K}(\textbf{D}))$.
}\end{theorem}

\noindent  From Lemma \ref{derivation}(ii),  Theorems \ref{dualsemiconcept} and  \ref{semiconceptembedding}, we then have a representation result in terms of object oriented semiconcepts.
\begin{theorem}
\label{representation pdba}
{\rm \cite{howlader2020} For a pure dBa $\textbf{D}$, the map $h:D \rightarrow \mathfrak{S}(\mathbb{K}^{c}(\textbf{D}))$ defined by $h(x):=(F_{\neg x},I_{x})$ for all $x\in D$ is an injective dBa homomorphism from \textbf{D} to $\mathcal{S}(\mathbb{K}^{c}(\textbf{D}))$.
}\end{theorem}

%\vskip 2pt
%Let us next address  {\bf Q1}. 
A dBa $\textbf{D}$ is  {\it complete} \cite{vormbrock2005semiconcept} if and only if the Boolean algebras $\textbf{D}_{\sqcup}$ and $\textbf{D}_{\sqcap}$ are complete. Vormbrock \cite{vormbrock2005semiconcept}  proved that any complete pure dBa $\textbf{D}$ for which   $\textbf{D}_{\sqcap}$ and $\textbf{D}_{\sqcup}$ are atomic Boolean algebras, is isomorphic to the algebra of semiconcepts of some context. Furthermore, any complete fully contextual dBa for which   $\textbf{D}_{\sqcap}$ and $\textbf{D}_{\sqcup}$ are atomic Boolean algebras is isomorphic to the  algebra of protoconcepts of some context.
However, all dBas are clearly not complete -- one may simply consider  Boolean algebras that are not complete. What about an isomorphism theorem of the above kind for  (fully contextual/pure) dBas  in general? An attempt in this regard was made in \cite{breckner2019topological}, for contextual dBas. However, a counterexample to the proof of the representation result of \cite{breckner2019topological} was given in \cite{howlader3}. In this paper, we establish  isomorphism theorems for fully contextual and  pure dBas in terms of algebras of certain protoconcepts and semiconcepts (respectively) of contexts equipped with topologies.
%Theorem   \ref{object-proto} gives that any dBa is quasi-isomorphic to a subalgebra of the object oriented protoconcept algebra of some context. Can this subalgebra be characterized?

\section{Some further  results on dBas}
\label{AIdBa}
\noindent 
%**We know that every Boolean algebra is a dBa such that $\neg a=\lrcorner a$ and $\neg\neg a=a$. In this section, we show that the converse of the statement is also true. So we may say dBa differs from Boolean algebra by its negations. Then we study some result of Boolean algebra  in the context of  dBa and use this result in Section \ref{RT}  and Section \ref{CE}.
In this section, we derive some further  results on dBas that will be used in Sections \ref{RT}  and \ref{CE}.     In particular, we  give a relationship between Boolean algebras and dBas in Theorem \ref{Boolean and dBa} below. To prove this theorem, we require the following.

\begin{proposition}
\label{meet join}
{\rm  Let $\textbf{D}:=(D,\sqcup,\sqcap,\neg,\lrcorner,\top,\bot)$ be a dBa. For any  $x,y\in D$, the following hold.
 \begin{enumerate}[{(i)}]
\item $x\sqcap y\sqsubseteq x\vee y\sqsubseteq x\sqcup y$.
\item $x\sqcap y\sqsubseteq x\wedge y\sqsubseteq x\sqcup y$.
\end{enumerate}}
\end{proposition}
\begin{proof}
The proof of (ii) is dual to the proof of (i), and we only  prove (i). For any $x,y\in D$, $\neg x\sqcap \neg y\sqsubseteq \neg x$ and $\neg x\sqcap \neg y\sqsubseteq \neg y.$ So by Proposition \ref{pro2}(ii), $\neg\neg x\sqsubseteq \neg(\neg x\sqcap\neg y)$ and $\neg\neg y\sqsubseteq \neg(\neg x\sqcap\neg y).$ Then Proposition \ref{pro1.5}(vi) gives $\neg\neg x\sqcap \neg\neg y\sqsubseteq \neg(\neg x\sqcap\neg y)\sqcap\neg\neg y$ and $\neg\neg y\sqcap \neg(\neg x\sqcap\neg y)\sqsubseteq\neg(\neg x\sqcap\neg y)\sqcap \neg(\neg x\sqcap\neg y).$ Therefore $\neg\neg x\sqcap \neg\neg y\sqsubseteq\neg(\neg x\sqcap\neg y)\sqcap \neg(\neg x\sqcap\neg y).$  By Proposition \ref{pro2}(i), $\neg\neg x\sqcap \neg\neg y\sqsubseteq\neg(\neg x\sqcap\neg y)$,  that is $(x\sqcap x)\sqcap (y\sqcap y)\sqsubseteq x\vee y.$ By axiom $(1a)$ and $(3a)$, $x\sqcap y\sqsubseteq x\vee y$.

\noindent We know that $x,y\sqsubseteq x\sqcup y.$ Proposition \ref{pro2}(ii) gives $\neg(x\sqcup y)\sqsubseteq \neg x,\neg y.$ Therefore by  Proposition \ref{pro1.5}(vi), $\neg (x\sqcup y)\sqcap\neg y\sqsubseteq \neg x\sqcap\neg y$ and $\neg (x\sqcup y)\sqcap\neg (x\sqcup y)\sqsubseteq \neg (x\sqcup y)\sqcap\neg y.$ So $\neg (x\sqcup y)\sqcap \neg (x\sqcup y)\sqsubseteq \neg x\sqcap\neg y.$ By  Proposition \ref{pro2}(i),  $\neg(x\sqcup y)\sqsubseteq \neg x\sqcap\neg y,$ and by Proposition \ref{pro2}(ii),  $\neg(\neg x\sqcap\neg y)\sqsubseteq \neg\neg (x\sqcup y)=(x\sqcup y)\sqcap (x\sqcup y)\sqsubseteq x\sqcup y.$ Hence $x\vee y\sqsubseteq x\sqcup y.$
\end{proof}

%** add pure**
\begin{theorem}
\label{Boolean and dBa}
{\rm Any Boolean algebra $(D,\sqcap,\sqcup,\neg,\top,\bot)$ forms a fully contextual as well as pure dBa $\textbf{D}:=(D,\sqcap,\sqcup,\neg,\lrcorner,\top,\bot),$ where for all $a\in D$, $\lrcorner a:=\neg a$. On the other hand, a dBa $\textbf{D}:=(D,\sqcap,\sqcup,\neg,\lrcorner,\top,\bot)$ forms a Boolean algebra $(D,\sqcap,\sqcup,\neg,\top,\bot),$ if for all $a\in D$, $\neg a=\lrcorner a$ and $\neg\neg a=a$.  }
\end{theorem}
\begin{proof}
It is easy to see that in a Boolean algebra  $(D,\sqcap,\sqcup,\neg,\top,\bot)$, if we set   $\lrcorner a:=\neg a$ for all $a\in D$, $(D,\sqcap,\sqcup,\neg,\lrcorner,\top,\bot)$ forms a dBa. The dBa is also pure, due to the idempotence of the operators $\sqcap$ and $\sqcup$ in a Boolean algebra. Further, note that in this case $D_\sqcap=D_\sqcup=D$, and the dBa is fully contextual as well. We also have 
%$(D,\sqcap,\sqcup,\neg,\top,\bot)$ is a Boolean algebra, 
$\neg\neg a=a$ for all $a\in D$. \\
Now 
let $\textbf{D}$ be a dBa such that for all $a\in D$, $\neg a=\lrcorner a$ and $\neg\neg a=a$. Let $x,y\in D$ such that  $x\sqsubseteq y$ and $y\sqsubseteq x$. By Proposition \ref{pro1.5}(iv),  $x\sqcap x=y\sqcap y$ and $x\sqcup x=y\sqcup y$. Using Proposition \ref{pro2}(iii), $\neg\neg x=\neg\neg y$ and so $x=y$. Therefore $(D,\sqsubseteq)$ is a partially ordered set. From Definition \ref{DBA}(2a and 2b) it follows that $\sqcap,\sqcup$ is commutative, while Definition \ref{DBA}(3a and 3b) gives that $\sqcap,\sqcup$ is associative. Using Definition \ref{DBA}(5a) and Proposition \ref{pro2}(iii), $x\sqcap (x\sqcup y)=x\sqcap x=\neg\neg x$. So $x\sqcap (x\sqcup y)= x$. Again using Definition \ref{DBA}(5b) and Proposition \ref{pro2}(iii), $x\sqcup (x\sqcap y)= x$. Therefore $(D,\sqcap,\sqcup,\neg,\top,\bot,\sqsubseteq)$ is a bounded complemented lattice. To show it is a distributive lattice, let $x,y,z\in D.$ Proposition \ref{meet join} implies that $x\sqcap y\sqsubseteq x\wedge y $ and $x\vee y\sqsubseteq x\sqcup y$. Using Proposition \ref{pro1.5}(v) and Proposition \ref{pro2}(ii),  $\neg y\sqcap\neg z\sqsubseteq \neg y,\neg z$. So $\neg\neg y\sqsubseteq \neg(\neg y\sqcap\neg z)$ and $\neg\neg z\sqsubseteq \neg (\neg y\sqcap \neg z)$. Therefore $y\sqsubseteq \neg(\neg y\sqcap\neg z)=y\vee z$ and $ z\sqsubseteq \neg (\neg y\sqcap \neg z)=y\vee z$. Proposition \ref{pro1.5}(vi) gives $y\sqcup z \sqsubseteq y\vee z$, as $(y\vee z)\sqcup (y\vee z)=\lrcorner\lrcorner (y\vee z)=\neg\neg(y\vee z)=y\vee z$. So $y\sqcup z=y\vee z$. Dually we can show that $y\sqcap z=y\wedge z$. From Definition \ref{DBA}(6a and 6b) it follows that $(D,\sqcap,\sqcup,\neg,\top,\bot,\sqsubseteq)$ is a complemented distributive lattice and hence a Boolean algebra.
\end{proof}

For each dBa $\textbf{D}$ there is a pure subalgebra $\textbf{D}_{p}$, by Proposition \ref{puresub}. In the case of fully contextual dBas, these subalgebras play a special role.
%there is a special relationship between the algebra $\textbf{D}$ and the pure subalgebra $\textbf{D}_{p}$.*
\begin{theorem}
\label{iso}
{\rm Let $\textbf{D}$ and $\textbf{M}$ be fully contextual dBas. Then $\textbf{D}$ is isomorphic to $\textbf{M}$ if and only if $\textbf{D}_{p}$ is isomorphic to  $\textbf{M}_{p}$. Moreover, every dBa isomorphism  from $\textbf{D}_{p}$  to  $\textbf{M}_{p}$ can be uniquely extended to a dBa isomorphism from $\textbf{D}$ to  $\textbf{M}$. }
\end{theorem}
\begin{proof}
Let  $f$ be a dBa isomorphism from  $\textbf{D}$ to $\textbf{M}$. We show that $f\vert_{D_{p}}$, the restriction of $f$ to $\textbf{D}_{p}$,  is the required dBa isomorphism from $\textbf{D}_{p}$  to  $\textbf{M}_{p}$. For that, it is enough to prove that $f\vert_{D_{p}}(D_{p})= M_{p}$. Let $x\in D_{p}$. Then either $x\sqcap x=x$ or $x\sqcup x=x$, which implies that either $f\vert_{D_{p}}(x)\sqcap f\vert_{D_{p}}(x)=f(x)\sqcap f(x)=f(x\sqcap x)=f(x)=f\vert_{D_{p}}(x)$ or $f\vert_{D_{p}}(x)\sqcup f\vert_{D_{p}}(x)=f(x)\sqcup f(x)=f(x\sqcup x)=f(x)=f\vert_{D_{p}}(x)$, as $f$ is a homomorphism. So $f\vert_{D_{p}}(D_{p})\subseteq M_{p}$. Conversely, let $y\in M_{p}$. As $f$ is surjective,  there exists a $c\in D$ such that $f(c)=y$. Now $f(c\sqcap c)=y\sqcap y$ and $f(c\sqcup c )=y\sqcup y$. Since  $y\in M_{p}$, either $y\sqcap y=y$ or $y\sqcup y=y$, which implies that either $f(c\sqcap c)=y=f(c)$ or $f(c\sqcup c )=y=f(c)$. So either $c\sqcap c=c$ or $c\sqcup c=c$, as $f$ is injective. So $c\in D_{p}$. Therefore  $M_{p}=f\vert_{D_{p}}(D_{p})$.

Let $\textbf{D}_{p}$ be isomorphic to  $\textbf{M}_{p}$ and $h$ be a dBa isomorphism from  $\textbf{D}_{p}$ to  $\textbf{M}_{p}$. Let $x\in D$. Then $x\sqcap x\in D_{\sqcap}\subseteq D_{p}$, $x\sqcup x\in D_{\sqcup}\subseteq D_{p}$ and $(x\sqcap x)\sqcup (x\sqcap x)=(x\sqcup x)\sqcap (x\sqcup x)$ by Definition \ref{DBA}(12). As $h$ is a dBa isomorphism, $h(x\sqcap x)\in M_{\sqcap}\subseteq M_{p}$, $h(x\sqcup x)\in M_{\sqcup}\subseteq M_{p}$ and $h(x\sqcap x)\sqcup h(x\sqcap x)=h(x\sqcup x)\sqcap h(x\sqcup x)$. As $\textbf{M}$ is fully contextual,   there exists a unique $c^{x}\in M$ such that $c^{x}_{\sqcap}= h(x\sqcap x)$ and $c^{x}_{\sqcup}= h(x\sqcup x)$. Let us define a map $f: D\rightarrow M$ by $f(x):=c^{x}$, for all $x\in D$. From the definition of $c^{x}$, it follows that $f$ is well-defined. We show that $f$ is the required dBa isomorphism.  \\
%Indeed, $f$ is a dBa homomorphism:\\
(i) Let $x,y\in D$ and $f(x)=c^{x}, f(y)=c^{y}$. We will show that $f(x\sqcap y)=c^{x\sqcap y}=c^{x}\sqcap c^{y}$. Now $c^{x}_{\sqcap}=h(x\sqcap x)$, $c^{x}_{\sqcup}=h(x\sqcup x)$ and $c^{y}_{\sqcap}=h(y\sqcap y)$, $c^{y}_{\sqcup}=h(y\sqcup y)$. So $c^{x}_{\sqcap}\sqcap c^{y}_{\sqcap}=h(x\sqcap x)\sqcap h(y\sqcap y)=h(x\sqcap y)$, using Definition \ref{DBA}(1a) and the fact that $h$ is a homomorphism. This gives $h(x\sqcap y)\sqcap h(x\sqcap y)=(c^{x}_{\sqcap}\sqcap c^{y}_{\sqcap})_{\sqcap}$  and  $h(x\sqcap y)\sqcup h(x\sqcap y)=(c^{x}_{\sqcap}\sqcap c^{y}_{\sqcap})_{\sqcup}$. By the uniqueness property of $c^{x\sqcap y}$, $c^{x\sqcap y}=c^{x}_{\sqcap}\sqcap c^{y}_{\sqcap}=c^{x}\sqcap c^{y}$ (again, using Definition \ref{DBA}(1a)). Hence $f(x\sqcap y)=f(x)\sqcap f(y)$.\\
Dually, one gets $f(x\sqcup y)=f(x)\sqcup f(y)$.\\
(ii) In order to establish that $f$ preserves the other dBa operations, we first show that for all $x\in D_{p}$, $f(x)=h(x)$, that is, $f\vert_{D_{p}}=h$. Let $x\in D_{p}$. Then either $x\sqcap x=x$ or $x\sqcup x=x$. Let us assume that $x\sqcap x=x$. Then, using (i),  $f(x)=f(x\sqcap x)=f(x)\sqcap f(x)=c^{x}\sqcap c^{x}=c^{x}_{\sqcap}=h(x\sqcap x)=h(x)$ and if $x\sqcup x=x$, $f(x)=f(x\sqcup x)=f(x)\sqcup f(x)=c^{x}\sqcup c^{x}=c^{x}_{\sqcup}=h(x\sqcup x)=h(x)$. Therefore $f$ is an extension of $h$.\\
 Now by Propositions \ref{pro1.5} and \ref{pro2}, it follows that $\top,\bot, \neg x,\lrcorner x\in D_{p}$ for all $x\in D$. So, using Definition \ref{DBA}(4a) and (ii),  $f(\neg x)=f(\neg(x\sqcap x))= h(\neg(x\sqcap x))=\neg h(x\sqcap x)=\neg f(x\sqcap x)=\neg (f(x)\sqcap f(x))=\neg f(x)$, as $h$ is a dBa homomorphism and $f\vert_{D_{p}}=h$. Similarly, one gets $f(\lrcorner x)=\lrcorner f(x), f(\top)=\top$, and $f(\bot)=\bot$. So $f$ is a dBa homomorphism.\\ 
(iii)  $f$ is injective: let $x,y\in D$, such that $c^{x}= c^{y}$. Then $h(x\sqcap x)=h(y\sqcap y)$ and $h(x\sqcup x)=h(y\sqcup y)$. So $x\sqcap x=y\sqcap y$ and $x\sqcup x=y\sqcup y$, as $h$ is injective. By Proposition \ref{pro1.5}(iv), $x\sqsubseteq y$ and $y\sqsubseteq x$. $\sqsubseteq$ is a partial order, as $\textbf{D}$ is a contextual dBa. So $x=y$.\\
(iv) To show $f$ is surjective, let $a\in M$. Since $h$ is a dBa isomorphism from $\textbf{D}_{p}$ to $\textbf{M}_{p}$, as done above, one can find a unique $x^{a}\in D$ such that $x^{a}_{\sqcap}=h^{-1}(a\sqcap a)$ and $x^{a}_{\sqcup}=h^{-1}(a\sqcup a)$. So $c^{x^{a}}_{\sqcap}=h(x^{a}_{\sqcap})=a\sqcap a$ and $c^{x^{a}}_{\sqcup}=h(x^{a}_{\sqcup})=a\sqcup a$. By Proposition \ref{pro1.5}(iv), $c^{x^{a}}\sqsubseteq a$ and $a\sqsubseteq c^{x^{a}}$. As $\textbf{D}$ is a contextual dBa, $\sqsubseteq$ is a partial order on $D$. So $f(x^{a})=c^{x^{a}}=a$. Hence $f$ is surjective.\\
 Therefore $f$ is a dBa isomorphism from $\textbf{D}$ to $\textbf{M}$ such that $h=f\vert_{D_{p}}$.
 
 If possible, assume that there exists another dBa isomorphism $f_{1}$ from $\textbf{D}$
 to $\textbf{M}$ such that $f_{1}\vert_{D_{p}}=h$. Let $x\in D$. Then $f(x)\sqcap f(x)=f(x\sqcap x)=h(x\sqcap x)=f_{1}(x\sqcap x)=f_{1}(x)\sqcap f_{1}(x)$. Dually, one can show that $f(x)\sqcup f(x)=f_{1}(x)\sqcup f_{1}(x)$. So $f(x)\sqsubseteq f_{1}(x)$ and $f_{1}(x)\sqsubseteq f(x)$, which implies that $f(x)=f_{1}(x)$, as $\textbf{D}$ is contextual.
\end{proof}

 In the following, let $\textbf{D}:=(D,\sqcup,\sqcap,\neg,\lrcorner,\top,\bot)$ be a dBa.

\begin{notation} 
\label{finite boolean join and meet}
{\rm For a non-empty finite subset $B$ of $ D$, \\ $\sqcap B:=\sqcap_{a\in B}a$, $\sqcup B:=\sqcup_{a\in B}a$, $\vee B:=\vee_{a\in B}a$, $\wedge B:=\wedge_{a\in B}a$.}
\end{notation} 
\begin{note}
\label{finite-B-meet}
{\rm Using Proposition \ref{pro2}(iv)  and induction on the cardinality of $B$, one can show that for each non-empty finite subset $B$ of $D$, $\vee B\in D_{\sqcap}$ and $\wedge B\in D_{\sqcup}$.}
\end{note}
\begin{proposition}
\label{relation of join and meet}
{\rm For each non-empty  finite subset $B$ of $ D,$  the following hold. 
\begin{enumerate}[{(i)}]
\item $\sqcap B\sqsubseteq \vee B\sqsubseteq \sqcup B$.
\item $\sqcap B\sqsubseteq \wedge B\sqsubseteq \sqcup B$.
\end{enumerate}}
\end{proposition}
\begin{proof}
The proof of (ii) is similar to (i). (i) can be easily proved using induction on the cardinality of $B$ and Proposition \ref{meet join}.
%. Let $B:=\{a,b\}.$ Then by Proposition \ref{meet join} we have $\sqcap B\sqsubseteq \vee B\sqsubseteq \sqcup B$. Now we assume that $1$ is true every finite subset $B$ such that cardinality of $B$ $|B|\leq n$. Let $B$ be finite subset of $D$ with $|B|=n+1$ and $a\in B$ . Then $B=B\setminus \{a\}\cup\{a\}$. Then by induction hypothesis  we have  $\sqcap (B\setminus\{a\})\sqsubseteq \vee (B\setminus\{a\})\sqsubseteq \sqcup (B\setminus\{a\}).$ Therefore by Proposition \ref{meet join} we have $\sqcap(B\setminus\{a\})\sqcap a\sqsubseteq \vee(B\setminus\{a\})\vee a\sqsubseteq \sqcup(B\setminus\{a\})\sqcup a$. Hence $\sqcap B\sqsubseteq \vee B\sqsubseteq \sqcup B$.
\end{proof}

\begin{theorem}
\label{primary ideal thorem}
{\rm The following hold for $\textbf{D}$.
\begin{enumerate}[{(i)}]
\item If $I$ is a proper ideal in $\textbf{D}$ then there exists a primary ideal $I^{1}$ in $\textbf{D}$ such that $I\subseteq I^{1}$.
\item If $F$ is a proper filter in $\textbf{D}$ then there exists a primary filter $F^{1}$ in $\textbf{D}$ such that $F\subseteq F^{1}$.
\end{enumerate}}
\end{theorem}
\begin{proof}
(i) Let $I$ be a proper ideal in $\textbf{D}.$ Then by Lemma \ref{lema1}(iii),  $I\cap D_{\sqcup}$ is an ideal in the Boolean algebra $\textbf{D}_{\sqcup}$. Since $I$ is proper in $\textbf{D}$, $I\cap D_{\sqcup}$  must also be proper in $\textbf{D}_{\sqcup}$. (For, if not, $\top\in I\cap D_{\sqcup}$  and so $\top \in I$, which is not possible in $\textbf{D}_{\sqcup}$.)  Then there exists a prime ideal $I_{0}$ in $\textbf{D}_{\sqcup}$ such that $I\cap D_{\sqcup}\subseteq I_{0}$. We define $I^{1}:=\{x\in D~:~x\sqsubseteq y ~\mbox{for some}~ y\in I_{0}\},$ and 
let $x_{1},x_{2}\in I^{1}.$ So there are $y_{1},y_{2}\in I_{0}$ such that $x_{1}\sqsubseteq y_{1}, x_{2}\sqsubseteq y_{2}$. Using Proposition \ref{pro1.5}(vi), we get 
%$x_{1}\sqcup x_{2}\sqsubseteq y_{1}\sqcup x_{2}$ and $x_{2}\sqcup y_{1}\sqsubseteq y_{1}\sqcup y_{2}$. Therefore 
$x_{1}\sqcup x_{2}\sqsubseteq y_{1}\sqcup y_{2}$. Since $y_{1},y_{2}\in I_{0}$ and $I_{0}$ is an ideal in $\textbf{D}_{\sqcup}$,  $y_{1}\sqcup y_{2}\in I_{0}$. Thus $x_{1}\sqcup x_{2}\in I^{1}$. Let $x\in D$ such that $x\sqsubseteq x_{0}$ for some $x_{0}\in I^{1}$. Then there exists $y\in I_{0}$ such that $x\sqsubseteq x_{0}\sqsubseteq y$ and so $x\in I^{1}$. Therefore $I_{0}$ is an ideal in $\textbf{D}$.

 Now we will show that $I^{1}\cap D_{\sqcup}= I_{0}$. Let $x\in I^{1}\cap D_{\sqcup}.$ Then $x\in I^{1}$ implies that there exists $y\in I_{0}\subseteq D_{\sqcup}$ such that $x\sqsubseteq y$. Since $x,y\in D_{\sqcup}$ and $I_{0}$ is an ideal of $\textbf{D}_{\sqcup}$, by Proposition \ref{pro1} we get $x\in I_{0}.$ Therefore $I^{1}\cap D_{\sqcup}\subseteq I_{0}.$ By definition  of $I^{1}$ it is clear that $I_{0}\subseteq I^{1}.$ Therefore $I^{1}\cap D_{\sqcup}=I_{0}.$ Thus by Proposition \ref{comparison of two ideal} it follows that $I^{1}$ is a primary ideal. Now we claim that $I\subseteq I^{1}$. Let $x\in I.$ Then $x\sqcup x\in I$, as $I$ is an ideal. Therefore $x\sqcup x\in I\cap D_{\sqcup}\subseteq I_{0}$. This implies that $x\sqcup x\in I_{0}$. So $x\sqcup x\in I^{1}$. Since $x\sqsubseteq x\sqcup x$ and $I^{1}$ is an ideal, $x\in I^{1}.$ Hence $I\subseteq I^{1}$.
 \vskip 2pt 
\noindent   Dually one can prove (ii).
\end{proof}

%\begin{proposition}
%\label{maximalthm}
%{\rm The following hold for $\textbf{D}$.
%\begin{enumerate}
%\item[(i)] If $I$ is primary ideal in $\textbf{D}$ then $I$ is a maximal ideal.
%\item[(ii)] If $F$ is primary filter in $\textbf{D}$ then $F$ is a maximal filter.
%\end{enumerate}}
%\end{proposition}
%\begin{proof} %[\textbf{Proof of i:}]
%(i) Let $I$ be a primary ideal in $\textbf{D}$, and let $I\subseteq I_{0}$ for some ideal in $\textbf{D}.$ If possible, let $I\neq I_{0}.$ Then there exists $x \in D$ such that $x\in I_{0}$ but $x\notin I.$ So $\lrcorner x\in I\subseteq I_{0}.$ Therefore $\top=x\sqcup\lrcorner x\in I_{0}.$ Hence $I_{0}=D$.
%Dually we can prove (ii).
%\noindent {\it \textbf{Proof of ii :}}Let $F$ be a primary filter. To show let $F\subseteq F_{0}$ for some filter in $\textbf{D}.$ If possible let $F\neq F_{0}.$ Then there exists $x\in F_{0}$ but $x\notin F.$ So $\neg x\in F\subseteq F_{0}.$ Therefore $\bot=x\sqcap\neg x\in F_{0}.$ Hence $F_{0}=D.$
%\end{proof}

\begin{proposition}
\label{inv-img-prim}
{\rm Let $\textbf{M}$ and $\textbf{D}$ be two dBas, and let $h:\textbf{M}\rightarrow \textbf{D}$ be a dBa homomorphism. The following hold.
\begin{enumerate}[{(i)}]
\item If $a\sqsubseteq b$ then $h(a)\sqsubseteq h(b)$, for all $a,b\in M.$
\item If $I$ is a primary ideal in $\textbf{D}$ then $h^{-1}(I)$ is primary ideal in $\textbf{M}.$
\item If $F$ is a primary filter in $\textbf{D}$ then $h^{-1}(F)$ is a primary filter in $\textbf{M}.$
\end{enumerate}
Moreover, if $h$ is a dBa quasi-isomorphism, the following hold.
\begin{enumerate}
\item[{(iv)}] If $I$ is a primary ideal in $\textbf{M}$ then $h(I)$ is a primary ideal in $\textbf{D}$.
\item[{(v)}] If $F$ is a primary filter in $\textbf{M}$ then $h(F)$ is a primary filter in $\textbf{D}$. 
\end{enumerate}}
\end{proposition}
\begin{proof}
(i) Let $a,b\in M$ such that $a\sqsubseteq b.$ Then $a\sqcap b=a\sqcap a$ and $a\sqcup b=b\sqcup b.$ Therefore $h(a)\sqcap h(b)=h(a\sqcap b)=h(a\sqcap a)=h(a)\sqcap h(a)$ and $h(a\sqcup b)=h(b\sqcup b)=h(b)\sqcup h(b).$ Hence $h(a)\sqsubseteq h(b).$\\
The proof of (iii) is dual to the proof of (ii) and the proof of (v) is dual to that of (iv). We prove (ii) and (iv).\\ 
(ii) Let $I$ be a primary ideal in $\textbf{D},$ and let $a,b\in h^{-1}(I).$ Then $h(a\sqcup b)=h(a)\sqcup h(b)\in I$, as $h(a),h(b)\in I$ and $I$ is an ideal. Therefore $a\sqcup b\in h^{-1}(I)$. Now let $a\in h^{-1}(I)$ and $x\sqsubseteq a$ for some $x\in M.$ Then by (1), $h(x)\sqsubseteq h(a)$. Therefore $h(x)\in I,$ which implies $x\in h^{-1}(I)$. If possible, suppose $h^{-1}(I)=M.$ Then $\top_{M}\in h^{-1}(I)$ implies that $\top_{D}=h(\top_{M})\in I$, which is not possible. Hence $h^{-1}(I)$ is a proper ideal in $\textbf{M}$. Now let $x\in M.$ Then either $h(x)\in I$ or $\lrcorner h(x)\in I$, as $I$ is primary. That is, either $ h(x)\in I$ or $  h(\lrcorner x)\in I.$ So either $x\in h^{-1}(I)$ or $\lrcorner x\in h^{-1}(I)$. Hence $h^{-1}(I)$ is a primary ideal.\\
%{\it\textbf{Proof of 3:}}Let $F$ be an primary filter in $\textbf{D}$. Now let $a,b\in h^{-1}(F)$. Then $h(a\sqcap b)=h(a)\sqcap h(b)\in I$ as $h(a),h(b)\in F$ and $F$ is an filter. Therefore $a\sqcap b\in h^{-1}(F)$. Now let $a\in h^{-1}(F)$ and $a\sqsubseteq x$ for some $x\in M.$ Then by (1) we have $h(a)\sqsubseteq h(x)$. Therefore $h(x)\in F$ and hence $x\in h^{-1}(F)$. Now if possible let $h^{-1}(F)=M$ then $\bot_{M}\in h^{-1}(F)$ implies that $\bot_{D}=h(\bot_{M})\in F$ which is not possible. Hence $h^{-1}(F)$ is proper filter in $\textbf{M}$. Now Let $x\in M$ then either $h(x)\in F$ or $\neg h(x)\in F$ as $F$ is primary filter . That is either $ h(x)\in F$ or $  h(\neg x)\in F$. So either $x\in h^{-1}(F)$ or $\neg x\in h^{-1}(F)$. Hence $h^{-1}(F)$ is a primary filter
\noindent (iv) Let $I$ be a primary ideal in $\textbf{M}$ and $h(a),h(b)\in h(I)$ for some $a,b\in I$. Since $I$ is an ideal and $a,b\in I$, $h(a)\sqcup h(b)=h(a\sqcup b)\in h(I)$. Now let $x\sqsubseteq h(a)$ for some $x\in D$. As $h$ is surjective, there exists a $d\in M$ such that $h(d)=x$. So $h(d)\sqsubseteq h(a)$. Therefore $d\sqsubseteq a$, as $h$ is quasi-injective. So $d\in I$. Therefore $x=h(d)\in h(I)$. If possible, suppose $h(I)=D$. Then $\top\in h(I)$, and there exists $y\in I$ such that $h(y)=\top=h(\top),$ which implies that $h(\top)\sqsubseteq h(y)$. So $\top\sqsubseteq y$, as $h$ is quasi-injective. Then $\top\in I$, which is a contradiction. Hence $h(I)$ is a proper ideal in $\textbf{D}$. Now let $z\in D$. As $h$ is surjective, there exists $e\in M$ such that $h(e)=z$. As $I$ is a primary ideal, either $e\in I$ or $\lrcorner e\in I$. Therefore $z=h(e)\in h(I)$ or $\lrcorner z=\lrcorner h(e)=h(\lrcorner e)\in h(I)$. Hence $h(I)$ is a primary ideal.
%\noindent {\it\textbf{Proof of 2:}}Let $F$ be an primary filter in $\textbf{M}$ and $h(a),h(b)\in h(F)$ for some $a,b\in F$. Since $F$ is an filter and $a,b\in F$ this implies that $h(a)\sqcap h(b)=h(a\sqcap b)\in h(F)$. Now let $h(a)\sqsubseteq x$ for some $x\in D.$ As $h$ is surjective then there exists a $d\in M$ such that $h(d)=x.$ So $h(a)\sqsubseteq h(d)$. Therefore $a\sqsubseteq d$ as $h$ is quasi-injective. So $d\in F$. Therefore $x=h(d)\in h(F)$. Now if possible let $h(F)=D$ then $\bot\in h(F).$ Therefore there exists $y\in F$ such that $h(y)=\bot=h(\bot).$ So $y\sqsubseteq \bot$ as $h$ is quasi-injective. Hence $\bot\in F $ which is a contradiction. Therefore $h(F)$ is proper filter in $d$. Now let $z\in D$. As $h$ is surjective then there exists $e\in M$ such that $h(e)=z$. As $F$ is primary filter then either $e\in F$ or $\neg e\in F$. Therefore $z=h(e)\in h(F)$ or $\neg z=\neg h(e)=h(\neg e)\in h(F)$. Hence $h(F)$ is primary filter.
\end{proof}

\section{Contexts on topological spaces}
\label{TC}
As mentioned in Section \ref{sec:Introduction}, for proving the topological representation theorem for (fully contextual/pure) dBas, we enhance the standard context defined by Wille by adding topologies on the sets $\mathcal{F}_{pr}(\textbf{D})$ of all  primary filters and  $\mathcal{I}_{pr}(\textbf{D})$ of all primary ideals. The resulting structure is an instance of a  {\it context on topological spaces} (CTS), that we define now.  It will be shown in the sequel that the structure is, in fact, an instance of a  special kind of CTS, denoted as  ``CTSCR" (Definition  \ref{topcon}  below). 
%Then we show that set of clopen protoconcept of CTSCR form dBa. The set of all clopen semiconcepts form a pure dBa. 
%The definition of a  CTSCR (given in Definition  \ref{topcon}  in the sequel), is based on  that of  a {\it context on topological space} CTS (Definition \ref{cntx on top} below). In order to   define a CTSCR, we need to present some notions and related results. 
%It should be mentioned here that in \cite{Hartung1992} also,  a structure called  CTSCR is  used.  
%One can easily verify that our definition is different from  that  discussed in \cite{Hartung1992}; however, we borrow the  name, as it appears apt in our case too.  

%It is based on a {\it CTS}, that is  first defined.
\begin{definition}
\label{cntx on top}
{\rm $\mathbb{K}^{T}:=((G,\rho),(M,\tau),R)$ is called a {\it context on topological spaces} (CTS) if
\begin{enumerate}[{(a)}]
\item $(G,\rho)$ and $(M,\tau)$ are topological spaces, and
\item $\mathbb{K}:=(G,M,R)$ is a context.
\end{enumerate}}  
\end{definition}

\noindent It may be noted that a CTS  is  a generalization of the  topological context   of \cite{Hartung1992}.  
 Utilizing the presence of the topologies in a CTS, special kinds of object oriented protoconcepts and object oriented semiconcepts of a CTS shall be considered now.
\begin{definition}
{\rm A {\it clopen object oriented protoconcept} $(A,B)$  of a CTS $\mathbb{K}^{T}$ is an object oriented protoconcept of $\mathbb{K}$ such that $A$ is clopen (closed and open) in $(G,\rho)$ and $B$ is clopen in $(M,\tau)$. 
\vskip 2pt \noindent The set of all clopen object oriented protoconcepts of $\mathbb{K}^{T}$ is denoted by $\mathfrak{R}^{T}(\mathbb{K}^{T})$.}
\end{definition}
\noindent Similarly,
\begin{definition}
{\rm A {\it clopen object oriented semiconcept} $(A,B)$  of $\mathbb{K}^{T}$ is an object oriented semiconcept of $\mathbb{K}$ such that $A$ is clopen in $(G,\rho)$ and $B$ is clopen in $(M,\tau)$.
\vskip 2pt \noindent 
$\mathfrak{S}^{T}(\mathbb{K}^{T})$ denotes the set of all clopen object oriented semiconcepts of $\mathbb{K}^{T}$.}
\end{definition}

Let us give an example.

\begin{example} \label{exampleCTS} {\rm Consider the CTS $\mathbb{K}^{T}:=((G,\tau),(M,\rho),R),$ where  $G:=\{a,b,c,d,e\}$, $M:=\{1,2,3,4\}$, $\tau:=\{\{a,b,c\},\{d,e\},G,\emptyset\}$, $\rho:=\{\{2\},\{1,3,4\},M,\emptyset\}$ and  $R$:=$\{(a,2),(b,2),(c,4),(d,1),(d,4),(e,1), \\(e,3)\}$.
 $(\{a,b,c\},\{2\})$ is a clopen object oriented semiconcept (hence also an object oriented protoconcept) of $\mathbb{K}^{T}$. Note that $\{a,b,c\}^{\blacksquare} = \{2\}$.  }  
\end{example}
\begin{observation}
\label{psocsemicocept}
{\rm For the subset  $A:=\{a,b\}$ of $G$ in Example \ref{exampleCTS}, $(A, A^{\blacksquare})$ is an object oriented semiconcept of $\mathbb{K}:=(G,M,R)$, but not a clopen object oriented semiconcept of the CTS $\mathbb{K}^{T}$. %So we can say in general that $\mathfrak{S}^{T}(\mathbb{K}^{T})$ is a proper subset of $\mathfrak{S}(\mathbb{K})$, and hence $\mathfrak{R}^{T}(\mathbb{K}^{T})$ is a proper subset of $\mathfrak{R}(\mathbb{K})$. 
}
\end{observation}
%\begin{lemma}[***]
%\label{Aj}
%{\rm Let $\mathbb{K}^{T}:=((G,\rho),(M,\tau),R)$ be CTS. Then $\mathfrak{R}^{T}(\mathbb{K}^{T})_{p}=\mathfrak{S}^{T}(\mathbb{K}^{T})$.}
%\end{lemma}
%\begin{proof}
 %Let $(A,B)\in \mathfrak{R}^{T}(\mathbb{K}^{T}_{pr}(\textbf{D}))_{p}$. Then either $(A,B)=(A,B)_{\sqcap}=(A,A^{\blacksquare})$ or $(A,B)=(A,B)_{\sqcup}=(B^{\lozenge}, B)$. In both the cases $(A,B)\in \mathfrak{S}^{T}(\mathbb{K}^{T}_{pr}(\textbf{D}))$. Conversely let $(A,B)\in\mathfrak{S}^{T}(\mathbb{K}^{T}_{pr}(\textbf{D}))$ then either $(A,B)=(A,A^{\blacksquare})$ or $(A,B)=(B^{\lozenge}, B)$, which implies that $(A,B)\in \mathfrak{R}^{T}(\mathbb{K}^{T}_{pr}(\textbf{D}))_{p}$. Hence $\mathfrak{R}^{T}(\mathbb{K}^{T})_{p}=\mathfrak{S}^{T}(\mathbb{K}^{T})$.
%\end{proof}
Different kinds of {\it homomorphisms of CTS} are defined as follows.
%**Replace "context" by "contexts" whenever any kind of homomorphisms are mentioned**
\begin{definition}
\label{hom of ctx on top}
{\rm Let $\mathbb{K}^{T}_{1}:=((X_{1},\tau_1),(Y_{1},\rho_1),R_{1})$ and $\mathbb{K}^{T}_{2}:=((X_{2},\tau_{2}),(Y_{2},\rho_{2}),R_{2})$ be two CTS. A {\it  CTS-homomorphism} $(\alpha,\beta):\mathbb{K}^{T}_{1}\rightarrow \mathbb{K}^{T}_{2}$ consists of a pair of maps such that
\begin{enumerate}[(a)]
 \item $(\alpha,\beta):\mathbb{K}_{1}\rightarrow \mathbb{K}_{2}$, is a context homomorphism, and 
\item  $\alpha:X_{1}\rightarrow X_{2}$ and $\beta:Y_{1}\rightarrow Y_{2}$  are  continuous functions.
\end{enumerate}
\noindent Denote $f:= (\alpha,\beta)$. If $\alpha$ and $\beta$ are injective, the homomorphism $f:\mathbb{K}^{T}_{1}\rightarrow \mathbb{K}^{T}_{2}$    is called a {\it CTS-embedding}.\\ A  CTS-embedding $f$ is called a {\it CTS-isomorphism}  if $\alpha$ and $\beta$ are  surjective. \\ If $f:\mathbb{K}^{T}_{1}\rightarrow \mathbb{K}^{T}_{2}$ is a CTS-isomorphism  and $\alpha$, $\beta$ are homeomorphisms, $f$ is called a {\it CTS-homeomorphism}. We say $\mathbb{K}^{T}_{1}$  {\it is homeomorphic to} $\mathbb{K}^{T}_{2}$.}
\end{definition}
\noindent It is easy to see the following.
        \begin{proposition}
        \label{invhomeo}
        {\rm If $(\alpha,\beta):\mathbb{K}^{T}_{1}\rightarrow \mathbb{K}^{T}_{2}$ is a CTS-homeomorphism then  $(\alpha^{-1},\beta^{-1}):\mathbb{K}^{T}_{2}\rightarrow \mathbb{K}^{T}_{1}$ is also  a CTS-homeomorphism.}
        \end{proposition} 
       \noindent  $(\alpha^{-1},\beta^{-1})$ in Proposition \ref{invhomeo} is  called  the {\it inverse} of $(\alpha,\beta)$.
        \vskip 2pt
        Let $(X,\rho)$ and $(Y,\tau)$ be two topological spaces,  and $R$ be a binary relation between $X$ and $Y$. In \cite{berge1997topological, guide2006infinite}, a relation  $R$ is called  a {\it many-valued mapping } or {\it correspondence} from $X$ into $Y$, as $R$ maps each $x\in X$ to a subset of $Y$. $X^{*}:=\{x\in X: R(x)\neq\emptyset\}$ is called the {\it domain} of the mapping and $Y^{*}:=\cup_{x\in X}R(x)$, the {\it range} or set of values of $R$. The {\it lower inverse} for $R$ of a subset $B$  of $Y$  is defined by $R^{-}(B):=\{x\in X: R(x)\cap B\neq\emptyset\}$. On the other hand, its  {\it upper inverse} is defined by $R^{+}(B):=\{x\in X: R(x)\subseteq B\}.$ Next we recall the definition of {\it continuity} of a many-valued mapping (relation) \cite{guide2006infinite}.
    In the following, let $(X,\rho)$  and $(Y,\tau)$ be two topological spaces and $R$ be a binary relation between $X$ and $Y$.
       
        \begin{definition}
\label{continuty def}
{\rm \cite{guide2006infinite} 
%Let $(X,\rho)$ and $(Y,\tau)$ be two topological spaces  and $R$ be a binary relation between $X$ and $Y$. 
For any $x_{0}\in X$, a relation $R$ is {\it lower semi-continuous at} $x_{0}$ if for each open set $O$ in $(Y,\tau)$ with  $R(x_{0})\cap O\neq\emptyset,$ there exists a neighbourhood $U(x_{0})$ in $(X,\rho)$ such that \[x\in U(x_{0})~\mbox{implies that}~ R(x)\cap O\neq\emptyset.\]

\noindent $R$ is {\it upper semi-continuous at} $x_{0} \in X$, if for each open set $O$ in $(Y,\tau)$ containing $R(x_{0})$ there is a neighbourhood $U(x_{0})$  in $(X,\rho)$ such that \[x\in U(x_{0})~\mbox{implies that}~ R(x)\subseteq O.\]

\noindent $R$ is {\it continuous at} $x_{0} \in X$, if it is both lower and upper semi-continuous at $x_{0}$. 
%\end{definition}
%\begin{definition}
%{\rm \cite{berge1997topological,guide2006infinite} Let $(X,\rho)$ and $(Y,\tau)$ be two topological spaces  and $R$ be a binary relation between $X$ and $Y$. 
$R$ is {\it lower semi-continuous in $X$} if $R$ is lower semi-continuous at each point of $X$. 
$R$ is {\it upper semi-continuous in $X$} if $R$ is upper semi-continuous at each point of $X$, and  $R$ is {\it continuous} in $X$ if $R$ is both lower and upper semi-continuous in $X$.}
\end{definition}

\begin{observation}
\label{luinv-img}
{\rm Let us recall the definitions of $\square,\lozenge,\blacksquare,\blacklozenge$ given in Section \ref{dba} and the definitions of upper and lower inverse of a relation $R$ given above. 
\begin{enumerate}[{(i)}]
    \item It can be seen that $B^{\lozenge}=R^{-}(B), B^{\square}=R^{+}(B)$ and $A^{\blacklozenge}=(R^{-1})^{-}(A),A^{\blacksquare}=(R^{-1})^{+}(A).$
\item If $R$ is a function, $B^{\lozenge}=B^{\square}=R^{-1}(B)$. Moreover, if R is a bijection, $A^{\blacklozenge}=A^{\blacksquare}=R(A)$.
\end{enumerate}}
\end{observation}

\noindent We now  define a CTSCR.

\begin{definition}\label{topcon}
{\rm A   {\it CTSCR} is a CTS $\mathbb{K}^{T}:=((G,\rho),(M,\tau),R)$  where  $R$ and $R^{-1}$  are continuous in $G$ and $M$ respectively.}
\end{definition}

\noindent It will be shown in Theorem \ref{clopen dba} below that the set of all clopen object oriented protoconcepts of a CTSCR forms a fully contextual dBa, while the set of all clopen object oriented semiconcepts forms a pure dBa. 
%The representation theorem presented in the next section will establish that every pure dBa is isomorphic to the algebra of clopen object oriented semiconcepts of some CTSCR. 

\vskip 2pt 
We shall  give an example of a CTSCR, but before that let us note the following results which will be useful while demonstrating semi-continuity and continuity of relations and their converses.
\begin{theorem}
{\rm \cite{guide2006infinite} 
\bit  \item[{\bf I.}] The following are equivalent. 
\label{upper semi-continuty}
\begin{enumerate}[{(i)}]
\item $R$ is upper semi-continuous in $X$. 
\item For each open set $O$ in $(Y,\tau)$,  $O^{\square}$ is open in $(X,\rho)$.
\item For each closed set $A$ in $(Y,\tau)$,    $A^{\lozenge}$ is closed in $(X,\rho)$.
\end{enumerate}
\item[{\bf II.}] The following are equivalent. 
\label{lower semi-continuty}
\begin{enumerate}[{(i)}]
\item $R$ is lower semi-continuous in $X$.
\item For each open set $O$ in $(Y,\tau)$,  $O^{\lozenge}$ is open in $(X,\rho)$.
\item  For each closed set $B$ in $(Y,\tau)$,  $B^{\square}$ is closed in $(X,\rho)$.
\end{enumerate}
\eit}
\end{theorem}

%\begin{proof}
%Let $R$ be upper semi-continuous and $O$ be an open set in $(Y,\tau)$. Let $x\in O^{\square}$  that is $R(x)\subseteq O$ then there exists  a open neighbourhood $U(x)$  of $x$ such that $U(x)\subseteq O^{\square}$ and hence $O^{\square}$ is open in $(X,\rho)$. Now we assume that for each open set $O$ in $(Y,\tau)$ $O^{\square}$ is open in  $(X,\rho)$. Let $x_{0}\in X$ such that open set $O$ contain $R(x_{0})$  then $O^{\square}$ is open neighbourhood of such that $x\in O^{\square}\implies R(x)\subseteq O$ and hence $R$ is upper continuous. Equivalence of $2$ and $3$ follows from $3$ of Theorem \ref{property of box}
%\end{proof}

\begin{corollary}
\label{continuty of R}
{\rm The following are equivalent.
\begin{enumerate}[{(i)}]
\item $R$ is continuous in $X$.
\item if $B$ is open in $(Y,\tau)$ then both $B^{\lozenge}$ and $B^{\square}$  are  open in $(X,\rho)$.
\item if $B$ is closed in $(Y,\tau)$ then both $B^{\lozenge}$ and $B^{\square}$  are  closed in $(X,\rho)$.
\end{enumerate}}

\end{corollary}
%\begin{proof} 
%Theorem \ref{lower semi-continuty} gives the proof. 
%\end{proof}
%**removed proof of corollary\\
\noindent For the converse $R^{-1}$ of  $R$, one can similarly derive 
% the following results.
\begin{theorem}
\label{conv upper semi-continuty}
{\rm \noindent \bit \item[{\bf I.}] The following are equivalent. 
\begin{enumerate}[{(i)}]
\item $R^{-1}$ is lower semi-continuous in $Y$.
 \item For each open set $O$ in $(X,\rho)$, $O^{\blacklozenge}$ is open in $(Y,\tau)$.
 \item For each closed set $O$ in $(X,\rho)$, $O^{\blacksquare}$ is closed in $(Y,\tau)$.
 \end{enumerate}
\item[{\bf II.}]  The following are equivalent. 
\begin{enumerate}[{(i)}]
\item $R^{-1}$ is upper semi-continuous in $Y$. 
\item For each open set $O$ in $(X,\rho)$, $O^{\blacksquare}$ is open in $(Y,\tau)$.
\item For each closed set $A$ in $(X,\rho)$, $A^{\blacklozenge}$ is closed in $(Y,\tau)$.
\end{enumerate}
\eit}
\end{theorem}

\begin{corollary}
\label{convers continuty}
{\rm  The following are equivalent. 
\begin{enumerate}[{(i)}]
\item $R^{-1}$ is continuous in $Y$.
\item if $A$ is open in $(X,\rho)$ then both $A^{\blacklozenge}$ and $A^{\blacksquare}$  are  open in $(Y,\tau)$.
\item if $A$ is closed in $(X,\rho)$ then both $A^{\blacklozenge}$ and $A^{\blacksquare}$ are  closed in $(Y,\tau)$.
\end{enumerate}} 
\end{corollary}

Now let us give an example of a CTSCR.% *along with clopen object oriented protoconcept and semiconcept in it.*
\begin{example}
\label{exmple of topocxt}
{\rm Let $(X,\tau_{1})$ and $(Y,\tau_{2})$ be two non-empty totally disconnected spaces, and $C$ be a fixed non-empty clopen subset of $(Y,\tau_{2}).$  Then $\mathbb{K}_{+}^{T}:=((X,\tau_{1}),(Y,\tau_{2}),R),$ where $xRy$ if and only if $y\in C$, is a CTSCR.} 
\end{example}
\begin{proof}
%To show $\mathbb{K}_{+}^{T}$ is a CTSCR, 
One needs to show that $R$ and $R^{-1}$ are   continuous. For establishing continuity of  $R$, utilizing Corollary \ref{continuty of R}, we just show that for any open set $O$ in $(Y,\tau_{2})$, $O^{\square},O^{\lozenge}$ are open in $(X,\tau_{1})$. Observe that for any $B\subseteq Y$,
\[B^{\square} = \begin{cases} X &\mbox{if } C\subseteq B\\
\emptyset & \mbox{if } C\nsubseteq B \end{cases} ~\mbox{and}~ B^{\lozenge} = \begin{cases} X &\mbox{if } B\nsubseteq C^{c}\\
\emptyset & \mbox{if } B\subseteq C^{c} \end{cases} \] 
Therefore for any open set $O$ in $(Y,\tau_{2})$, $O^{\square}$ and $O^{\lozenge}$ are always open  in $(X,\tau_{1})$ and hence $R$ is continuous. Now consider the case of $R^{-1}$. From the definition of $R$ it is clear that \[R^{-1}(y)=\begin{cases}\emptyset &\mbox{if} ~y\notin C\\X &\mbox{if} ~y\in C\end{cases}\] 
So for any $A\subseteq X$,\[A^{\blacksquare}=\begin{cases}C^{c}&\mbox{if}~A\subset X\\Y & \mbox{if}~A=X\end{cases} ~\mbox{and}~ A^{\blacklozenge}=C\]  
Thus for any open set $O$ of $(X,\tau_{1})$,   $O^{\blacksquare},  O^{\blacklozenge}$ are open in $(Y,\tau_{2})$, making $R^{-1}$ continuous by Corollary \ref{convers continuty}.
\end{proof}

For CTSCRs, we rename the maps given in Definition \ref{hom of ctx on top}.
 \begin{definition}
 \label{topocxthomeo}
  {\rm 
 In case of CTSCRs, a CTS-homomorphism, CTS-embedding, CTS-isomorphism and CTS-homeomorphism are respectively called  {\it CTSCR-homomorphism}, {\it CTSCR-embedding}, {\it CTSCR-isomorphism},
and  {\it CTSCR-homeomorphism}.}

\end{definition}

\begin{theorem}
\label{clopen dba}
{\rm Let $\mathbb{K}^{T} :=((X,\tau_{1}),(Y,\tau_{2}),R)$ be a  CTSCR. 
%Then the following hold.
\begin{enumerate}[{(i)}]
\item $\underline{\mathfrak{R}}^{T}(\mathbb{K}^{T}):=(\mathfrak{R}^{T}(\mathbb{K}^{T}),\sqcup,\sqcap,\neg,\lrcorner,\top,\bot)$ is a fully contextual dBa. 
\item $\mathcal{S}^{T}(\mathbb{K}^{T}):=(\mathfrak{S}^{T}(\mathbb{K}^{T}),\sqcup,\sqcap,\neg,\lrcorner,\top,\bot)$ is a pure dBa. 
\end{enumerate}}
\end{theorem}
\begin{proof}
(i) Let $(A,B)$ and $(C,D)$ belong to $\mathfrak{R}^{T}(\mathbb{K}^{T}).$ Then $(A,B)\sqcap (C,D)=(A\cup C, (A\cup C)^{\blacksquare}), \neg (A,B)=(A^{c},A^{c\blacksquare})$ and $(A,B)\sqcup (C,D)=((B\cap D)^{\lozenge}, B\cap D ), \lrcorner (A,B)=(B^{c\lozenge}, B^{c})$. Since $A,C$ are clopen in $(X,\tau_{1})$, $A\cup C$ and $A^{c}$ are also clopen in $(X,\tau_{1})$. Similarly $B,D$ are clopen in $(Y,\tau_{2})$ implies that $B\cap D$ and $B^{c}$ are clopen in $(Y,\tau_{2})$. Since $R^{-1}$ is continuous, %and $A\cup C, A^{c}$ are closed and open 
by Corollary \ref{convers continuty}, $ (A\cup C)^{\blacksquare}$ and $A^{c\blacksquare}$ are closed and open in $(Y,\tau_{2})$. Similarly, continuity of $R$ implies that $(B\cap D)^{\lozenge}$ and $B^{c\lozenge}$ are clopen in $(X,\tau_{1})$, by Corollary \ref{continuty of R}. Therefore $\mathfrak{R}^{T}(\mathbb{K}^{T})$  is closed under $\sqcup,\sqcap,\neg,\lrcorner$. It is clear that $\top=(\emptyset,\emptyset)$ and $\bot=(G,M)$ both belong to $\mathfrak{R}^{T}(\mathbb{K}^{T})$. It is routine to verify that the set of  all clopen object oriented protoconcepts satisfies the axioms of dBa with respect to $\sqcup,\sqcap,\neg,\lrcorner,\top,\bot$. Since $\mathfrak{R}^{T}(\mathbb{K}^{T})\subseteq \mathfrak{R}(\mathbb{K})$,  by Proposition \ref{order object-proto}, $\sqsubseteq $ is a partial order on $\mathfrak{R}^{T}(\mathbb{K}^{T}).$ Therefore $\underline{\mathfrak{R}}^{T}(\mathbb{K}^{T})$ is a contextual dBa. Next we  show that $\underline{\mathfrak{R}}^{T}(\mathbb{K}^{T})$ is fully contextual. Let $(A,B), (C,D)\in \mathfrak{R}^{T}(\mathbb{K}^{T}) $  such that $(A,B)=(A,B)_{\sqcap}=(A, A^{\blacksquare})$ and $(C,D)=(C,D)_{\sqcup}=(D^{\lozenge}, D)$. Let us also assume that $(A,B)_{\sqcup}=(C,D)_{\sqcap}$.  Then $(A^{\blacksquare\lozenge}, A^{\blacksquare})=(D^{\lozenge}, D^{\lozenge\blacksquare})$, which implies that $A^{\blacksquare\lozenge}=D^{\lozenge}$. So $(A,D)\in \mathfrak{R}^{T}(\mathbb{K}^{T})$ such that $(A, D)_{\sqcap}=(A,B)$ and $(A,D)_{\sqcup}=(C,D)$ and it is clearly unique. \\
%Hence $\underline{\mathfrak{R}}^{T}(\mathbb{K}^{T})$ is a fully contextual dBa. \\
(ii) Similar to the proof in (i), it can be shown that $\mathfrak{S}^{T}(\mathbb{K}^{T})$ is closed with respect to the operations $\sqcup,\sqcap,\neg,\lrcorner,\top,\bot$ and forms a pure dBa. 
\end{proof}
\begin{corollary}
\label{ScxttoB}
{\rm Let $\mathbb{K}^{T}:=((X,\tau_{1}),(Y, \tau_{2}), R)$ be a CTSCR such that $R$ is a  homeomorphism from $(X,\tau_{1})$ to $(Y, \tau_{2})$. Then $\underline{\mathfrak{R}}^{T}(\mathbb{K}^{T})=\mathcal{S}^{T}(\mathbb{K}^{T})$ is a Boolean algebra.}
\end{corollary}
\begin{proof}
Since $R$ is a bijection, it follows from Observation \ref{luinv-img}(ii) that for all $A\subseteq X$ and $B\subseteq Y$, $A^{\blacksquare}=R(A)$ and $B^{\lozenge}=R^{-1}(B)$, respectively. Let $(A,B)\in \mathfrak{R}^{T}(\mathbb{K}^{T})$. Then $A^{\blacksquare\lozenge}=B^{\lozenge}$, which is equivalent to $R^{-1}(R(A))=R^{-1}(B)$. As $R$ is a bijection, $R(A)=B$, which implies that $(A,R(A))=(A, A^{\blacksquare})=(A,B)\in \mathfrak{S}^{T}(\mathbb{K}^{T})$. So $\mathfrak{S}^{T}(\mathbb{K}^{T})=\mathfrak{R}^{T}(\mathbb{K}^{T})$. From Theorem \ref{clopen dba} it follows that $\underline{\mathfrak{R}}^{T}(\mathbb{K}^{T})=\mathcal{S}^{T}(\mathbb{K}^{T})$ is a fully contextual pure dBa. Let $(A,B)\in \mathfrak{R}^{T}(\mathbb{K}^{T})$. As  $R$ and $R^{-1}$ are bijections, we get the following equalities. $\neg (A,B)=(A^{c}, R(A^{c}))=(A^{c},(R(A))^{c})=((R^{-1}(B))^{c}, B^{c})=(R^{-1}(B^{c}),B^{c})=\lrcorner (A, B)$ and $\neg\neg(A,B)=(A,B)$. By Theorem \ref{Boolean and dBa}, $\underline{\mathfrak{R}}^{T}(\mathbb{K}^{T})$ is a Boolean algebra.
\end{proof}

\noindent 
%In Theorem \ref{clopen dba}, it was proved that for each CTSCR $\mathbb{K}^{T}$, $\underline{\mathfrak{R}}^{T}(\mathbb{K}^{T})$ is a dBa while $\mathcal{S}^{T}(\mathbb{K}^{T})$ is a pure dBa. 
Theorem \ref{injecmor} below gives us a relation between  the algebras $\underline{\mathfrak{R}}^{T}(\mathbb{K}_{1}^{T}),\underline{\mathfrak{R}}^{T}(\mathbb{K}_{2}^{T})$ and $\mathcal{S}^{T}(\mathbb{K}_{1}^{T}),\mathcal{S}^{T}(\mathbb{K}_{2}^{T})$ when  the CTSCRs $\mathbb{K}^{T}_{1}, \mathbb{K}^{T}_{2}$  are  homeomorphic. For that, we use the following results about isomorphic contexts. 
%We shall use these  to prove Theorem \ref{pdba hom set} leading from surjective dBa homomorphisms to .** 

  \begin{proposition}
\label{obe1}
{\rm Let $\mathbb{K}_{1}:=(X_{1},Y_{1},R_{1}),\mathbb{K}_{2}:=(X_{2},Y_{2},R_{2})$ be two contexts and $(\alpha,\beta)$  a context isomorphism from $ \mathbb{K}_{1}$ to $\mathbb{K}_{2}$. Then the following hold.
\begin{enumerate}[{(i)}]
\item For $A\subseteq X_1$, $\alpha(A)^\blacksquare=\beta(A^\blacksquare)$ and $\beta(A^\blacklozenge)=\alpha(A)^\blacklozenge.$
\item For $B\subseteq Y_1$, $\alpha(B^\square)=\beta(B)^\square$ and $\beta(B)^\lozenge=\alpha(B^\lozenge).$
\end{enumerate}}
\end{proposition}
\begin{proof} (i) Let $A\subseteq X_1$. We will show that $\alpha(A)^\blacksquare=\beta(A^\blacksquare)$. Let $a\in\alpha(A)^\blacksquare $, $x\in X_1$ and $xR_{1}\beta^{-1}(a)$. Then $\alpha(x)R_{2} a$, as $f$ is a context isomorphism. Therefore $\alpha(x)\in \alpha(A).$ So $x\in A$, as $\alpha$ is a bijection. This implies that $\beta^{-1}(a)\in A^{\blacksquare}.$ Hence $a\in \beta(A^\blacksquare)$. For the other direction, let $b\in\beta(A^\blacksquare).$ Then $\beta^{-1}(b)\in A^{\blacksquare}$. Let $x^\prime\in X_{2}$ and $x^\prime R_{2} b.$ This means $\alpha^{-1}(x^\prime)R_{1}\beta^{-1}(b)$, as $f$ is a context isomorphism. So $\alpha^{-1}(x^\prime)\in A,$ and $x^\prime\in \alpha(A)$, which implies that $b\in \alpha(A)^{\blacksquare}$. 

\noindent Next let us show that $\beta(A^\blacklozenge)=\alpha(A)^\blacklozenge$. Let $y\in \beta(A^\blacklozenge).$ Then $\beta^{-1}(y)\in A^{\blacklozenge}$, which implies that $R_{1}^{-1}(\beta^{-1}(y))\cap A\neq\emptyset$. So there exists $x\in A$ such that $xR_{1}\beta^{-1}(y)$, which gives $\alpha(x)R_{2}y$, as $f$ is a context isomorphism. This implies that $R^{-1}_{2}(y)\cap \alpha(A)\neq\emptyset$ and hence $y\in \alpha(A)^\blacklozenge$. On the other hand, let $y_{0}\in \alpha(A)^\blacklozenge.$ Then $R^{-1}_{2}(y_{0})\cap \alpha(A)\neq \emptyset.$  So there exists $\alpha(x)\in \alpha(A)$ such that $\alpha(x)R_{2}y_{0}$. As $f$ is a context isomorphism, $xR_{1}\beta^{-1}(y_{0}).$ This implies that $R_{1}^{-1}(\beta^{-1}(y_{0}))\cap A\neq \emptyset$. So $\beta^{-1}(y_{0})\in A^{\blacklozenge}$, whence $y_{0}\in \beta(A^{\blacklozenge}).$\\
(ii) This can be proved similarly.
%\noindent Proof 2: Let $B\subseteq Y$ and we will show that $\alpha(B^\square)=\beta(B)^\square$. Let $a\in \alpha(B^\square)$ and $aR_{2}m$ for some $m\in Y$. Then $\alpha^{-1}(a)\in B^\square$ as $a\in \alpha(B^\square)$ and $\alpha^{-1}(a)R_{1}\beta^{-1}(m)$ since $f=(\alpha,\beta)$ is a context isomorphism. Therefore $\beta^{-1}(m)\in B$, which implies that $m\in \beta(B).$ Hence $a\in \beta(B)^{\square}$. Now let $a^{\prime}\in \beta(B)^{\square}$ and for some $m\in Y$ $\alpha^{-1}(a^{\prime})Rm.$ Then $a^{\prime}R_{2}\beta(m).$  So $\beta(m)\in \beta(B)$, which implies that $m^{\prime}\in B.$ Therefore $\alpha^{-1}(a^{\prime})\in B^{\square}$, which implies that $a^{\prime}\in \alpha(B^{\square}).$  Hence $\alpha(B^\square)=\beta(B)^\square$.\noindent Now we will show that $\beta(B)^\lozenge=\alpha(B^\lozenge)$. Let $x\in \beta(B)^{\lozenge}.$ Then $R_{2}(x)\cap\beta(B)\neq\emptyset$. Therefore there exists $y\in B$ such that $xR_{2}\beta(y).$ As $f=(\alpha,\beta)$ is a context isomorphism we have $\alpha^{-1}(x)Ry$, which implies that $R_{1}(\alpha^{-1}(x))\cap B\neq\emptyset$. Therefore $\alpha^{-1}(x)\in B^{\lozenge}$, which implies that $x\in \alpha(B^{\lozenge})$. Now we assume that $x\in \alpha(B^{\lozenge})$ then $\alpha^{-1}(x)\in B^{\lozenge}$ and so $R_{1}(\alpha^{-1}(x))\cap B\neq\emptyset$. Therefore there exists $y\in B$ such that $\alpha^{-1}(x)R_{1} y$ and so $xR_{2}\beta(y)$. Therefore $R_{2}(x)\cap\beta(B)\neq\emptyset$, which implies that $x\in \beta(B)^{\lozenge}$. Hence $\beta(B)^\lozenge=\alpha(B^\lozenge)$
\end{proof}
%\noindent Now we show that there is a homeomorphism between two CTS and one is a CTSCR then so is the other.
 %In topology a property is said to be ``topologically invariant'',
%called topological property or topologically invariant 
%if it is invariant under homeomorphisms. 
 
\begin{definition}
\label{topoisocxt}
{\rm Let  $\mathbb{K}^{T}_{1}:=((X_{1},\tau),(Y_{1},\rho),R_{1})$ and  $\mathbb{K}^{T}_{2}:=((X_{2},\tau_{1}),(Y_{2},\rho_{1}),R_{2})$ be two CTSCRs and $f:=(\alpha,\beta)$ be a CTSCR-homeomorphism from $\mathbb{K}^{T}_{1}$ to $\mathbb{K}^{T}_{2}$. The map $f_{\alpha\beta}:\mathfrak{R}^{T}(\mathbb{K}^{T}_{2})\rightarrow\mathfrak{R}^{T}(\mathbb{K}^{T}_{1})$ is defined as $f_{\alpha\beta}((A,B)):=(\alpha^{-1}(A),\beta^{-1}(B)) $ 
for all $(A,B)\in\mathfrak{R}^{T}(\mathbb{K}^{T}_{2})$.} 
\end{definition}

\begin{theorem}
\label{injecmor}
{\rm If $f:=(\alpha,\beta)$ is a CTSCR-homeomorphism then  $f_{\alpha\beta}$ is a dBa isomorphism from $\underline{\mathfrak{R}}^{T}(\mathbb{K}^{T}_{2})$  to $\underline{\mathfrak{R}}^{T}(\mathbb{K}^{T}_{1})$. Moreover, the restriction of  $f_{\alpha\beta}$ on $\mathcal{S}^{T}(\mathbb{K}^{T}_{2})$ gives a dBa isomorphism from $\mathcal{S}^{T}(\mathbb{K}^{T}_{2})$ to $\mathcal{S}^{T}(\mathbb{K}^{T}_{1})$.
} 
\end{theorem}
\begin{proof}
As $f:=(\alpha,\beta)$ is a CTSCR-isomorphism from   
$\mathbb{K}^{T}_{1}$ to $\mathbb{K}^{T}_{2}$, $f$ is a context isomorphism and $\alpha, \beta$ are continuous  functions on $X_{1},X_{2}$ respectively.   Let $(A,B)\in\mathfrak{R}^{T}(\mathbb{K}^{T}_{2}).$ Then $A^{\blacksquare\lozenge}=B^{\lozenge}$. By Proposition \ref{contx inv iso}, $(\alpha^{-1},\beta^{-1})$ is a context isomorphism. So, using Proposition \ref{obe1}, $\alpha^{-1}(A)^{\blacksquare\lozenge}=\beta^{-1}(A^{\blacksquare})^{\lozenge}=\alpha^{-1}(A^{\blacksquare\lozenge})=\alpha^{-1}(B^{\lozenge})=\beta^{-1}(B)^{\lozenge}.$ Therefore $(\alpha^{-1}(A),\beta^{-1}(B))$ is an object oriented protoconcept of the context $\mathbb{K}_{1}$ and $(\alpha^{-1}(A),\beta^{-1}(B))\in \mathfrak{R}^{T}(\mathbb{K}^{T}_{1})$,  as $\alpha,\beta$ are continuous. Hence $f_{\alpha\beta}$ is a well-defined map. Let $x:=(A_{1},B_{1}),y:=(A_{2},B_{2})\in\mathfrak{R}^{T}(\mathbb{K}^{T}_{2}) $. Using Proposition \ref{obe1} we get, 
\begin{align*}
f_{\alpha\beta}(x\sqcap y)=&(\alpha^{-1}(A_{1}\cup A_{2}),\beta^{-1}((A_{1}\cup A_{2})^{\blacksquare}))\\
=&(\alpha^{-1}(A_{1}\cup A_{2}),\alpha^{-1}(A_{1}\cup A_{2})^{\blacksquare})\\
=&(\alpha^{-1}(A_{1})\cup\alpha^{-1}(A_{2}),(\alpha^{-1}(A_{1})\cup\alpha^{-1}(A_{2}))^{\blacksquare})\\
=&f_{\alpha\beta}(x)\sqcap f_{\alpha\beta}(y).
\end{align*}
 Dually, $f_{\alpha\beta}(x\sqcup y)=f_{\alpha\beta}(x)\sqcup f_{\alpha\beta}(y)$. $f_{\alpha\beta}(\lrcorner x)=(\alpha^{-1}(B_{1}^{c\lozenge}),\beta^{-1}(B_{1}^{c}))=(\beta^{-1}(B_{1}^{c})^{\lozenge},\beta^{-1}(B_{1}^{c}))=\lrcorner f_{\alpha\beta}(x)$. Similarly, $f_{\alpha\beta}(\neg x)=\neg f_{\alpha\beta}(x)$. Since $\alpha $ and $\beta$ are bijective,  $f_{\alpha\beta}(\top)=\top$ and $f_{\alpha\beta}(\bot)=\bot$. Therefore $f_{\alpha\beta}$ is a dBa homomorphism. If $f_{\alpha\beta}(x)=f_{\alpha\beta}(y)$ then $\alpha^{-1}(A_{1})=\alpha^{-1}(A_{2})$ and $\beta^{-1}(B_{1})=\beta^{-1}(B_{2})$. Therefore  $A_{1}=A_{2}$ and $B_{1}=B_{2}$, as $\alpha,\beta$ are bijections, which implies that $x=y.$ So $f_{\alpha\beta}$ is an injective dBa homomorphism.\\
Let $(A,B)\in\mathfrak{R}^{T}(\mathbb{K}^{T}_{1}).$ Then $A^{\blacksquare\lozenge}=B^{\lozenge}$. Therefore using  Proposition \ref{obe1}, $\alpha(A)^{\blacksquare\lozenge}=\beta(B)^{\lozenge}$. Since $\alpha,\beta$  are homeomorphisms, $(\alpha(A),\beta(B))\in\mathfrak{R}^{T}(\mathbb{K}^{T}_{2}) $. So $f_{\alpha\beta}((\alpha(A),\beta(B)))=(\alpha^{-1}(\alpha(A)),\beta^{-1}(\beta(B)))=(A,B)$, as $\alpha,\beta$  are bijective. Hence $f_{\alpha\beta}$ is a dBa isomorphism. \\
Second part of the theorem follows from Theorem \ref{iso}.
\end{proof} 

In the next theorem, we show that the property of being a CTSCR is invariant under CTS-homeomorphisms. 
%Analogous to the definition of topological property we can say that topologcal context is a property for cts.**
\begin{theorem}
\label{compatibility}
{\rm Let $\mathbb{K}^{T}_{1}:=((X_{1},\tau_{1}),(Y_{1},\rho_1),R_{1})$ and  $\mathbb{K}^{T}_{2}:=((X_{2},\tau_{2}),(Y_{2},\rho_{2}),R_{2})$ be two CTS such that $\mathbb{K}^{T}_{1}$ is homeomorphic to $\mathbb{K}^{T}_{2}$.  Then $\mathbb{K}^{T}_{1}$  is a CTSCR if and only if $\mathbb{K}^{T}_{2}$ is a CTSCR.}
\end{theorem}
\begin{proof}
Let $f:=(\alpha,\beta)$ from $\mathbb{K}^{T}_{1}$ to $\mathbb{K}^{T}_{2}$ be a CTS-homeomorphism, and suppose that  $\mathbb{K}^{T}_{1}$  is a CTSCR. 
%We will show that $\mathbb{K}^{T}_{2}$ is also CTSCR. For this 
We show that $R_{2}$  and $R^{-1}_{2},$  are continuous. %Now we show that $R_{2}$ is continuous.  
Let $O$ be an open set in $(Y_{2}, \rho_{2}).$ By Corollary \ref{continuty of R}, it is sufficient to show that   $O^{\square}$ and $O^{\lozenge}$ are open in $(X_{2},\tau_{2})$. $\beta^{-1}(O)$ is open in $(Y_{1},\tau_{1})$, as $\beta$ is continuous and $O$ is open in $(Y_{2},\rho_{2})$. Now continuity of $R_{1}$ implies that $\beta^{-1}(O)^{\square}$ and $\beta^{-1}(O)^{\lozenge}$ are open  in $(X_{1},\tau_{1})$. Since $\alpha $ is a homeomorphism from $(X_{1},\tau_{1})$ to $(X_{2},\tau_{2})$, $\alpha(\beta^{-1}(O)^{\square})$ and $\alpha(\beta^{-1}(O)^{\lozenge})$ are open in $(X_{2},\tau_{2})$. By Proposition \ref{obe1}(ii),  $\alpha(\beta^{-1}(O)^{\lozenge})=\beta(\beta^{-1}(O))^{\lozenge}=O^{\lozenge}$ and $\alpha(\beta^{-1}(O)^{\square})=\beta(\beta^{-1}(O))^{\square}=O^{\square}.$ Hence $O^{\square}$ and $O^{\lozenge}$ are open in $(X_{2},\tau_{2})$.  Thus $R_{2}$ is continuous.
 
The proof of continuity of $R_{2}^{-1}$ is similar,  making use of the continuity of $R_{1}^{-1}$ and the fact that $\beta$ is homeomorphism.
\end{proof}

\noindent It may be remarked that because of Theorem \ref{compatibility} and Theorem \ref{clopen dba}, either one of $ \mathbb{K}^{T}_{1}, \mathbb{K}^{T}_{2}$  in Theorem \ref{injecmor} could have been taken   to be a CTS and the other as a CTSCR.

\vskip 2pt Theorem \ref{clopen dba} establishes that given a CTSCR, one obtains a fully contextual and a pure dBa. What about the converse -- that is, given a (fully contextual/pure) dBa $\textbf{D}$, can one define an ``appropriate''   CTSCR corresponding to $\textbf{D}$?  We address this question in the next section, and  obtain representation results for  dBas in terms of the corresponding  CTSCRs.
 
\section{Representation theorem for  dBas}
\label{RT}
%The converse of Theorem \ref{clopen dba} is partially true that is for each dBa $\textbf{D}$ we construct a CTSCR $\mathbb{K}^{T}_{pr}(\textbf{D})$ such that $\textbf{D}$ is quasi-embedded in the  algebra of the clopen protoconcept of $\mathbb{K}^{T}_{pr}(\textbf{D}).$ For contextual dBa, the quasi- embedding becomes an embedding. In particular, for a pure dBa, we show that the embedding is a dBa isomorphism. First, we construct a CTSCR.
%We now address the question that given a dBa (contextual/pure) whether we can obtained a CTSCR? 

Let $\textbf{D}$ be any dBa. In this section, we shall work with the complement of the standard context $\mathbb{K}(\textbf{D}):=(\mathcal{F}_{p}(\textbf{D}),\mathcal{I}_{p}(\textbf{D}),\Delta),$ namely $\mathbb{K}^{c}(\textbf{D}):=(\mathcal{F}_{p}(\textbf{D}),\mathcal{I}_{p}(\textbf{D}),\nabla),$ where $\nabla$ is the complement of $\Delta$. So for $F\in \mathcal{F}_{p}(\textbf{D})$ and $I\in \mathcal{I}_{p}(\textbf{D})$, $F \nabla I$ if and only in $F\cap I=\emptyset$. By  Proposition \ref{comparison of two ideal} (cf. Section \ref{filterideal}), $\mathbb{K}^{c}(\textbf{D})=(\mathcal{F}_{pr}(\textbf{D}),\mathcal{I}_{pr}(\textbf{D}),\nabla)$. The context $\mathbb{K}^{c}(\textbf{D})$ will henceforth be denoted by $\mathbb{K}_{pr}(\textbf{D})$.
We shall define  a CTS $\mathbb{K}_{pr}^{T}(\textbf{D})$ based on the context $\mathbb{K}_{pr}(\textbf{D})$. 
%corresponding to the dBa \textbf{D} shall be presented. *
%The definition is motivated by the work of Hartung in \cite{Hartung1992}  and Wille in  \cite{wille}.  
\vskip 2pt
\noindent  Recall the  sets $F_{x},I_{x}$ given in Notation \ref{Fx and Ix}. 
 %and **Proposition  \ref{comparison of two ideal}**.
 \begin{proposition}
\label{clopensemiconcept}
{\rm 
%Let $\textbf{D}$ be dBa and $\mathbb{K}_{pr}(\textbf{D})$ be the corresponding context. Then 
For any $x\in D$, the  following hold.
\begin{enumerate}[{(i)}]
\item (a) $F_{x}^{\blacksquare}= I_{\neg x}$, (b)  $F_{x}^{\blacklozenge}= I_{\lrcorner (x\sqcap x)}$.
\item (a) $I_{x}^{\square}= F_{\lrcorner x}$, (b) $I_{x}^{\lozenge}= F_{\neg (x\sqcup x)}$.
\end{enumerate}}
\end{proposition}

\begin{proof}
 
%\begin{align*}
%I_{x}^{c}&=\{I\in \mathcal{I}_{pr}(\textbf{D}):I\notin I_{x} \}\\&=\{I\in \mathcal{I}_{pr}(\textbf{D}):x\notin I \}\\&=\{I\in \mathcal{I}_{pr}(\textbf{D}):\lrcorner x\in I \}\\&=I_{\lrcorner x}
%\end{align*}
% and 
 %\begin{align*}
% F_{x}^{c}&=\{F\in \mathcal{F}_{pr}(\textbf{D}):F\notin F_{x} \}\\&=\{F\in \mathcal{F}_{pr}(\textbf{D}):x\notin F \}\\&=\{F\in \mathcal{F}_{pr}(\textbf{A}):\neg x\in F \}\\&=F_{\neg x}
 %\end{align*}
 
(i)(a). 
%First we will show that $F_{x}^{\blacksquare}=I_{\neg x}$. 
Let $I\in I_{\neg x}$ and $F\cap I=\emptyset$, for some $F\in\mathcal{F}_{pr}(\textbf{D})$. As $\neg x\in I$, we get  $\neg x\notin F$. So $x\in F$, $F$ being a primary filter. $F\nabla I$ then implies that $F\in F_{x}$. Therefore $I_{\neg x}\subseteq F_{x}^{\blacksquare}$.\\
On the other hand,  let $I\in F_{x}^{\blacksquare}$, and if possible, suppose $\neg x\notin I$. Then $I\cap F(\{\neg x\})=\emptyset,$ where $F(\{\neg x\})$ is the filter generated by $\neg x$, as otherwise $\neg x\in I$. Therefore by Theorem \ref{PITDB}, there exists a primary filter $F$ containing $F(\{\neg x\})$  and such that $F\cap I=\emptyset$. This implies that  $\neg x\in F$, that is $x\notin F$ -- which  contradicts the assumption that $I\in F_{x}^{\blacksquare}.$ Hence $F_{x}^{\blacksquare}=I_{\neg x}.$
\\
(i)(b). Let $I\in I_{\lrcorner(x \sqcap x)}.$ Then  $x\sqcap x\notin I$. Let $F_{0}$ be the filter generated by $x\sqcap x$. $F_{0}\cap I=\emptyset$, as otherwise $x\sqcap x\in I$. Therefore by Theorem \ref{PITDB}, there is a primary filter $F$ containing  $F_{0}$ such that $F\cap I=\emptyset$ and $x\in F$. Hence $I_{\lrcorner(x\sqcap x)}\subseteq F_{x}^{\blacklozenge}$.\\
 Now suppose $I\in F_{x}^{\blacklozenge}$. Then there exists $F_{1}\in F_{x}$ such that $I\cap F_{1}=\emptyset$. Since $x\in F_{1}$, $x\sqcap x\in F_{1}$ so $x\sqcap x\notin I$ and hence $\lrcorner(x\sqcap x)\in I$.\\
 The proofs of (ii)(a) and (b)  are similar to the above.
%\noindent 2. To prove $I_{x}^{\square}=F_{\lrcorner x}$ let $F$ be primary filter containing $\lrcorner x.$ Let $I\in \mathcal{I}_{pr}(\textbf{D})$ such that $F\cap I=\emptyset.$ Then $\lrcorner x \notin I.$ Therefore $x\in I.$ Hence $F_{\lrcorner x}\subseteq I_{x}^{\square}$. Now let $F \in I_{x}^{\square}$ we will show that $\lrcorner x\in F$. Let us assume that $\lrcorner x\notin F$. Then we can find a primary ideal          $I_{0}$ that contains $\lrcorner x$ and $F\cap I_{0}=\emptyset$, which contradict that $F\in I_{x}^{\square}.$ Hence $\lrcorner x\in F.$ Therefore $I_{x}^{\square}=F_{\lrcorner x}.$\\
%Proof of  the second part  follows from the Lemma \ref{derivation}(2),Theorem \ref{property of box}(5) and the fact that $\lrcorner\lrcorner x=x\sqcup x.$
\end{proof}

\begin{definition}
\label{spec}
{\rm A topology $\mathcal{J}$ on $\mathcal{I}_{pr}(\textbf{D})$ is defined by taking $\mathcal{B}:=\{I_{x} : x\in D\}$ as a subbase for the closed sets of $\mathcal{J}.$ Similarly, a topology $\mathcal{T}$ on $\mathcal{F}_{pr}(\textbf{D})$ is defined by taking $\mathcal{B}_{0}:=\{F_{ x}:x\in D\}$ as a subbase for the closed sets of $\mathcal{T}.$}
\end{definition}

\begin{note}
\label{clopen subbase}
{\rm   For each $x\in D$, $I_{x}$ is  clopen in $(\mathcal{I}_{pr}(\textbf{D}),\mathcal{J})$ and $F_{x}$ is clopen in $(\mathcal{F}_{pr}(\textbf{D}),\mathcal{T})$. Moreover, 
for each open set $O$  in $(\mathcal{F}_{pr}(\textbf{D}),\mathcal{T})$, $O^{c}=\cap_{j\in J}\cup_{a\in D_{j}}F_{a}$, where $D_{j}$ is some finite subset of $D$ for each $j$ belonging to some index set $ J$. Therefore 
 $O=\cup_{j\in J}\cap_{a\in D_{j}}F^{c}_{a}=\cup_{j\in J}\cap_{a\in D_{j}}F_{\neg a}$.\\
 Similarly, for each open set $O$ in $(\mathcal{I}_{pr}(\textbf{D}),\mathcal{J})$, $O=\cup_{j\in J}\cap_{a\in D_{j}}I_{\lrcorner a}$.  }
\end{note}
%**What happened to rest of the note?**

\noindent Let us recall 
 \begin{theorem}[\textbf{Alexander's Subbase Lemma} \cite{MR946626}]
 \label{ASBT}
{\rm  Let $(X,\tau)$ be a topological space and $S_0$ a subbase of closed sets of $(X,\tau)$. If every family of closed sets in $S_0$ with finite intersection property has a non-empty intersection, then  $(X,\tau)$ is compact.} 
\end{theorem}
\noindent The following can now be established.
   \begin{theorem}
   \label{property of topspe}
{\rm $(\mathcal{F}_{pr}(\textbf{D}),\mathcal{T})$ and $(\mathcal{I}_{pr}(\textbf{D}),\mathcal{J})$
 are compact and totally disconnected topological spaces. Hence both are Hausdorff spaces.}
 \end{theorem}
\begin{proof}
We prove  for $(\mathcal{F}_{pr}(\textbf{D}), \mathcal{T})$; the other case follows dually.\\ To show that $(\mathcal{F}_{pr}(\textbf{D}), \mathcal{T})$ is compact, we consider the subbase $\mathcal{B}_{0}$ for $\mathcal{T}$ given in Definition \ref{spec} and a family $\mathfrak{C}\subseteq\mathcal{B}_{0} $ having the finite intersection property. We show that $\cap\mathfrak{C}\neq \emptyset$ and use Theorem \ref{ASBT}. \\
%We shall show that $\cap\mathfrak{C}\neq\emptyset$. 
Let $S:=\{a\in D: F_{a}\in\mathfrak{C} \}$ and $F(S)$ be the filter generated by $S$. We claim that $F(S)\neq D$. %If possible let $F(S)=D.$ 
Indeed, if not, $\bot\in F(S)$ -- which implies that $a_{1}\sqcap \ldots \sqcap a_{n}\sqsubseteq \bot$ for some $a_{1},\ldots, a_{n}$ in $S.$  Therefore $F_{a_{1}\sqcap\ldots\sqcap a_{n}}=\emptyset$, as $F_{\bot}=\emptyset.$ So $\cap^{n}_{i=1}F_{a_{i}}=F_{\sqcap^{n}_{i=1}a_{i}}=\emptyset$, which is a contradiction. \\
As $F(S)\neq D$, by Theorem \ref{primary ideal thorem}, there exists $F\in \mathcal{F}_{pr}(\textbf{D})$ such that $S\subseteq F.$ 
Hence $F\in \cap\mathfrak{C}$, that is $\cap\mathfrak{C}\neq \emptyset$. \\
%Therefore by Alexander subbase theorem $(\mathcal{F}_{pr}(\textbf{D}),\mathcal{T})$ is compact. 
Now let $x\in D.$ $F_{x}$ is clopen  in $\mathcal{F}_{pr}(\textbf{D})$. Let $F_{1},F_{2}\in \mathcal{F}_{pr}(\textbf{D})$  and $F_{1}\neq F_{2}.$ Then either $F_{1}\nsubseteq F_{2}$ or $F_{2}\nsubseteq F_{1}$. Without loss of generality, suppose  $F_{1}\nsubseteq F_{2}$. So there is $x\in F_{1}$ but $x\notin F_{2}$, implying that  $F_{1}\in F_{x}$ and $F_{2}\notin F_{x}.$  Hence $(\mathcal{F}_{pr}(\textbf{D}),\mathcal{T})$ is totally disconnected.
\end{proof}
Corresponding to the dBa \textbf{D}, we next consider a CTS $\mathbb{K}_{pr}^{T}(\textbf{D}):=((\mathcal{F}_{pr}(\textbf{D}),\mathcal{T}),(\mathcal{I}_{pr}(\textbf{D}),\mathcal{J}),\nabla)$ that is based on the context $\mathbb{K}_{pr}(\textbf{D})$ (mentioned at the beginning of this section) and topologies $\mathcal{T}$ and $\mathcal{J}$ as given in Definition \ref{spec}. It will be shown that $\mathbb{K}^{T}_{pr}(\textbf{D})$ is a CTSCR (Theorem \ref{topcontext}  below). To prove this, we shall use {\it Rado's Selection Principle}.  	
			
 \begin{theorem}[\textbf{Rado's Selection Principle} \cite{rado1949axiomatic}]{\rm  Let $I$ be an arbitrary index set and $S$ be any set. Let $\mathcal{U}:=\{A_{i}:~i\in I\}$ be a family of non-empty finite subsets of a set $S$ and $ \mathfrak{I}$ be the collection of all non-empty finite subsets of $I$.  Further,  for each $J\in \mathfrak{I}$, let there be a function $f_{J}: J\rightarrow\cup_{i\in J} A_{i}~\mbox{such that}~f_{J}(i)\in A_{i}~\mbox{for all}~ i\in J.$ Then there exists a function $f:I\rightarrow \cup_{i\in I}A_{i}~\mbox{such that}~ f(i)\in A_{i}~\mbox{for all}~i\in I$, and for every $J\in \mathfrak{I} $ there is some $K\in \mathfrak{I}$ with $J\subseteq K$ such that $f_{K}(j)=f(j)$ for all $j\in J$.}
   \end{theorem}
   
\noindent   Let us fix some notations for the dBa $\textbf{D}$. 
\begin{notation}{\rm 
%$\mathcal{U}_{0}:=\{D_{i}\in \mathcal{P}(D)~:~ D_{i}~\mbox{a finite subset of}~  D\}_{ i\in J}$, 
$\mathcal{U}_{0}:=\{D_{i}\}_{ i\in J}$ denotes the family of all non-empty finite subsets $D_{i}$ of $D$, indexed over a set $J$. 
$\mathfrak{J}_{0}$ denotes
%:=\{E~:~E\subseteq_{finite} J\}$, 
the set of all non-empty finite subsets of $J$.
 For each $i\in J$, \\$D_{i\sqcap}:=\{a\sqcap a~:~ a\in D_{i}\}$ and $D_{i\sqcup}:=\{a\sqcup a~:~ a\in D_{i}\}$. Note that $D_{i\sqcap}\subseteq D_{\sqcap}$ and $D_{i\sqcup}\subseteq D_{\sqcup}$.\\
 For each $E\in \mathfrak{J}_{0}$, \\$\mathfrak{F}_{E}:=\{f:E\rightarrow \cup_{j\in J}D_{i\sqcap}~\mbox{such that}~f(j)\in D_{j\sqcap}\}$ and \\
 $\mathfrak{G}_{E}:=\{g:E\rightarrow \cup_{j\in J}D_{i\sqcup}~\mbox{such that}~ g(j)\in D_{j\sqcup}\}.$ \\$\sqcap f(E):=\sqcap_{j\in E}f(j)$ for each $f\in \mathfrak{F}_{E} $, and $\sqcup g(E):=\sqcup_{j\in E}g(j)$ for each $g\in \mathfrak{G}_{E} $. \\
 Note that $\sqcap f(E)\in D_{\sqcap}$ and $\sqcup g(E)\in D_{\sqcup}$ for each $f\in  \mathfrak{F}_{E} $ and $g\in  \mathfrak{G}_{E} $ respectively. \\
 %For $E\in \mathfrak{J}_{0}$,
  %$\mathcal{U}_{0},\mathfrak{I}_{0},\mathfrak{F}_{E},\mathfrak{G}_{E}$ and $\downarrow D_{i\sqcap},\uparrow D_{i\sqcup}$ defined as above. 
   $\mathcal{D}_{E}:=\{\sqcap f(E)~:~ f\in \mathfrak{F}_{E}\}$,  $ \mathcal{B}_{E}:=\{\sqcup g(E)~:~ g\in \mathfrak{G}_{E}\}$.
% *$\downarrow D_{i\sqcap}$ is the down-set of $D_{i\sqcap}$ in the Boolean algebra $\textbf{D}_{\sqcap}$  that is $\downarrow D_{i\sqcap}=\{x\in D_{\sqcup}: x\sqsubseteq y~\mbox{for some}~y\in D_{i\sqcap}\}$ and $\uparrow D_{i\sqcup}$ is the up-set of $D_{i\sqcup}$ in the Boolean algebra $\textbf{D}_{\sqcup}$ that is $\uparrow D_{i\sqcup}=\{x\in D_{\sqcup}: y\sqsubseteq x~\mbox{for some}~y\in D_{i\sqcup}\}$*.
}
\end{notation}
% \noindent 
   \begin{observation}
   \label{require to show topological context}
  {\rm \noindent 
  %Let $E\in \mathfrak{J}_{0}$,
  %$\mathcal{U}_{0},\mathfrak{I}_{0},\mathfrak{F}_{E},\mathfrak{G}_{E}$ and $\downarrow D_{i\sqcap},\uparrow D_{i\sqcup}$ defined as above. 
 % and $\mathcal{D}_{E}:=\{\sqcap f(E)~:~ f\in \mathfrak{F}_{E}\}$,  $ \mathcal{B}_{E}:=\{\sqcup g(E)~:~ g\in \mathfrak{G}_{E}\}$. We have the following.
   \begin{enumerate}[{(i)}]
   \item  $\mathfrak{F}_{E}$ is non-empty and finite, as both $E$ and $\cup_{j\in E}D_{i\sqcap}$ are non-empty and finite. Similarly, $\mathfrak{G}_{E}$ is non-empty and finite.  
  \item $\mathcal{D}_{E} (\mathcal{B}_{E})$ is a non-empty and finite subset of $D_{\sqcap}(D_{\sqcup})$, as  $\mathfrak{F}_{E}(\mathfrak{G}_{E})$ is non-empty and finite.
  % \item $\cap_{j\in E}\downarrow D_{j\sqcap}=\{x\in D_{\sqcap}~:~x\sqsubseteq_{\sqcap}\sqcap f(E)~\mbox{for some}~f\in \mathfrak{F}_{E}\}$.
  % \item $\vee \mathcal{D}_{E}$ exists in the Boolean algebra $\textbf{D}_{\sqcap}.$ It is denoted by $a_{E}$.
   %\item $\cap_{j\in E}\uparrow D_{j\sqcup}=\{x\in D_{\sqcup}~:~\sqcup g(E)\sqsubseteq_{\sqcup} x~\mbox{for some}~g\in \mathfrak{G}_{E}\}$.
   %\item $\wedge \mathcal{B}_{E}$ exists in the Boolean algebra $\textbf{D}_{\sqcup}.$ It is denoted by $b_{E}$.
   \end{enumerate}}
   \end{observation}
Let us recall  Notation \ref{finite boolean join and meet} and introduce 
   \begin{notation}
    {\rm For each $E\in\mathfrak{J}_{0}$, $a_{E}:=\vee \mathcal{D}_{E}=\vee_{f\in \mathfrak{F}_{E}}\sqcap f(E)$ and $b_{E}:=\wedge \mathcal{B}_{E}=\wedge_{g\in \mathfrak{G}_{E}}\sqcup g(E)$.}
   \end{notation}

   Recall Proposition \ref{pro1} stating that $\textbf{D}_{\sqcap}:=(D_{\sqcap},\sqcap,\vee,\neg,\bot,\neg\bot)$ and  $\textbf{D}_{\sqcup}:=(D_{\sqcup},\sqcup,\wedge,\lrcorner,\top,\lrcorner\top)$ are Boolean algebras, so that $\vee$ gives the least upper bound operator in $\textbf{D}_{\sqcap}$ and  $\wedge$ gives the greatest lower bound operator in $\textbf{D}_{\sqcup}$. Therefore, $\sqcap f(E)\sqsubseteq a_{E}$ for any $f\in \mathfrak{F}_{E}$ and $b_{E}\sqsubseteq \sqcup g(E) $ for all $g\in \mathfrak{G}_{E} $. 
   %, and similarly for $b_{E}$.
    \begin{theorem}
   \label{topcontext}
  {\rm  $\mathbb{K}_{pr}^{T}(\textbf{D})$ is a CTSCR.}
   \end{theorem}
  \begin{proof}
  We  prove that $\nabla^{-1}$ is (i) upper semi-continuous and (ii) lower semi-continuous. Proofs similar to those of (i) and (ii) will imply that $\nabla $ has these two properties as well.\\
  (i)   Using Theorem \ref{conv upper semi-continuty} (\textbf{II}),  it is sufficient to show that for each open set $O$ in $\mathcal{T}$, $O^{\blacksquare}$ is open in $\mathcal{J}$. Suppose $O$ is an open set in $\mathcal{T}$. Then from Theorem \ref{property of box}(vi) it follows that $O^{\blacksquare}_{\nabla^{-1}}=O^{c\prime}_{\Delta^{-1}}$. 
   $O^{c}$  is closed in $\mathcal{T}$, therefore $O^c=\cap_{j\in J}\cup_{a\in D_{j}}F_{a}$. To show $O^{c\prime}_{\Delta^{-1}}$ is open, we will show that $O^{c\prime}_{\Delta^{-1}}$ is the union of some open sets $I_{a}$, for which we need to find elements $a$ in D.  Let $B:=\{a_{E}~:~ E\in\mathfrak{J}_{0}\}$.
   We show that $O_{\Delta^{-1}}^{c\prime}=\cup_{a\in B}I_{a}$.

    Let $I\in \cup_{a\in B}I_{a}.$ Then there exists $a_{E_{0}}\in B$ such that $a_{E_{0}}\in I$.
     Let $F\in O^{c}$ -- that is   if and only if $F\in\cup_{a\in D_{j}}F_{a}$ for all $j\in J$, which is equivalent to $F\cap D_{j}\neq\emptyset$ for all $j\in J.$ Since $F$ is a filter, $F\cap D_{j}\neq\emptyset$  if and only if $F\cap D_{j\sqcap}\neq\emptyset$ for all $j\in J.$ So  $F\in O^{c}$ implies that $F\cap D_{j\sqcap}\neq\emptyset$ for all $j\in E_{0}$.
    As $F\cap D_{j\sqcap}$ is finite, we choose and fix $b_{j}\in F\cap D_{j\sqcap}$ for each $j\in E_{0}$
      and define a function $f$ from $E_{0}$ to  $\cup_{j\in E_{0}}D_{j\sqcap}$ by $f(j):=b_{j}$ for all $j\in E_{0}.$ Then  $f\in \mathfrak{F}_{E_{0}}$.  
      Since $F$ is a filter in \textbf{D}, $\sqcap f(E_{0})\in F$.
      As pointed out before the theorem, $\sqcap f(E_{0})\sqsubseteq a_{E_{0}}$. So $a_{E_{0}}\in F$, as $F$ is a filter in \textbf{D}. Therefore $F\cap I\neq\emptyset$, which implies that $I \in O_{\Delta^{-1}}^{c\prime}$. Thus $\cup_{a\in B}I_{a}\subseteq O_{\Delta^{-1}}^{c\prime}$.
  
   The other direction, that is $O_{\Delta^{-1}}^{c\prime}\subseteq \cup_{a\in B}I_{a} $, is proved by contraposition. Let $I$ be a primary ideal of $\textbf{D}$ such that $a\notin I$ for all $a\in B$.
By Observation \ref{require to show topological context}, $\mathfrak{F}_{E}\neq\emptyset$ for each $E\in \mathfrak{J}_{0}$. Now we show that there exists a function $f_{E}\in \mathfrak{F}_{E}$ such that  $\sqcap f_{E}(E)\notin I$.
   
   \noindent  Indeed, if possible suppose that $\sqcap f(E)\in I$ for all $f\in \mathfrak{F}_{E}$. Then $\sqcup_{f\in \mathfrak{F}_{E}}\sqcap f(E)\in I$, as $I$ is an ideal and $\mathfrak{F}_{E}$ finite. Now by Proposition  \ref{relation of join and meet}, $a_{E}:=\vee \mathcal{D}_{E}=\vee_{f\in \mathfrak{F}_{E}}\sqcap f(E)\sqsubseteq\sqcup_{f\in \mathfrak{F}_{E}}\sqcap f(E) $. So $a_{E}\in I$, which is a contradiction as $a_{E}\in B$. 

\noindent Therefore by Rado's Selection Principle, there exists a function $f_{0}$ from $J$ to $\cup_{j\in J}D_{j\sqcap}$ such that $f_{0}(j)\in D_{j\sqcap}$, and for each $E\in \mathfrak{J}_{0}$ there exists $K\in \mathfrak{J}_{0}$ with $E\subseteq K$ such that $f_{0}(j)=f_{K}(j)$ for all $j\in E$. Let $F$ be the filter in \textbf{D} generated by $\{f_{0}(j): j\in J\}$. We claim that $I\cap F=\emptyset$. If possible, suppose there exists $x\in D$ such that $x\in I\cap F$. Then there is $E\in \mathfrak{J}_{0}$ such that  $\sqcap f_{0}(E)\sqsubseteq x$, as $x\in F$. So there exists $K\in \mathfrak{J}_{0}$ such that $E\subseteq K$ and for all $j\in E$, $f_{K}(j)=f_{0}(j)$, which implies that $\sqcap f_{0}(E)=\sqcap f_{K}(E)\sqsupseteq \sqcap f_{K}(K)$. Then $\sqcap f_{K}(K)\sqsubseteq x$. This implies $\sqcap f_{K}(K)\in I$, which is a contradiction. Therefore $I\cap F= \emptyset$ and by the prime ideal theorem for dBas (Theorem \ref{PITDB}), there exists  $F_{0}\in\mathcal{F}_{pr}(\textbf{D}) $ such that $I\cap F_{0}= \emptyset$ and $F\subseteq F_{0}$. Then $F_{0}\cap D_{j\sqcap}\neq \emptyset$ for each $j\in J$, as $\{f_{0}(j): j\in J\}\subseteq F\subseteq F_{0}$. This implies that $F_{0}\cap D_{j}\neq \emptyset$ for each $j\in J$. So $F_{0}\in O^{ c}$ and hence  $I\notin O^{c\prime}_{\Delta^{-1}}$. So $I\in O^{ c\prime}_{\Delta^{-1}}$ implies that $I\in \cup_{a\in B_{J}}I_{a}$. 

Thus $O^{ c \prime}_{\Delta^{-1}}=\cup_{a\in B}I_{a}$, and $O^{\blacksquare}_{\nabla^{-1}}=O^{ c \prime}_{\Delta^{-1}}=\cup_{a\in B}I_{a}$, which is open as $I_{a}$ is open for all $a\in D$. Hence $\nabla^{-1}$ is upper semi-continuous. 
  
 \noindent  (ii) Let $O$ be an open set in $(\mathcal{F}_{pr}(\textbf{D}),\mathcal{T})$. Then by Note \ref{clopen subbase}, 
 $O=\cup_{j\in J}\cap_{a\in D_{j}}F_{\neg a}
 =\cup_{j\in J}F_{\neg(\vee_{a\in D_{j}}a)},$ and 
 $O^{\blacklozenge}=(\cup_{j\in J}F_{\neg(\vee_{a\in D_{j}}a)})^{\blacklozenge}
 =\cup_{j\in J}F^{\blacklozenge}_{\neg(\vee_{a\in D_{j}}a)}
 =\cup_{j\in J}I_{\lrcorner(\neg(\vee_{a\in D_{j}}a)\sqcap \neg(\vee_{a\in D_{j}}a))}
 =\cup_{j\in J}I_{\lrcorner(\neg(\vee_{a\in D_{j}}a))}$, \\which is an open set as $ I_{a}$ 
 is  open for all $a\in D$. Hence by Theorem \ref{conv upper semi-continuty}(\textbf{I}), $\nabla^{-1}$ is lower semi-continuous.
  \end{proof} 
   \begin{corollary}
   \label{BtoScxt}
{\rm If $\textbf{D}$ is Boolean  then $\nabla$ is a homeomorphism from $(\mathcal{F}_{pr}(\textbf{D}),\mathcal{T})$ to $(\mathcal{I}_{pr}(\textbf{D}),\mathcal{J})$.}
\end{corollary}
\begin{proof}
Continuity of $\nabla$ and $\nabla^{-1}$ follows from  Theorem \ref{topcontext}. We show that $\nabla$ is a bijection from $\mathcal{F}_{pr}(\textbf{D})$ to $\mathcal{I}_{pr}(\textbf{D})$. Let $F\in \mathcal{F}_{pr}(\textbf{D})$, and $I_{1},I_{2}\in\mathcal{I}_{pr}(\textbf{D})$ be such that $F\nabla I_{1}, F\nabla I_{2}$. Then $F\cap I_{1}=\emptyset$, $F\cap I_{2}=\emptyset$, which implies that $I_{1}, I_{2}\subseteq F^{c}$. Since $\textbf{D}$ is a Boolean algebra, the notions of primary filter and ideal coincide with those of prime filter and ideal respectively. $F^{c}$ is then a prime ideal as $F$ is a prime filter, and $I_{1}=F^{c}=I_{2}$ as prime ideals are maximal. So $\nabla$ is a function from $\mathcal{F}_{pr}(\textbf{D})$  to $\mathcal{I}_{pr}(\textbf{D})$. Similarly, one can show that $\nabla^{-1}$ is a function from $\mathcal{I}_{pr}(\textbf{D})$ to $\mathcal{F}_{pr}(\textbf{D})$. Clearly, $\nabla^{-1}\circ \nabla$ is the identity map on $\mathcal{F}_{pr}(\textbf{D})$. So $\nabla$ is a bijection from $\mathcal{F}_{pr}(\textbf{D})$ to $\mathcal{I}_{pr}(\textbf{D})$.
\end{proof}
  \noindent  How are the CTSCRs corresponding to  two quasi-isomorphic dBas related? To see that, let us introduce the following  functions.
   \begin{definition}
   \label{alphbetah}
  {\rm Let $\textbf{M}$ and $\textbf{D}$ be two dBas and $h:\textbf{M}\rightarrow \textbf{D}$ a dBa homomorphism. The function $\alpha_{h}:\mathcal{F}_{pr}(\textbf{D})\rightarrow\mathcal{F}_{pr}(\textbf{M})$ is given by $\alpha_{h}(F):=h^{-1}(F)$  for any $F\in\mathcal{F}_{pr}(\textbf{D}),$ and the function  $\beta_{h}:\mathcal{I}_{pr}(\textbf{D})\rightarrow\mathcal{I}_{pr}(\textbf{M})$ is given by $\beta_{h}(I):=h^{-1}(I)$ for any $I\in \mathcal{I}_{pr}(\textbf{D})$.} 
\end{definition}

\noindent By Proposition \ref{inv-img-prim}(ii and iii) (cf. Section \ref{AIdBa}),  $\alpha_{h}$ and $\beta_{h}$ are well-defined functions. Moreover, 

\begin{proposition}
\label{cntxmor}
{\rm  $\alpha_h^{-1}(F_{x})=F_{h(x)}$  and $\beta_h^{-1}(I_{x})=I_{h(x)}$, for all $x\in D$.}
\end{proposition}
\begin{proof}
Let $x\in D$. Then 
$\alpha_h^{-1}(F_{x})=\{F\in \mathcal{F}_{pr}(\textbf{D}): x\in h^{-1}(F)\}=\{F\in \mathcal{F}_{pr}(\textbf{D}): h(x)\in F\}=F_{h(x)}\\ \mbox{and}~ 
\beta_h^{-1}(I_{x})=\{I\in \mathcal{I}_{pr}(\textbf{D}): x\in h^{-1}(I)\}=\{I\in \mathcal{I}_{pr}(\textbf{D}): h(x)\in I\}=I_{h(x)}$.
\end{proof}

We show in Theorem \ref{db iso to cntx iso} below that if  dBas $\textbf{M}$ and $\textbf{D}$ are quasi-isomorphic, the  CTSCRs $\mathbb{K}^{T}_{pr}(\textbf{M}):=((\mathcal{F}_{pr}(\textbf{M}),\mathcal{T}),(\mathcal{I}_{pr}(\textbf{M}), \mathcal{J}),\nabla_{1})$ and $\mathbb{K}^{T}_{pr}(\textbf{D}):=((\mathcal{F}_{pr}(\textbf{D}),\mathcal{T}),(\mathcal{I}_{pr}(\textbf{D}), \mathcal{J}),\nabla_{2})$ corresponding to $\textbf{M}$ and $\textbf{D}$ are homeomorphic. In order to prove the result, we use the following.

\begin{theorem}
\label{fun homeomorphism}
{\rm \cite{munkres1975topology} Let $(X,\tau), (Y,\rho)$ be two topological spaces and $f: X\rightarrow Y$ be a bijective continuous function. If $(X,\tau)$ is compact and $(Y,\rho)$ is Hausdorff, then $f$ is a homeomorphism.}
\end{theorem}
\begin{theorem}
\label{db iso to cntx iso}
{\rm  Let $h: \textbf{M} \ra \textbf{D}$ be a dBa quasi-isomorphism. Then $(\alpha_{h},\beta_{h}): \mathbb{K}^{T}_{pr}(\textbf{D}) \ra \mathbb{K}^{T}_{pr}(\textbf{M})$  is a CTSCR-homeomorphism.} 
\end{theorem}
\begin{proof}
 Let $F\in \mathcal{F}_{pr}(\textbf{D})$ and $I\in \mathcal{I}_{pr}(\textbf{D})$ such that $F\cap I=\emptyset$. If possible, let $h^{-1}(F)\cap h^{-1}(I)\neq\emptyset$. Then there exists  $x\in h^{-1}(F)\cap h^{-1}(I).$ Hence $h(x)\in F\cap I$, which is a contradiction. So $h^{-1}(F)\cap h^{-1}(I)=\emptyset$. Now let $\alpha_{h}(F)\cap \beta_{h}(I)=h^{-1}(F)\cap h^{-1}(I)=\emptyset$ and if possible, let $F\cap I\neq\emptyset.$ Then there exists $x\in F\cap I$ and $y_1\in h^{-1}(F), y_2\in h^{-1}(I) $ such that $h(y_1)=h(y_{2})=x$. Therefore $y_{2}\in h^{-1}(F)$,  implying that $y_{2}\in h^{-1}(F)\cap h^{-1}(I)$, which is a contradiction. Hence $F\cap I=\emptyset$. So $F\nabla_{1} I$ if and only if $\alpha_{h}(F)\nabla_{2}\beta_{h}(I)$, which implies that $(\alpha_{h},\beta_{h})$ is a context homomorphism. Now let us define a correspondence $f_{h}$ from $\mathcal{F}_{pr}(\textbf{M})$ to  $\mathcal{F}_{pr}(\textbf{D})$ by $f_{h}(F):=h(F)$ for all $F\in \mathcal{F}_{pr}(\textbf{M}).$ By Proposition  \ref{inv-img-prim}(v) it follows that $f_{h}$ is a well-defined map. Similarly, $g_{h}$ from $ \mathcal{I}_{pr}(\textbf{M})$ to $\mathcal{I}_{pr}(\textbf{D})$ defined by $g_{h}(I):=h(I)$ is a  well-defined map. As $h$ is surjective, $f_{h}\circ \alpha_{h}=Id_{\mathcal{F}_{pr}(\textbf{D})}$, the identity map on $\mathcal{F}_{pr}(\textbf{D})$ and  $g_{h}\circ \beta_{h}=Id_{\mathcal{I}_{pr}(\textbf{D})}$, the identity map on $\mathcal{I}_{pr}(\textbf{D})$. Therefore $f_{h}=\alpha^{-1}_{h}$ and $g_{h}=\beta^{-1}_{h}$, that is $\alpha_{h},\beta_{h}$ are  bijections. Hence $(\alpha_{h},\beta_{h})$ is a context isomorphism.

 Next we  show that $(\alpha_{h},\beta_{h})$ is a CTSCR-homeomorphism, for which  $\alpha_{h}$ and $\beta_{h}$ both must be shown to be continuous. Let $O$ be an open set in $(\mathcal{F}_{pr}(\textbf{M}),\mathcal{T}).$ By Note \ref{clopen subbase}, $O=\cup_{j\in J}\cap_{x\in M_{j}}F_{\neg x}$,  where $M_{j}$ is a finite subset of $M$ for all $j\in J.$ So $\alpha_{h}^{-1}(O)=\alpha_{h}^{-1}(\cup_{j\in J}\cap_{x\in M_{j}}F_{\neg x})=\cup_{j\in J}\cap_{x\in M_{j}}\alpha_{h}^{-1}(F_{\neg x})=\cup_{j\in J}\cap_{x\in M_{j}}F_{\neg h (x)}$ by Proposition \ref{cntxmor}. Hence $\alpha_{h}$ is continuous. By Theorems \ref
{fun homeomorphism} and  \ref{property of topspe} it follows that $\alpha_{h}$ is a homeomorphism. Similarly we can show that $\beta_{h}$ is a homeomorphism. Hence $(\alpha_{h},\beta_{h})$ is a CTSCR-homeomorphism.
  %Now let $O^{'}$ be an open set in $\mathcal{I}_{pr}(\textbf{M})$ then $O=\cup_{j\in J}\cap_{x\in M_{j}}I_{\lrcorner x}$ where $M_{j}\subseteq A$ for all $j\in J.$ Then $\beta_{h}^{-1}(O)=\beta_{h}^{-1}(\cup_{j\in J}\cap_{x\in M_{j}}I_{\lrcorner x})=\cup_{j\in J}\cap_{x\in M_{j}}\beta_{h}^{-1}(I_{\lrcorner x})=\cup_{j\in J}\cap_{x\in M_{j}}I_{\lrcorner h(x)}$, by Proposition \ref{cntxmor}.
%Hence $\beta_{h}$ is continuous. Since $\mathcal{F}_{pr}(\textbf{M}),\mathcal{I}_{pr}(\textbf{M})$ and $\mathcal{I}_{pr}(\textbf{D}),\mathcal{F}_{pr}(\textbf{D})$ are compact housdroff spaces, by Theorem \ref{fun homeomorphism}, it follows that $\beta_{h}$ and $\alpha_{h}$ are homeomorphisms. 
    \end{proof}
    \begin{corollary}
\label{idba hom set}
{\rm 
%Let $\textbf{D}$ and $\textbf{M}$ be two pure dBas and 
If $h$ is a dBa isomorphism from  $\textbf{M}$ to $\textbf{D},$  $(\alpha_{h},\beta_{h})$ is a CTSCR-homeomorphism  from $\mathbb{K}^{T}_{pr}(\textbf{D})$ to $\mathbb{K}^{T}_{pr}(\textbf{M})$.}
\end{corollary}
\begin{proof}
Follows from Theorem \ref{db iso to cntx iso}, as any dBa isomorphism is a dBa quasi-isomorphism.
\end{proof}

%\vskip 2pt
%We can now  establish that if two pure dBas are related through a surjective dBa homomorphism, the corresponding CTSCRs get related through an embedding. Recall Definition \ref{alphbetah} giving the maps $\alpha_{h}, \beta_{h}$ corresponding to a dBa homomorphism $h:\textbf{M}\rightarrow \textbf{D}$.  
 %\begin{theorem}
%\label{pdba hom set}
%{\rm Let $\textbf{D}$ and $\textbf{M}$ be two dBas and $h:\textbf{D}\rightarrow \textbf{M}$ be a surjective dBa homomorphism.
% Then the pair $(\alpha_{h},\beta_{h})$ is a CTSCR-embedding from  $\mathbb{K}^{T}_{pr}(\textbf{M})$ into $\mathbb{K}^{T}_{pr}(\textbf{D})$. 
%}
%\end{theorem}
%\begin{proof}
%$\alpha_{h},\beta_{h}$ are injective, since $h$ is surjective. Let $F\in \mathcal{F}_{pr}(\textbf{M}) $  and $I\in \mathcal{I}_{pr}(\textbf{M}).$ Then $F\cap I= \emptyset$ if and only if $h^{-1}(F)\cap h^{-1}(I)=\emptyset$, as $h$ is surjective. To show $\alpha_h$ is continuous, let $O$ be an open set in $(\mathcal{F}_{pr}(\textbf{D}),\mathcal{T})$. Then by Note \ref{clopen subbase}, $O=\cup_{j\in J}\cap_{x\in D_{j}}F_{\neg x}$. Therefore 
%$\alpha_{h}^{-1}(O)=\alpha_{h}^{-1}(\cup_{j\in J}\cap_{x\in D_{j}}F_{\neg x})=\cup_{j\in J}\cap_{x\in D_{j}}\alpha_{h}^{-1}(F_{\neg x})=\cup_{j\in J}\cap_{x\in D_{j}}F_{\neg h (x)}$, by Proposition \ref{cntxmor}, and as h is a dBa homomorphism. So $\alpha_{h}$ is continuous. Similarly one can show that $\beta_{h}$ is continuous. Therefore $(\alpha_{h},\beta_{h})$ is a CTSCR-embedding  from  $\mathbb{K}^{T}_{pr}(\textbf{M})$ into $\mathbb{K}^{T}_{pr}(\textbf{D})$.
%\end{proof}
   Let us  note the following.
   \begin{proposition}
\label{semicon}
{\rm For all $x\in D$, 
 $(F_{\neg x},I_{x})\in \mathfrak{R}^{T}(\mathbb{K}_{pr}^{T}(\textbf{D}))$. Moreover, if $\textbf{D}$ is a pure dBa, $(F_{\neg x},I_{x})\in \mathfrak{S}^{T}(\mathbb{K}_{pr}^{T}(\textbf{D}))$ for all $x\in D$.}
\end{proposition}
\begin{proof}
$F_{x}$ is clopen in $(\mathcal{F}_{pr}(\textbf{D}),\mathcal{T})$  and $I_{x}$ is clopen in $(\mathcal{I}_{pr}(\textbf{D}),\mathcal{J})$ for each $x\in D$ (cf. Note \ref{clopen subbase}). Now  $F_{\neg x}^{\blacksquare\lozenge}=F_{\neg(x\sqcup x)}=I_{x}^{\lozenge}$ by Proposition \ref{clopensemiconcept}. Therefore $(F_{\neg x},I_{x})$ is an object oriented protoconcept. So $(F_{\neg x},I_{x})\in \mathfrak{R}^{T}(\mathbb{K}_{pr}^{T}(\textbf{D}))$.
 
 \noindent In case  $D$ is a pure dBa, either $x=x\sqcap x$ or $x=x\sqcup x$ for all $x\in D$. If $x=x\sqcap x$, $(F_{\neg x},I_{x})=(F_{\neg x},I_{x\sqcap x}).$ Then it follows from  Proposition \ref{clopensemiconcept}(i(a))  that $(F_{\neg x},I_{x})$ is an object oriented semiconcept, as $\neg\neg x=x\sqcap x.$ On the other hand, if $x=x\sqcup x$, $(F_{\neg x},I_{x})=(F_{\neg( x\sqcup x)},I_{x}).$ From  Proposition \ref{clopensemiconcept}(ii(b)) it follows that $(F_{\neg x},I_{x})$ is an object oriented semiconcept. So $(F_{\neg x},I_{x})\in \mathfrak{S}^{T}(\mathbb{K}_{pr}^{T}(\textbf{D}))$.
\end{proof}

 We now give a characterization of the set of clopen object oriented semiconcepts of  $\mathbb{K}^{T}_{pr}(\textbf{D})$. This will be used to obtain the representation theorems for fully contextual as well as pure dBas. 
\begin{theorem}[\textbf{Characterization theorem}]
\label{characterization}
{\rm Let $\textbf{D}$ be a dBa. Then $(A,B)\in \mathfrak{R}^{T}(\mathbb{K}^{T}_{pr}(\textbf{D}))_{p}$ if and only if $(A,B)=(F_{\neg x}, I_{x})$ for some $x\in D_{p}$. Moreover, if $(A,B)\in \mathfrak{R}^{T}(\mathbb{K}^{T}_{pr}(\textbf{D}))_{\sqcap}$ then $x\in D_{\sqcap}$ and if $(A,B)\in \mathfrak{R}^{T}(\mathbb{K}^{T}_{pr}(\textbf{D}))_{\sqcup}$ then $x\in D_{\sqcup}$. }
\end{theorem}
\begin{proof}
Let us recall that $\mathfrak{R}^{T}(\mathbb{K}^{T}_{pr}(\textbf{D}))_{p}=\mathfrak{R}^{T}(\mathbb{K}^{T}_{pr}(\textbf{D}))_{\sqcap}\cup \mathfrak{R}^{T}(\mathbb{K}^{T}_{pr}(\textbf{D}))_{\sqcup}$. Then the following cases  arise.\\
(i) Let $(A,B)\in\mathfrak{R}^{T}(\mathbb{K}^{T}_{pr}(\textbf{D}))_{\sqcap} $. Then $(A,B)=(A,B)\sqcap (A,B)=(A, A^{\blacksquare})$.
Since $A$ is closed, $A=\cap_{j\in J}\cup_{a\in D_{j}}F_{a}$, where $D_{j}'s$ are finite subsets of $D$ for each $j\in J$. As $A$ is open, $A^{c}=\cup_{j\in J}\cap_{a\in D_{j}}F^{c}_{a}$ is a closed subset of $ (\mathcal{F}_{pr}(\textbf{D}), \mathcal{T})$. Therefore $A^{c}$ is compact in $ (\mathcal{F}_{pr}(\textbf{D}), \mathcal{T})$, as  $ (\mathcal{F}_{pr}(\textbf{D}), \mathcal{T})$ is compact. So there exists a finite subset $E$ of $J$ such that $A^{c}=\cup_{j\in E}\cap_{a\in D_{j}}F^{c}_{a}=\cup_{j\in E}\cap_{a\in D_{j}}F_{\neg a}=\cup_{j\in E}F_{\sqcap_{a\in D_{j}}\neg a}=\cup_{j\in E}F_{\neg (\vee_{a\in D_{j}}a)}$ by Lemma \ref{derivation}.
This means $A=\cap_{j\in E}F_{\neg\neg (\vee_{a\in D_{j}}a)}=\cap_{j\in E}F_{ \vee_{a\in D_{j}}a}$, as $\vee_{a\in D_{j}}a$ in $D_{\sqcap}$ for each $j\in E$ by Note \ref{finite-B-meet}. This  implies that $A=F_{\sqcap_{j\in E}(\vee_{a\in D_{j}}a)}=F_{x}$, where $x=\sqcap_{j\in E}(\vee_{a\in D_{j}}a)$. Therefore $(A,A^{\blacksquare})=(F_{x},F^{\blacksquare}_{x})=(F_{x},I_{\neg x})=(F_{x\sqcap x}, I_{\neg x})=(F_{\neg\neg x}, I_{\neg x})$, by Lemma \ref{derivation}(iv), Proposition \ref{clopensemiconcept}, and \ref{pro2}(iii). So $(A,B)=(F_{\neg\neg x}, I_{\neg x})$, where $\neg x\in D_{\sqcap}\subseteq D_{p}$ by Proposition \ref{pro2}(i).\\
(ii)  Let $(A,B)\in\mathfrak{R}^{T}(\mathbb{K}^{T}_{pr}(\textbf{D}))_{\sqcup} $. Then $(A,B)=(A,B)\sqcup (A,B)=(B^{\lozenge},B)$. Similar to the proof of (i), we can show that $(B^{\lozenge}, B)=((I_{b})^{\lozenge}, I_{b})=(F_{\neg(b\sqcup b)}, I_{b})$ for some $b\in D$. Using Lemma \ref{derivation}(iii), $(A,B)=(F_{\neg(b\sqcup b)}, I_{b\sqcup b})$, where $b\sqcup b\in D_{\sqcup}\subseteq D_{p}$ by Proposition \ref{pro2}(iv).
\end{proof}

 \begin{corollary}
  \label{property-ston}
{\rm Let $\textbf{D}$ be a dBa. 
Then $F\nabla I$ if and only if for all $(X,Y)\in\mathfrak{R}^{T}(\mathbb{K}^{T}_{pr}(\textbf{D}))_{p}, I\in Y$ implies that $F\in X$.}
\end{corollary}
\begin{proof} 
%** Where is the Theorem mentioned in the proof?**
Let $F\nabla I$ and $(X,Y)\in \mathfrak{R}^{T}(\mathbb{K}^{T}_{pr}(\textbf{D}))_{p}$ such that $I\in Y$. Then $F\cap I=\emptyset$ and by Theorem \ref{characterization}, $(X,Y)=(F_{x},I_{\neg x})$ for some $x\in D_{p}$. $\neg x\in I$ and $F\cap I=\emptyset$ imply that $\neg x\notin F$, and so $ x \in F$. Hence $F\in X=F_{x}$. 
For the converse, let $F\in\mathcal{F}_{pr}(\textbf{D})$, and $I\in \mathcal{I}_{pr}(\textbf{D})$ be such that for all $(X,Y)\in\mathfrak{R}^{T}(\mathbb{K}^{T}_{pr}(\textbf{D}))_{p}$, $I\in Y$ implies  $F\in X$. If possible, let $F\cap I\neq\emptyset$. Then there exists $a\in F\cap I$. Therefore $I\in I_{a\sqcup a}$ and $a\sqcup a\in F$, that is $\neg( a\sqcup a)\notin F$. So $F\notin F_{\neg (a\sqcup a)}$, which is a contradiction, as $(F_{\neg(a\sqcup a)},I_{a\sqcup a})\in\mathfrak{R}^{T}(\mathbb{K}^{T}_{pr}(\textbf{D}))_{p}$. Therefore $F\cap I=\emptyset$, which implies $F\nabla I.$
\end{proof}
%\begin{definition}**
%{\rm Let $\textbf{D}$ and $\textbf{M}$ be two dBas. We say $\textbf{D}$ is quasi-embedded into $\textbf{M}$ if there is a quasi-injective dBa homomorphism $h$ from $\textbf{D}$ into $\textbf{M}$ and $h$ is called a dBa quasi-embedding. If $h$ is an injective dBa homomorphism, we say $\textbf{D}$ is embedded into $\textbf{M}$ and in this cases $h$ is called a dBa embedding }
%\end{definition}
As mentioned in Theorem \ref{object-proto} (cf. Section \ref{repres}), any dBa \textbf{D} can be quasi-embedded into the algebra $\underline{\mathfrak{R}}(\mathbb{K}_{pr}(\textbf{D}))$ of object oriented protoconcepts, through the map 
 $h:D\rightarrow\mathfrak{R}(\mathbb{K}_{pr}(\textbf{D}))$ given by  $h(x):=(F_{\neg x}, I_{x})$ for all $x\in D$.  Proposition \ref{semicon} gives that  for all $x\in D$, $h(x) \in \mathfrak{R}^{T}(\mathbb{K}_{pr}^{T}(\textbf{D}))$, which is a subset of $\mathfrak{R}(\mathbb{K}_{pr}(\textbf{D}))$.  The next representation theorem establishes  that, in fact, the dBa $\textbf{D}$ can be quasi-embedded into the subalgebra $\underline{\mathfrak{R}}^{T}(\mathbb{K}_{pr}^{T}(\textbf{D}))$ of $\underline{\mathfrak{R}}(\mathbb{K}_{pr}(\textbf{D}))$, using the map $h$.
\begin{theorem}[\textbf{Representation theorem for  dBas and contextual dBas}]
\label{embedding of dBa}
\noindent
 \begin{enumerate}[{(i)}]
\item {\rm Any dBa $\textbf{D}$ is quasi-embedded  into  $\underline{\mathfrak{R}}^{T}(\mathbb{K}_{pr}^{T}(\textbf{D}))$,  where $h:D\rightarrow\mathfrak{R}^{T}(\mathbb{K}_{pr}^{T}(\textbf{D}))$  defined by $h(x):=(F_{\neg x}, I_{x})$ for all $x\in D$, gives the required quasi-injective dBa homomorphism. Moreover, $\textbf{D}_{p}$ is isomorphic to $\mathcal{S}^{T}(\mathbb{K}_{pr}^{T}(\textbf{D}))$. }
\item {\rm Any  contextual dBa $\textbf{D}$ is embedded into  $\underline{\mathfrak{R}}^{T}(\mathbb{K}_{pr}^{T}(\textbf{D}))$, the above map $h$ giving the required injective dBa homomorphism from $\textbf{D}$ into $\underline{\mathfrak{R}}^{T}(\mathbb{K}_{pr}^{T}(\textbf{D}))$.}
\end{enumerate}

\end{theorem}

\begin{proof}
(i) From Proposition \ref{semicon} it follows that $h$ is a well-defined  map. Let $x,y \in D$. Then \\ $h(x\sqcap y)=(F_{\neg (x\sqcap y)},I_{x\sqcap y})=(F_{\neg x},I_{x})\sqcap (F_{\neg y},I_{y})=h(x)\sqcap h(y)$, and \\$h(x\sqcup y)=(F_{\neg (x\sqcup y)},I_{x\sqcup y})=(F_{\neg x},I_{x})\sqcup (F_{\neg y},I_{y})=h(x)\sqcup h(y)$. \\Similarly, one can show that \\$h(\lrcorner x)=\lrcorner h(x), h(\neg x)=\neg h(x),h(\top)=(\emptyset,\emptyset)$ and $h(\bot)=(\mathcal{F}_{pr}(\textbf{D}),\mathcal{I}_{pr}(\textbf{D}))$. \\Therefore $h$ is a dBa homomorphism. 
 
 Now we  show that $h$ is quasi-injective. Let $x,y\in D$ and $h(x)\sqsubseteq h(y)$. If possible, suppose $x\not\sqsubseteq y$. Then by Proposition \ref{pro1}(iii), either $x_{\sqcap}\not\sqsubseteq_{\sqcap} y_{\sqcap}$ or $x_{\sqcup}\not\sqsubseteq_{\sqcup} y_{\sqcup}$. If $x_{\sqcap}\not\sqsubseteq_{\sqcap} y_{\sqcap}$  then by the prime ideal theorem  of Boolean algebras, there exists a prime filter $F_0$ in $\textbf{D}_{\sqcap}$ such that $x\sqcap x \in F_{0}$ and $y\sqcap y\notin F_0$. Therefore by Lemma \ref{lema1}(ii) and Proposition \ref{comparison of two ideal}, it follows that there exists a primary filter $F$ such that $F_{0}=F\cap D_{\sqcap}$. So $x\sqcap x \in F$ and $\neg y \in F$ (as $y\sqcap y\notin F$). Thus $F\in F_{\neg y}$ but $F\notin F_{\neg x}$, which implies $F_{\neg y}\not\subseteq F_{\neg x}$.  Now if $x_{\sqcup} \not\sqsubseteq_{\sqcup} y_{\sqcup} $ then dually we can show that $I_{y}\not\subseteq I_{x}$.  Therefore in both  cases, $h(x)\not\sqsubseteq h(y)$, which is a contradiction. So $x\sqsubseteq y$. Conversely, let $x\sqsubseteq y.$ Then $I_{y}\subseteq I_{x}$. Further $x_{\sqcap}\sqsubseteq_{\sqcap} y_{\sqcap}$ and $x_{\sqcup}\sqsubseteq_{\sqcup} y_{\sqcup}$. Therefore by Proposition \ref{pro2}(ii),  $\neg y=\neg y_{\sqcap}\sqsubseteq_{\sqcap} \neg x_{\sqcap}=\neg x$. So $F_{\neg y}\subseteq F_{\neg x}.$ Hence $h(x)\sqsubseteq h(y)$, using Proposition \ref{order object-proto}.  
 
 From Theorem \ref{characterization} it follows that $h\vert_{D_{p}}$  is a surjective dBa homomorphism from $\textbf{D}_{p}$ onto $\underline{\mathfrak{R}}^{T}(\mathbb{K}^{T}_{pr}(\textbf{D}))_{p}$.   The restriction of quasi order $\sqsubseteq $ on $D_{p}$ becomes a partial order by Proposition \ref{order pure}. So  $h\vert_{D_{p}}$ is a dBa isomorphism from $\textbf{D}_{p}$ to $\underline{\mathfrak{R}}^{T}(\mathbb{K}^{T}_{pr}(\textbf{D}))_{p}$.
Proposition \ref{largest sub algebra} applied to the case of $\underline{\mathfrak{R}}^{T}(\mathbb{K}^{T}_{pr}(\textbf{D}))_{p}$ and $\mathcal{S}^{T}(\mathbb{K}^{T}_{pr}(\textbf{D}))$ gives $\underline{\mathfrak{R}}^{T}(\mathbb{K}^{T}_{pr}(\textbf{D}))_{p}=\mathcal{S}^{T}(\mathbb{K}^{T}_{pr}(\textbf{D}))$.
 
 %**From Lemma \ref{Aj} and Theorem \ref{clopen dba} it follows that % Let $(A,B)\in \mathfrak{R}^{T}(\mathbb{K}^{T}_{pr}(\textbf{D}))_{p}$. Then either $(A,B)=(A,B)_{\sqcap}=(A,A^{\blacksquare})$ or $(A,B)=(A,B)_{\sqcup}=(B^{\lozenge}, B)$. In both the cases $(A,B)\in \mathfrak{S}^{T}(\mathbb{K}^{T}_{pr}(\textbf{D}))$. So  $\underline{\mathfrak{R}}^{T}(\mathbb{K}^{T}_{pr}(\textbf{D}))_{p}$ is a subalgebra of $\mathcal{S}^{T}(\mathbb{K}^{T}_{pr}(\textbf{D}))$. Since $\mathcal{S}^{T}(\mathbb{K}^{T}_{pr}(\textbf{D}))$ is also a  subalgebra of $\underline{\mathfrak{R}}^{T}(\mathbb{K}^{T}_{pr}(\textbf{D}))$,
 
\vskip 2pt
\noindent (ii) If $\textbf{D}$ is a contextual dBa then the quasi-order $\sqsubseteq$ becomes a partial order, whence $h$ becomes an injective dBa homomorphism.
\end{proof}

\noindent Recall the special case when a  dBa $\textbf{D}$ is finite, as remarked upon in Section \ref{repres}. We now show that  in this case, the topological spaces $(\mathcal{F}_{pr}(\textbf{D}), \mathcal{T})$ and $(\mathcal{I}_{pr}(\textbf{D}), \mathcal{J})$  become  discrete. So, effectively,  the topologies do not play any role in the representation theorem above and the result coincides with Corollary \ref{finite rep object}. 

%are we have the following proposition
\begin{proposition}
{\rm For a finite dBa $\textbf{D}$, $(\mathcal{F}_{pr}(\textbf{D}), \mathcal{T})$ and $(\mathcal{I}_{pr}(\textbf{D}), \mathcal{J})$  are discrete topological spaces.
%the following hold.
%\begin{enumerate}[(i)]
%\item $(\mathcal{F}_{pr}(\textbf{D}), \mathcal{T})$ is a discrete topological space.
%\item $(\mathcal{I}_{pr}(\textbf{D}), \mathcal{J})$ is a discrete topological space.
%\end{enumerate}
}\end{proposition}

\begin{proof}
(i) %To show 
Consider  $(\mathcal{F}_{pr}(\textbf{D}), \mathcal{T})$. 
%is a discrete topological space, 
We show that any subset $A$ of $\mathcal{F}_{pr}(\textbf{D})$ is open in $(\mathcal{F}_{pr}(\textbf{D}), \mathcal{T})$. Let $A:=\{F_{1},\ldots, F_{n}\}$, where $F_{i}\in \mathcal{F}_{pr}(\textbf{D}),i=1,\ldots,n$. Then $F_{i}=\{x\in D~:~ z\sqsubseteq x~\mbox{for some }~ z\in F_{0i}\}$, where $F_{0i}=\uparrow a_{i}$ is a principal filter of the Boolean algebra $\textbf{D}_{\sqcap}$ generated by an atom $a_{i}$ of $\textbf{D}_{\sqcap}$. We claim that $A=F_{\vee_{i=1}^{n}a_{i}}$. Indeed, let $F_{k}\in A$,  $k\in \{1,\ldots, n\}$. As  $a_{k}\sqsubseteq a_{k}$, $a_{k}\in F_{k}$. By Proposition \ref{pro1} and Note \ref{finite-B-meet}, it follows that $a_{k}\sqsubseteq \vee_{i=1}^{n} a_{i}$. As $F_{k}$ is a filter, $\vee_{i=1}^{n} a_{i}\in F_{k}$, whence $F_{k}\in F_{\vee_{i=1}^{n}a_{i}}$.

Now let $F\in F_{\vee_{i=1}^{n}a_{i}}$. Since $F$ is a primary filter, $F=\{x\in D~:~ z\sqsubseteq x~\mbox{for some}~ z\in F_{0}\}$, where $F_{0}=\uparrow a$ is a principal filter of the Boolean algebra $\textbf{D}_{\sqcap}$ generated by the atom $a$. As $\vee_{i=1}^{n}a_{i}\in F$, $a\sqsubseteq \vee_{i=1}^{n}a_{i} $. So by Proposition \ref{pro1} and Note \ref{finite-B-meet}, $a\sqsubseteq_{\sqcap} \vee_{i=1}^{n}a_{i} $, which implies that $a=a_{k}$ for some $k\in \{1,\ldots, n\}$, as $a$ is an atom of $\textbf{D}_{\sqcap}$. Therefore $F\in A$. Hence $A=F_{\vee_{i=1}^{n}a_{i}}$.\\
(ii) For the case of $(\mathcal{I}_{pr}(\textbf{D}), \mathcal{J})$, similar to the proof in (i), one can show that any subset $B$ of $\mathcal{I}_{pr}(\textbf{D})$ has the form $B=I_{\wedge_{i=1}^{n}b_{i}}$ for some coatoms $b_{i}$ of the Boolean algebra $\textbf{D}_{\sqcup}$. 
%$(\mathcal{I}_{pr}(\textbf{D}), \mathcal{J})$ is a discrete topological space.
\end{proof}

\vskip 2pt
We next prove the isomorphism theorem for fully contextual dBas. For that, the following result will be used.
\begin{lemma}
\label{reqfcdBa}
{\rm For any dBa $\textbf{D}$ and $a\in D$, $F_{\neg a}=F_{\neg(a\sqcap a)}$.}
\end{lemma}
\begin{proof}
Let $F\in F_{\neg a}$. Then $\neg a\in F$. By Proposition \ref{pro1.5}(v) it follows that $a\sqcap a\sqsubseteq a$ and  Proposition \ref{pro2}(ii) gives $\neg a\sqsubseteq\neg(a\sqcap a)$. Therefore $\neg (a\sqcap a)\in F$, as $F$ is a filter. So $F_{\neg a}\subseteq F_{\neg(a\sqcap a)}$. Conversely, let $\neg(a\sqcap a)\in F$. As $F$ is a primary filter, $a\sqcap a\notin F$, which implies that $a\notin F$ (otherwise $a\sqcap a\in F$, as $F$ is a filter). Therefore $\neg a\in F$, as $F$ is a primary filter. This gives $F_{\neg(a\sqcap a)}\subseteq F_{\neg a}$. So $F_{\neg(a\sqcap a)}=F_{\neg a}$. 
\end{proof}

\begin{theorem}[\textbf{Representation theorem for fully contextual dBas}]
\label{iso-fullycxt dBa}
{\rm Any fully contextual dBa $\textbf{D}$ is isomorphic to $\underline{\mathfrak{R}}^{T}(\mathbb{K}_{pr}^{T}(\textbf{D}))$, the map $h:D\rightarrow \mathfrak{R}^{T}_{pr}(\mathbb{K}_{pr}^{T}(\textbf{D}))$ defined by $h(x):=(F_{\neg x}, I_{x})$.}
\end{theorem}
\begin{proof}
From Theorem  \ref{embedding of dBa} and the fact that a fully contextual dBa $\textbf{D}$ is a contextual dBa, it follows that $h$ is a dBa embedding. To complete the proof it remains to show that $h$ is surjective. For that, let $(X,Y)\in \mathfrak{R}^{T}_{pr}(\mathbb{K}_{pr}^{T}(\textbf{D})) $. Then $X$ is clopen in $ (\mathcal{F}_{pr}(\textbf{D}),\mathcal{T})$ and $Y$ is clopen in $ (\mathcal{I}_{pr}(\textbf{D}),\mathcal{J})$. Further, $X^{\blacksquare\lozenge}=Y^{\lozenge}$, which is equivalent to $X^{\blacksquare}=Y^{\lozenge\blacksquare}$, by Observation \ref{equivalent def of proto}. Now we consider the object oriented protoconcepts $(X, X^{\blacksquare}), (Y^{\lozenge}, Y)\in \mathfrak{R}^{T}(\mathbb{K}^{T}_{pr}(\textbf{D}))_{p}$. Then the following equations  hold. $(X, X^{\blacksquare})_{\sqcup}=(X, X^{\blacksquare})\sqcup (X, X^{\blacksquare})=(X^{\blacksquare\lozenge}, X^{\blacksquare})=(Y^{\lozenge},Y^{\lozenge\blacksquare})=(Y^{\lozenge}, Y)\sqcap (Y^{\lozenge}, Y)=(Y^{\lozenge}, Y)_{\sqcap}$. As $(X, X^{\blacksquare}), (Y^{\lozenge}, Y)\in \mathfrak{R}^{T}(\mathbb{K}^{T}_{pr}(\textbf{D}))_{p}$, $(X, X^{\blacksquare})=(F_{\neg x}, I_{x})$ and  $(Y^{\lozenge}, Y)=(F_{\neg b}, I_{b})$ for some $x\in D_{\sqcap}$ and $b\in D_{\sqcup}$ by Theorem \ref{characterization}.

From the above it follows that $(F_{x},I_{\neg x})_{\sqcup}=(F_{\neg(b\sqcup b)}, I_{b})_{\sqcap}$, which is equivalent to $(I_{\neg x}^{\lozenge}, I_{\neg x})=(F_{\neg(b\sqcup b)}, F_{\neg(b\sqcup b)}^{\blacksquare})$. Therefore $(F_{\neg(\neg x\sqcup\neg x)}, I_{\neg x})=(F_{\neg(b\sqcup b)}, I_{\neg\neg(b\sqcup b)})$, which is equivalent to $(F_{\neg(\neg x\sqcup\neg x)},\\ I_{\neg x\sqcup \neg x})=(F_{\neg\neg\neg(b\sqcup b)}, I_{\neg\neg(b\sqcup b)})$, as $I_{\neg x\sqcup \neg x}=I_{\neg x}$ by Lemma \ref{derivation}(iii) and $\neg\neg\neg(b\sqcup b)=\neg (b\sqcup b)$ by Proposition \ref{pro2}(ix). Therefore $h(\neg x\sqcup \neg x)= h(\neg\neg (b\sqcup b))$. As $h$ is injective, $\neg x\sqcup \neg x=\neg\neg(b\sqcup b)=(b\sqcup b)\sqcap (b\sqcup b)$ by Proposition \ref{pro2}(iii).  $\neg x\sqcap\neg x=\neg x\in D_{\sqcap}$ and $(b\sqcup b)\sqcup (b\sqcup b)=(b\sqcup b)\in D_{\sqcup }$ by Proposition \ref{pro2}(i and iv). Since $\textbf{D}$ is a fully contextual dBa, there exists a unique  $c\in D$ such that $c\sqcap c=\neg x$ and $c\sqcup c=b\sqcup b$. Therefore $X=F_{x}=F_{x\sqcap x}=F_{\neg\neg x}$ by  Lemma \ref{derivation}(iv) and Proposition \ref{pro2}(iii), which implies that $X=F_{\neg(c\sqcap c)}=F_{\neg c}$ by Lemma \ref{reqfcdBa}, and $Y=I_{b}=I_{b\sqcup b}=I_{c\sqcup c}=I_{c}$, by Lemma \ref{derivation}(iii). Therefore $h(c)=(F_{\neg c}, I_{c})=(X,Y)$, which implies that $h$ is surjective. Hence $h$ is a dBa isomorphism.
\end{proof}

\begin{corollary}
  {\rm If $\textbf{D}$ is a fully contextual dBa then $\mathbb{K}^{T}_{pr}(\textbf{D})$ is homeomorphic to $\mathbb{K}^{T}_{pr}( \mathfrak{R}^{T}(\mathbb{K}_{pr}^{T}(\textbf{D})) )$.}
  \end{corollary}
  \begin{proof}
  Follows directly from Theorems \ref{iso-fullycxt dBa} and \ref{db iso to cntx iso}.
  % From Theorem \ref{iso-fullycxt dBa}, it follows that $\textbf{D}$ is isomorphic  to  $\underline{\mathfrak{R}}^{T}(\mathbb{K}_{pr}^{T}(\textbf{D}))$. By Theorem \ref{db iso to cntx iso}, $\mathbb{K}^{T}_{pr}(\textbf{D})$ is homeomorphic to $\mathbb{K}^{T}_{pr}( \underline{\mathfrak{R}}^{T}(\mathbb{K}_{pr}^{T}(\textbf{D})))$.
  \end{proof}

Now we obtain the representation theorem for the class of pure dBas.
  \begin{theorem}[\textbf{Representation theorem for pure dBas}]
\label{RTDBA} 
{\rm Any pure dBa $\textbf{D}$ is isomorphic to $\mathcal{S}^{T}(\mathbb{K}_{pr}^{T}(\textbf{D}))$. 
%the map.  $h:D\rightarrow\mathfrak{S}^{T}(\mathbb{K}_{pr}^{T}(\textbf{D}))$  defined by $h(x):=(F_{\neg x},I_{x})$ for all $x\in D$, giving the required dBa isomorphism.
}
\end{theorem}

\begin{proof}
Since $\textbf{D}$ is pure, $\textbf{D}_{p}=\textbf{D}$ by Proposition \ref{puresub}. From Theorem \ref{embedding of dBa} it follows that $\textbf{D}$ is isomorphic to $\mathcal{S}^{T}(\mathbb{K}_{pr}^{T}(\textbf{D}))$.
%From Theorem \ref{embedding of dBa}, and the fact that a pure dBa is contextual (cf. Proposition \ref{order pure}) it follows that $h$ is an embedding. From Theorem \ref{characterization} and Proposition \ref{semicon} it follows that $\mathfrak{R}^{T}(\mathbb{K}_{pr}^{T}(\textbf{D}))_{p}\subseteq\mathfrak{S}^{T}(\mathbb{K}_{pr}^{T}(\textbf{D}))$. So  $\mathfrak{R}^{T}(\mathbb{K}_{pr}^{T}(\textbf{D}))_{p}=\mathfrak{S}^{T}(\mathbb{K}_{pr}^{T}(\textbf{D}))$, as $\mathfrak{S}^{T}(\mathbb{K}_{pr}^{T}(\textbf{D}))\subseteq \mathfrak{R}^{T}(\mathbb{K}_{pr}^{T}(\textbf{D}))_{p}$. Hence $h$ is surjective, which implies that $h$ is an isomorphism.
\end{proof}
  \begin{corollary}
  {\rm If $\textbf{D}$ is a pure dBa then $\mathbb{K}^{T}_{pr}(\textbf{D})$ is homeomorphic to $\mathbb{K}^{T}_{pr}( \mathcal{S}^{T}(\mathbb{K}_{pr}^{T}(\textbf{D})) )$.}
  \end{corollary}
  \begin{proof} By Theorems \ref{RTDBA} and \ref{db iso to cntx iso}.
  % From Theorem \ref{RTDBA}, it follows that $\textbf{D}$ is isomorphic  to  $\mathcal{S}^{T}(\mathbb{K}_{pr}^{T}(\textbf{D}))$. By Theorem \ref{db iso to cntx iso}, $\mathbb{K}^{T}_{pr}(\textbf{D})$ is homeomorphic to $\mathbb{K}^{T}_{pr}( \mathcal{S}^{T}(\mathbb{K}_{pr}^{T}(\textbf{D})))$.
  \end{proof}

\begin{corollary}[\textbf{Representation theorem for Boolean algebras}]
\label{boolean algebra}
{\rm %Let $\textbf{B}:=(B,\sqcup,\sqcap,^c,\top,\bot)$ be a Boolean algebra. Then 
Any Boolean algebra $\textbf{B}$ is  isomorphic to  $\mathcal{S}^{T}(\mathbb{K}_{pr}^{T}(\textbf{B}))$.} %and so $\underline{\mathfrak{H}}^{co}(\mathbb{K}_{pr}^{T}(\textbf{B}))$. }
\end{corollary}
\begin{proof} Follows from Theorem  \ref{RTDBA}, as any Boolean algebra is a pure dBa  (by Theorem \ref{Boolean and dBa}). 
\end{proof}
\noindent In \cite{howlader3} it is shown that every Boolean algebra $\textbf{B}$ is isomorphic to $\underline{\mathfrak{R}}^{T}(\mathbb{K}_{pr}^{T}(\textbf{B}))$. Note that this also follows from Theorem \ref{iso-fullycxt dBa} as a Boolean algebra $\textbf{B}$ is  fully contextual (Theorem \ref{Boolean and dBa}). Furthermore, this is in consonance with Corollary  \ref{boolean algebra}, as  $\underline{\mathfrak{R}}^{T}(\mathbb{K}_{pr}^{T}(\textbf{B}))= \mathcal{S}^{T}(\mathbb{K}_{pr}^{T}(\textbf{B}))$. In fact,  the set of all clopen object oriented concepts of $\mathbb{K}^{T}_{pr}(\textbf{B})$ coincides with both the set of clopen object oriented protoconcepts and that of clopen object oriented semiconcepts.

\vskip 3pt

We note here that, using the relations between $\blacksquare,\lozenge$ and $^\prime$ operators given in  Theorem \ref{property of box}(vi), one can rewrite the representation results for dBas  obtained in this section,  in terms of the algebra of protoconcepts and that of semiconcepts (cf. Notation \ref{semiproto}).

\section{ Duality results for dBas}
\label{CE} %In this section, we study duality theorems for dBas. 
Pure dBas as objects  and  dBa  isomorphisms  as morphisms constitute a category, denoted as $\textbf{PDBA}$. Fully contextual dBas  and  dBa isomorphisms  also form a category, denoted as $\textbf{FCDBA}$. 
On the other hand, abstraction of properties of the CTSCR $\mathbb{K}^{T}_{pr}(\textbf{D})$ corresponding to any  dBa $\textbf{D}$, leads us to the definition of {\it Stone contexts}. Stone contexts and CTSCR-homeomorphisms are observed to constitute a category, denoted as  $\textbf{Scxt}$.  
Moreover, relations between the collections of morphisms of these categories have already been obtained.  By Theorem \ref{iso}, for  fully contextual dBas $\textbf{D}$ and $\textbf{M}$, any map $f: D\rightarrow M$ is a dBa isomorphism if and only if  $f\vert_{D_{p}}$ is  a dBa isomorphism between the pure subalgebras $\textbf{D}_{p}$ and $\textbf{M}_{p}$ of $\textbf{D}$ and $\textbf{M}$ respectively.
 By Corollary \ref{idba hom set}, corresponding to any  dBa isomorphism  from dBa $\textbf{D}$ to dBa $\textbf{M}$, there is a  CTSCR- homeomorphism from $\mathbb{K}^{T}_{pr}(\textbf{M})$ to  $\mathbb{K}^{T}_{pr}(\textbf{D})$. A natural question then is to relate the  categories $\textbf{FCDBA}, \textbf{PDBA}$ and $\textbf{Scxt}$. It is  shown in this section that $\textbf{FCDBA}$ and  $\textbf{PDBA}$ are equivalent, whereas  $\textbf{PDBA}$ and $\textbf{Scxt}$ are dually equivalent.

 We divide the section into two subsections.  In Section \ref{SC} the goal is to define Stone contexts and give examples;  Section \ref{TCPdBaSC} presents the  relationships among the  categories $\textbf{FCDBA}, \textbf{PDBA}$ and $\textbf{Scxt}$.

\begin{notation}
{\rm For  objects \textbf{A}, \textbf{B} of a locally small category \cite{awodey2010category} $\mathfrak{C}$, the set of morphisms from \textbf{A} to \textbf{B} is denoted by $Hom_{\mathfrak{C}}(\textbf{A},\textbf{B})$. The categories we work on in this paper are all locally small.}
\end{notation}

 Let us briefly recall here, some properties of functors. Suppose $F$ is a functor from category $\mathfrak{C}$ to category $\mathfrak{E}$.  $F$ is  an  equivalence \cite{awodey2010category}   if and only if it is  faithful, full, and essential surjective. $F$ is faithful and full if and only if the restriction of $F$ on  $Hom_{\mathfrak{C}}(\textbf{A},\textbf{B})$ is a bijection from $Hom_{\mathfrak{C}}(\textbf{A},\textbf{B})$ to $Hom_{\mathfrak{E}}(F(\textbf{A}),F(\textbf{B}))$. To show $F$ is essential surjective, for each object $\textbf{X}$ of the category $\mathfrak{E}$ one finds an object $\textbf{A}$ in $\mathfrak{C}$   such that $F(\textbf{A})$ is isomorphic to $\textbf{X}$.

\subsection{\rm \textbf{Stone contexts}}
\label{SC}
 On abstraction of properties of the CTSCR $\mathbb{K}_{pr}^{T}(\textbf{D})$ corresponding to any  dBa $\textbf{D}$, one obtains the definition of a {\it Stone context} (Definition \ref{stone cxt} below). Let us note that a {\it totally disconnected space} is a topological space $(X,\tau)$ in which for any two $x,y\in X$  such that $x\neq y$, there is a clopen set $U$ in $(X,\tau)$ such that $x\in U$ and $y\notin U.$ We may remark here that, usually, a space $(G,\tau)$ is called totally disconnected if and only if every quasi-component consists of a single point. But for compact topological spaces, these two definitions coincide.
%\vskip 2pt  
Recall that
%\begin{definition}
a topological space $(X,\tau)$ is called a {\it Stone space} if and only if it is compact and totally disconnected \cite{davey2002introduction}.
\begin{definition}
\label{stone cxt}
{\rm A CTSCR $\mathbb{K}^{T}:=((G,\tau),(M,\rho),R)$ is called a {\it Stone context} if the following hold:
\begin{enumerate}[(a)]
\item $(G,\rho)$ and $(M,\tau)$ are Stone spaces,
\item $gRm,$  if for all $(A,B)\in\mathfrak{S}^{T}(\mathbb{K}^{T}),m\in B$ implies $g\in A$. 
\end{enumerate}}
\end{definition}

\begin{theorem}
\label{scxt}
{\rm For any dBa $\textbf{D}$, $\mathbb{K}^{T}_{pr}(\textbf{D}):=((\mathcal{F}_{pr}(\textbf{D}),\mathcal{T}),(\mathcal{I}_{pr}(\textbf{D}),\mathcal{J}),\nabla)$ is a Stone context.}
\end{theorem}
\begin{proof}
Follows from Theorems \ref{topcontext} and \ref{property of topspe},  and Corollary \ref{property-ston}.
\end{proof}

\noindent Apart from the example of Stone context provided by $\mathbb{K}^{T}_{pr}(\textbf{D}),$ we have the following.
\begin{example}
{\rm Any context $\mathbb{K}:=(X,Y,R)$, where X,Y are finite and $R\subseteq X\times Y$ is an arbitrary relation, can be trivially looked upon as a Stone context with finite discrete topology.}
\end{example}

\begin{example}
\label{example of stone context}
{\rm $\mathbb{K}_{+}^{T}$ given in Example \ref{exmple of topocxt} is a Stone context, when $(X,\tau_{1})$ and $(Y,\tau_{2})$ are non-empty Stone spaces. }
\begin{proof}
%To show $\mathbb{K}_{+}^{T}$ is Stone context 
It is  already established in Example \ref{exmple of topocxt} that $\mathbb{K}_{+}^{T}$ is a CTSCR.  We just verify condition (b) of Definition \ref{stone cxt}.
%$gRm$  holds if for all $(A,B)\in\mathfrak{S}^{T}(\mathbb{K}_{+}^{T}),m\in B$ implies $g\in A$. 
Let $g\in X$ and $m\in Y$ such that for all $(A,B)\in \mathfrak{S}^{T}(\mathbb{K}_{+}^{T}),$ $m\in B$ implies that $g\in A$.  If possible, assume that $g\cancel{R} m$. Then $m\notin C$ and so $m\in C^{c}$. Now we take $B=C^{c}$. From the proof of  Example \ref{exmple of topocxt}, it follows that $(\emptyset, C^{c})$ is an  object oriented semiconcept. Since $\emptyset$ and C  are clopen in $\tau_{1} ,\tau_{2}$ respectively, $(\emptyset,C^{c})\in\mathfrak{S}^{T}(\mathbb{K}_{+}^{T}) $. Therefore we get a clopen object oriented semiconcept $(\emptyset,C^{c})$ such that $m\in C^{c}$, but $g\notin \emptyset$. Thus we have a contradiction, and  $gRm$ must hold. Hence $\mathbb{K}_{+}^{T}$ is a Stone context. 
\end{proof}
\end{example}
\noindent A familiar example of a Stone space is  the Cantor set \cite{vallin2013elements}. One can construct a Stone context using the Cantor set. 
\begin{example}
{\rm Let $X=Y=C$, where $C$ is the Cantor set and $\tau$ be the subspace topology on $C$ induced by the  usual topology on $\mathbb{R}$. Let $\mathbb{K}^{T}:=((X,\tau),(Y,\tau),R),$ where $R(x)=[0,x]\cap C$ for all $x\in C$. Then $\mathbb{K}^{T}$ is a Stone context.}
\begin{proof}
$(X,\tau)$ and $(Y,\tau)$ both are Stone spaces. 
%To show $\mathbb{K}^{T}$ is a Stone context 
We show that $\mathbb{K}^{T}$ is a CTSCR, and check condition (b) of Definition \ref{stone cxt}.\\
%$gRm$  if for all $(A,B)\in\mathfrak{S}^{T}(\mathbb{K}^{T}),m\in B$ imply $g\in A$. 
To show that $R$ is continuous, let $O$ be an open set in $(Y,\tau)$. We verify  that $O^{\lozenge},O^{\square}$ are open in $(X,\tau)$. 
For any subset $A\subseteq C$,  $A\subseteq A^{\lozenge}$,  as for all $x\in A$, $([0,x]\cap C)\cap A\neq \emptyset$. Let $z\in O^{\lozenge}.$ Then $([0,z]\cap C)\cap O\neq\emptyset$. Therefore there exists  $x_{0}\in ([0,z]\cap C)\cap O.$ Now if $z=x_{0}$ then $z\in O\subseteq O^{\lozenge}.$ If $x_{0}< z$ then $n=\mbox{inf}(O)< z$, where $\mbox{inf}(O)$ is the infimum of $O$. We choose $\epsilon=z-n.$ Then $z\in (z-(z-n), z+(z-n))\cap C=(n,2z-n)\cap C\neq\emptyset$. Let $x\in (n,2z-n)\cap C.$ $([0,x]\cap C)\cap O\neq\emptyset$, as $n<x$. So $(n,2z-n)\cap C\subseteq O^{\lozenge}$. So in both cases, $z$ is an interior point of $O^{\lozenge}$. Thus $O^{\lozenge}$ is open in $(X,\tau)$. 
For the case of $O^{\square}$, let $z\in O^{\square}.$ This means $[0,z]\cap C\subseteq O$. So $z\in O=O^{*}\cap C$, where $O^{*}$ is an open set in $\mathbb{R}$. Therefore there exists an open interval $(a,b)$ in $\mathbb{R}$ such that $z\in (a,b)\cap C\subseteq O$. Let $x\in (a,b)\cap C$. Then either $a< x\leq z$ or $z\leq x< b$. If $x\leq z$ then $[0,x]\cap C\subseteq [0,z]\cap C\subseteq O.$ If $z< x< b$ then $[0,x]\cap C=([0,z]\cup[z,x])\cap C\subseteq O$, as $[0,z]\cap C, [z,x]\cap C$  are subsets of $O$. So $(a,b)\cap C\subseteq O^{\square}$. Therefore $O^{\square}$ is open in $(X,\tau)$. Hence $R$ is continuous.

\noindent   It is easy to see that $R^{-1}(y)=[y,1]\cap C$. To show that $R^{-1}$ is continuous, let $B$ be an open set in $(X,\tau)$, and consider  $B^{\blacklozenge}$ and $B^{\blacksquare}$. 
  %are open in $(Y,\tau)$. 
  Let $z\in B$. Then $([z,1]\cap C) \cap B\neq \emptyset$
and therefore $B\subseteq B^{\blacklozenge}$. Let $z\in B^{\blacklozenge}$. Then $([z,1]\cap C) \cap B\neq \emptyset$. Therefore there is  $z_{0}\in B$ such that $z\leq z_{0}$. If $z=z_{0}$ then $z\in B\subseteq B^{\blacklozenge}$, and if $z< z_{0}$ then $z\in (-z_{0},z_{0})\cap C\neq\emptyset$. Let $x\in (-z_{0},z_{0})\cap C $. Then $z_{0}\in ([x,1]\cap C)\cap B\neq\emptyset$, and so $(-z_{0},z_{0})\cap C\subseteq B^{\blacklozenge}$. Therefore in both cases, $z$ is an interior point of $B$. Hence $B^{\blacklozenge}$ is open in $(Y,\tau)$. %as $z$ is arbitrarily chosen. 
%To show $B^{\blacksquare}$ is open in $(Y,\tau)$ 
Next let $z\in B^{\blacksquare}$. Then $z\in [z,1]\cap C\subseteq B$, and there exists an open neighborhood such that $z\in (a,b)\cap C\subseteq B$. Now we can show that $(a,b)\cap C\subseteq B^{\blacksquare}$, and hence $B^{\blacksquare}$ is open in $(Y,\tau)$. Hence $R^{-1}$ is continuous. Therefore $\mathbb{K}^{T}$ is a CTSCR. 

Now let $g\in C$ and $m\in C$ such that for all $(A,B)\in \mathfrak{S}^{T}({\mathbb{K}^{T}})$, $m\in B$ implies that $g\in A$. If possible,  let $g\cancel{R}m$, that is $m\notin R(g)$. So $g<m$. As $C$ is nowhere dense in $\mathbb{R}$, there exists  $x\in (g,m)$ such that $x\notin C$. We choose a real number $a$ such that $1< a$. Then $O=[x,a]\cap C=(x,a)\cap C$ is a clopen set in $C$ such that $m\in O$, but $g\notin O^{\lozenge}$, as $g< x$. This is a contradiction, as $(O^{\lozenge},O)\in \mathfrak{S}^{T}(\mathbb{K^{T}}).$ %Therefore our assumption was wrong, which 
Therefore  $gRm$. % $\mathbb{K}^{T}$ is a Stone context.
\end{proof}
\end{example}

It is easy to show
%\begin{proof}
%Proof of the proposition is straightforward so we leave it.
%\end{proof}
\begin{proposition}
\label{catScxt}
{\rm Stone contexts and CTSCR-homeomorphisms form a category. It is denoted by $\textbf{Scxt}$.}
\end{proposition}
\subsection{\rm \textbf{The categorical duality between dBas and Stone contexts}}
\label{TCPdBaSC}
In this section, the categories $\textbf{FCDBA}$, $\textbf{PDBA}$, $\textbf{Scxt}$ and their relationships are presented.
% We divide the section into two subsections. Subsection \ref{FCDBA-duality} presents the categories $\textbf{FCDBA}$ and  $\textbf{PDBA}$ and establishes  the categorical equivalence of $\textbf{FCDBA}$ and  $\textbf{PDBA}$. Subsection \ref{dualityresult} presents the category  $\textbf{Scxt}$  and establishes that it is dually equivalent to $\textbf{PDBA}$.
\subsubsection{\rm \textbf{Equivalence of $\textbf{FCDBA}$ and $\textbf{PDBA}$}}
\label{FCDBA-duality}
%Proposition \ref{puresub}, Theorems \ref{clopen dba}, and \ref{topcontext}  gives a correspondence between (fully contextual)dBas and pure dBas. A connection between an ismorphism between two pure dBa and an ismorphism between two fully contextual dBa can be established using Theorem \ref{iso}. These two facts have motivated us to extend the study in the direction of category theory.  

It is  straightforward to show 

\begin{proposition}
{\rm \noindent
\begin{enumerate}[(i)]
\item Fully contextual dBas considered as objects and  dBa isomorphisms as morphisms yield a category. It is denoted  by $\textbf{FCDBA}$.
\item Pure dBas considered as objects and dBa isomorphisms as morphisms constitute a category. It is denoted  by $\textbf{PDBA}$.
\end{enumerate}}
\end{proposition}

\noindent %Now we investigate the relationship between the two categories. 
A natural correspondence $G$  from  $\textbf{FCDBA}$  to $\textbf{PDBA}$ is obtained as follows: \vskip 2pt
\noindent %$G:\textbf{FCDBA}\longrightarrow\textbf{PDBA}.$
 $G(\textbf{D}):=\textbf{D}_{p}$, for any object $\textbf{D}$ in $\textbf{FCDBA}$ and \\
  $G(f):=f\vert_{D_{p}}$,  for any $f\in Hom_{\textbf{FCDBA}}(\textbf{D},\textbf{M})$.
  \vskip 3pt 
\noindent  $G$ is well-defined and a covariant functor, using Proposition \ref{puresub} and Theorem \ref{iso}.  $G$ turns out to be an equivalence between \textbf{FCDBA} and \textbf{PDBA}, as we prove  below. 
\begin{theorem}
\label{algebricdualty}
{\rm \textbf{FCDBA} is equivalent to \textbf{PDBA}.}
\end{theorem}
\begin{proof}
%Let $\textbf{D}_{1},\textbf{D}_{2},$ and $\textbf{D}_{3}$ belong to $Obj(\textbf{FCDBA})$, and $f\in Hom_{\textbf{FCDBA}}(\textbf{D}_{1},\textbf{D}_{2})$, $g\in Hom_{\textbf{FCDBA}}(\textbf{D}_{2},\textbf{D}_{3})$. Then $G(g\circ f)=(g\circ f)\vert_{D_{p}}=g\vert_{D_{p}}\circ f\vert_{D_{p}}=G(g)\circ G(f)$. Let $\textbf{D}\in Obj(\textbf{FCDBA})$ and $1_{D}:D\rightarrow D$ be the identity morphism. Then $G(1_{D})=1_{D_{p}}$. Therefore $G$ is covariant functor. 
Due to Theorem \ref{iso}, $G$ is  a faithful and full functor. 
Now let $\textbf{D}$ be a pure dBa. $\mathbb{K}^{T}_{pr}(\textbf{D})$ is a CTSCR by Theorem \ref{topcontext}. So $\underline{\mathfrak{R}}^{T}(\mathbb{K}^{T}_{pr}(\textbf{D}))$  is a fully contextual dBa by Theorem \ref{clopen dba}(i). Applying  Proposition \ref{largest sub algebra} to $\underline{\mathfrak{R}}^{T}(\mathbb{K}^{T}_{pr}(\textbf{D}))_{p}$ and $\mathcal{S}^{T}(\mathbb{K}^{T}_{pr}(\textbf{D}))$, we get
  $\underline{\mathfrak{R}}^{T}(\mathbb{K}^{T}_{pr}(\textbf{D}))_{p}=\mathcal{S}^{T}(\mathbb{K}^{T}_{pr}(\textbf{D}))$ for any dBa $\textbf{D}$. Therefore $G(\underline{\mathfrak{R}}^{T}(\mathbb{K}^{T}_{pr}(\textbf{D})))=\mathcal{S}^{T}(\mathbb{K}^{T}_{pr}(\textbf{D}))$, which is isomorphic to $\textbf{D}$ by Theorem \ref{RTDBA}. So $G$ is essential surjective.
\end{proof}
\subsubsection{\rm \textbf{Dual equivalence of $\textbf{PDBA}$ and $\textbf{FCDBA}$ with $\textbf{Scxt}$ }}
\label{dualityresult}
%* A correspondence between pure dBas and Stone contexts can be established by using Theorem \ref{scxt}.  As a consequence of  Corollary \ref{idba hom set},  there is also a correspondence between  dBa isomorphisms and CTSCR-homeomorphisms. These two observation have motivated us to study a categorical relationship between pure dBas and Stone context.  The following propositions are straightforward to derive.* 

%For Proposition \ref{catScxt} below, we make use of Theorem \ref{composition of context}.
% Theorem \ref{pdba hom set} suggest the following study. Theorem \ref{pdba hom set} say that if we take surjective dBa homomorphism $h$  from a pure dBa $\textbf{D}$ to a pure dBa $\textbf{M}$ then we get a particular type of homomorphism on the cts $(\alpha_{h},\beta_{h})$ from $\mathbb{K}^{T}_{pr}(\textbf{M})$ to $\mathbb{K}^{T}_{pr}(\textbf{M})$.

%\begin{proof}
%From Theorem \ref{composition of context} it follows that composition is closed under the set of morphism and the rest of the proof is straightforward.
%\end{proof}

\noindent %Now we investigate the relationship between the two categories. 
Recall the maps $\alpha_{f}$ and $\beta_{f}$ corresponding to any dBa homomorphism $f$ (cf. Definition \ref{alphbetah}, Section \ref{RT}).
A natural correspondence $F$ is defined  from  $\textbf{PDBA}$  to $\textbf{Scxt}$ by: \vskip 2pt
\noindent %$F:\textbf{PDBA}\longrightarrow\textbf{Scxt}.$
 $F(\textbf{D}):=\mathbb{K}^{T}_{pr}(\textbf{D})$, for any object $\textbf{D}$ in $\textbf{PDBA}$ and \\
  $F(f):=(\alpha_{f},\beta_{f})$,  for any $f\in Hom_{\textbf{PDBA}}(\textbf{D},\textbf{M})$.
  \vskip 3pt 
\noindent   $F$ is a well-defined contravariant functor, using Theorem \ref{db iso to cntx iso} and Corollary \ref{idba hom set}.  We  show in Theorem \ref{duality} below that $F$ is an  equivalence between $\textbf{PDBA}$ and $\textbf{Scxt}^{op}$, the opposite category of $\textbf{Scxt}$. 
%In other words, $F$ will be proved to be a contravariant functor that is an equivalence. 
\vskip 2pt
Now to show that $F$ is essential surjective, for each Stone context $\mathbb{K}^{T}$ one needs to  find a pure dBa $\textbf{D}$ such that $\mathbb{K}^{T}_{pr}(\textbf{D})$ is homeomorphic to $\mathbb{K}^{T}$.
 $\textbf{D}$ is, expectedly, the pure dBa $\mathcal{S}^{T}(\mathbb{K}^{T})$. We construct a CTSCR-homeomorphism  from   $\mathbb{K}^{T}$ to $\mathbb{K}^{T}_{pr}(\mathcal{S}^{T}(\mathbb{K}^{T}))$ as follows.\vskip 2pt
\noindent For any CTSCR $\mathbb{K}^{T}:=((G,\tau_{1}),(M,\tau_{2}),R)$, define the  functions $k_{1}:G \ra \mathcal{P}(\mathfrak{S}^{T}(\mathbb{K}^{T}))$ and $k_{2}: M \ra \mathcal{P}(\mathfrak{S}^{T}(\mathbb{K}^{T}))$  given by \\
$k_{1}(g):=\{(A,B)\in\mathfrak{S}^{T}(\mathbb{K}^{T}): g\notin A\}$ for any $g\in G$, and\\
$k_{2}(m):=\{(A,B)\in\mathfrak{S}^{T}(\mathbb{K}^{T}):m\in B\}$
for any $m\in M$.
%where the functions $l_{1}:G \ra \mathcal{P}(\mathfrak{R}^{T}(\mathbb{K}^{T}))$ and $l_{2}: M \ra \mathcal{P}(\mathfrak{R}^{T}(\mathbb{K}^{T}))$  given by $l_{1}(g):=\{(A,B)\in\mathfrak{R}^{T}(\mathbb{K}^{T}): g\notin A\}$ for any $g\in G$, and $l_{2}(m):=\{(A,B)\in\mathfrak{R}^{T}(\mathbb{K}^{T}):m\in B\}$ for any $m\in M$.*
\begin{proposition}
\label{dualitypro}
{\rm 
\noindent 
\begin{enumerate}[{(i)}]
\item $k_{1}(g)$ is a primary filter in $\mathcal{S}^{T}(\mathbb{K}^{T})$, for each $g\in G$.
\item $k_{2}(m)$ is a primary ideal in $\mathcal{S}^{T}(\mathbb{K}^{T})$, for each $m\in M$.
%\item $l_{1}(g)$ is a primary filter in $\underline{\mathfrak{R}}^{T}(\mathbb{K}^{T})$, for each $g\in G$.
%\item $l_{2}(m)$ is a primary ideal in $\underline{\mathfrak{R}}^{T}(\mathbb{K}^{T})$, for each $m\in M$.
\end{enumerate}}
\end{proposition}

\begin{proof}(i) 
%For each $g\in G$ 
$k_{1}(g)$ is a proper subset of $\mathfrak{S}^{T}(\mathbb{K}^{T})$, as $\bot=(G,M)\notin k_{1}(g)$. Let $x:=(A,B)$ and $y:=(C,D)$ be two elements in $k_{1}(g)$. Then $g\notin A$ and $g\notin C$. $x\sqcap y=(A\cup C,(A\cup C)^{\blacksquare})\in k_{1}(g)$, as $g\notin A\cup C$. Next let $z:=(E,F)\in\mathfrak{S}^{T}(\mathbb{K}^{T}) $ such that  $x\sqsubseteq z$.
%, where $x=(A,B)\in k_{1}(g)$. Therefore 
By Proposition \ref{order object-semi}, $E\subseteq A$ and so $g\notin E$ (as $g\notin A$). Hence $z\in k_{1}(g)$. Further, for any $x:=(A,B)\in \mathfrak{S}^{T}(\mathbb{K}^{T})$, either $g\notin A^{c}$ or $g\notin A$. So $\neg x\in k_{1}(g)$ or $x\in k_{1}(g)$. Therefore $k_{1}(g)$ is a primary filter for each $g\in G$.\\
%{\it Proof of (ii):}For each $m\in M$ $k_{2}(m)$ is proper subset of $\mathfrak{S}^{T}(\mathbb{K}^{T})$ as $\top=(\emptyset,\emptyset)\notin k_{2}(m)$. Let $x=(A,B)$ and $y=(C,D)$ be two elements in $k_{2}(m)$ then $m\in B$ and $g\in D$. Then $x\sqcup y=((B\cap D)^{\lozenge},B\cap D)\in k_{1}(g)$ as $m\in B\cap D$. Now let $z=(E,F)\in\mathfrak{S}^{T}(\mathbb{K}^{T}) $ such that  $z\sqsubseteq x$. Therefore by Proposition \ref{order object-semi} we have $B\subseteq F$ and so $m\in F$ (as $m\in B$). Hence $z\in k_{2}(m)$. Now let $x=(A,B)\in \mathfrak{S}^{T}(\mathbb{K}^{T})$ then $m\in  B^{c}$ or $m\in B$ and so $\lrcorner x\in k_{2}(m)$ or $x\in k_{2}(m)$. Therefore $k_{2}(m)$ is a primary ideal for each $m\in M$
(ii) is proved dually.
\end{proof}
\begin{theorem}
\label{isom for boolean ctx}
{\rm Let $\mathbb{K}^{T}:=((G,\tau_{1}),(M,\tau_{2}),R)$ be a Stone context. Then $K:=(k_{1},k_{2})$  is a CTSCR-homeomorphism  from $\mathbb{K}^{T}$ to $\mathbb{K}^{T}_{pr}(\mathcal{S}^{T}(\mathbb{K}^{T}))$.}
\end{theorem}

\begin{proof}
One needs to show the following.\noindent 
\begin{enumerate}[{(i)}]
\item $(k_{1},k_{2})$ is a context homomorphism, that is, for any $g\in G$ and $m\in M$, $gRm$ if and only if $k_{1}(g)\nabla k_{2}(m)$. 
%\begin{itemize}
\item $k_{1}$ is a homeomorphism from $(G,\tau_{1})$ to $(\mathcal{F}_{pr}(\mathcal{S}^{T}(\mathbb{K}^{T})),\mathcal{T})$.
\item $k_{2}$ is a homeomorphism from $(M,\tau_{2})$ to $(\mathcal{I}_{pr}(\mathcal{S}^{T}(\mathbb{K}^{T})),\mathcal{J})$.
\end{enumerate}
\noindent (i) Let $g\in G, m\in M$ and $gRm$. If possible, let $ k_{1}(g)\cap k_{2}(m)\neq\emptyset$ and $(A,B)\in k_{1}(g)\cap k_{2}(m)$. Then $g\notin A$ and $m\in B$. So $g\in B^{\lozenge}$, as $gRm$. Since  object oriented semiconcepts are also object oriented protoconcepts, $A^{\blacksquare\lozenge}=B^{\lozenge}$,  which implies that $g\in A^{\blacksquare\lozenge}$. $ A^{\blacksquare\lozenge}\subseteq A $ by Theorem \ref{property of box}(ix). So $g\in A$, which  a contradiction. Hence $k_{1}(g)\cap k_{2}(m)=\emptyset$.\\
\noindent For the converse, let us assume that $k_{1}(g)\nabla k_{2}(m)$, that is $k_{1}(g)\cap k_{2}(m)=\emptyset$. Then for all $(A,B)\in \mathfrak{S}^{T}(\mathbb{K}^{T})$, $m\in B$ implies that $g\in A$ -- otherwise, $k_{1}(g)\cap k_{2}(m)\neq\emptyset$. Hence $gRm$, as $\mathbb{K}^{T}$ is a Stone context.\\
  (ii)
 %Now we show that $k_{1}$ is a homeomorphism from $(G,\tau_{1})$ to $(\mathcal{F}_{pr}(\mathcal{S}^{T}(\mathbb{K}^{T})),\mathcal{T})$.
It is given that $(G,\tau_{1})$ and $(\mathcal{F}_{pr}(\mathcal{S}^{T}(\mathbb{K}^{T})),\mathcal{T})$ are two compact Hausdorff topological spaces. Therefore by Theorem \ref{fun homeomorphism}, it is sufficient to show that $k_{1}$ is a continuous bijection. 

\noindent $k_{1}$ is injective: since $(G,\tau_{1})$ is a totally disconnected space, for any two $g_{1},g_{2}\in G$ with $g_{1}\neq g_{2}$, there exists a clopen set $A\subseteq G$ such that $g_{2}\in A$ and $g_{1}\notin A$. Therefore $(A,A^{\blacksquare})\in k_{1}(g_{1})$ and $(A,A^{\blacksquare})\notin k_{1}(g_{2})$, which imply $k_{1}(g_{1})\neq k_{1}(g_{2})$. Hence $k_{1}$ is injective. \\
$k_{1}$ is continuous: for this, we  first show that $k^{-1}_{1}(F_{c})$ is open in $(G,\tau_{1})$ for any  open set $F_{c}$ in $ (\mathcal{F}_{pr}(\mathcal{S}^{T}(\mathbb{K}^{T})),\mathcal{T})$. Let $c:=(A,B)\in \mathfrak{S}^{T}(\mathbb{K}^{T}).$ Then 
$ k_{1}^{-1}(F_{c})=\{g\in G:k_{1}(g)\in F_{c}\}=\{g\in G: c\in k_{1}(g)\}=\{g\in G:g\notin A\}=A^{c}$, which is open in $(G,\tau_{1})$. If $O$ is an open set in $(\mathcal{F}_{pr}(\mathcal{S}^{T}(\mathbb{K}^{T})),\mathcal{T})$ then by Note \ref{clopen subbase}, $O=\cup_{j\in J}\cap_{a\in D_{j}}F_{\neg a},$ where $D_{j},j \in J,$ is a finite subset of $ \mathfrak{S}^{T}(\mathbb{K}^{T}),$  $J$ being  an index set.  $k_{1}^{-1}(O)=k_{1}^{-1}(\cup_{j\in J}\cap_{a\in D_{j}}F_{\neg a})=\cup_{j\in J}\cap_{a\in D_{j}}k_{1}^{-1}(F_{\neg a})$. Therefore $k_{1}^{-1}(O)$ is open in $(G,\tau_{1})$, which implies that $k_{1}$ is continuous.\\
Lastly,  let us show that $k_{1}$ is surjective.  $k_{1}(G)$ is compact in $(\mathcal{F}_{pr}(\mathcal{S}^{T}(\mathbb{K}^{T})),\mathcal{T})$, since $k_{1}$ is continuous. Therefore $k_{1}(G)$ is closed in $(\mathcal{F}_{pr}(\mathcal{S}^{T}(\mathbb{K}^{T})),\mathcal{T})$. We now prove  that $k_{1}(G)$  is dense in $(\mathcal{F}_{pr}(\mathcal{S}^{T}(\mathbb{K}^{T}),\mathcal{T})$, as then we would have $k_{1}(G)=\mathcal{F}_{pr}(\mathcal{S}^{T}(\mathbb{K}^{T})$. To establish this, we show that any non-empty  open set $O$ in $(\mathcal{F}_{pr}(\mathcal{S}^{T}(\mathbb{K}^{T}), \mathcal{T})$ intersects $k_{1}(G)$. 
%Let  $O$ be a non-empty  openset in $(\mathcal{F}_{pr}(\mathcal{S}^{T}(\mathbb{K}^{T}), \mathcal{T})$. 
%**complete mess -- you are using $O_1$ below, then why are you writing this  form of $O$??** 
By Note \ref{clopen subbase}, $O
 =\cup_{j\in J}\cap_{a\in D_{j}}F_{\neg a},$ where $D_{j},j \in J,$ is a finite subset of $ \mathfrak{S}^{T}(\mathbb{K}^{T}).$ As $O$ is non-empty, there is $j \in J$ such that the open set $O_{j}=\cap_{a\in D_{j}}F_{\neg a}$ is non-empty. It is then 
 %for any open set $O$, 
 %$O\cap k_{1}(G)\neq\emptyset,$ 
  sufficient to show that for all such non-empty  $O_{j}$, $O_{j}\cap k_{1}(G)\neq\emptyset$.  Let $D_j:= \{c_1,\ldots,c_n\}$, where $c_{i}:=(A_{i},B_{i})\in  \mathfrak{S}^{T}(\mathbb{K}^{T})$ for $i=1,\ldots,n$.   
   % for each non-empty open set $O_{j}=\cap_{i=1}^{n}F_{\neg c_{i}}$, $j\in J,$
% $O^{c}$ is basic closed set in $(\mathcal{F}_{pr}(\mathcal{S}^{T}(\mathbb{K}^{T}), \mathcal{T})$. Therefore $O^{c}=\cup_{i=1}^{n}F_{c_{i}}$ 
%and $c_{i}:=(A_{i},B_{i})\in D_{j}\subseteq \mathfrak{S}^{T}(\mathbb{K}^{T})$ for $i=1,\ldots,n$,  $O_{j}\cap k_{1}(G)\neq\emptyset$.  *
%we have the following.
Now $O_{j}=\cap_{i=1}^{n}F_{\neg c_{i}}=F_{\neg c_{1}\sqcap  \ldots\sqcap\neg c_{n}}=F_{\neg (c_{1}\vee\ldots\vee c_{n})}=\{F\in \mathcal{F}_{pr}(\mathcal{S}^{T}(\mathbb{K}^{T}): c_{1} \vee \ldots\vee c_{n}\notin F\}.$  $c_{1}\vee\ldots\vee c_{n}=\neg((A_{1}^{c},A_{1}^{c\blacksquare})\sqcap\ldots\sqcap (A_{n}^{c},A_{n}^{c\blacksquare}) )=(\cap A_{i},(\cap A_{i} )^{\blacksquare})$. Then $\cap A_{i}\neq\emptyset$, otherwise $c_{1}\vee \ldots\vee c_{n}=(\emptyset,\emptyset^{\blacksquare})=\top\sqcap\top=\neg \bot$ implies that $O_{j}$ is empty (as every primary filter contains $\neg \bot$), giving a contradiction. So either $\cap A_{i}=G$ or $\cap A_{i}\subsetneq G$. Now if $\cap A_{i}= G$ then  $c_{1}\vee\ldots\vee c_{n}=(G,M)=\bot$, and so $O_{j}=\mathcal{F}_{pr}(\mathcal{S}^{T}(\mathbb{K}^{T})$. Hence $k_{1}(G)\cap O_{j}\neq \emptyset$. If $\cap A_{i}$ is a proper subset of $G$ then there exists  $g\in G$ such that $g\notin \cap A_{i}$.  We consider $k_{1}(g)=\{(A,B)\in\mathfrak{S}^{T}(\mathbb{K}^{T}):g\notin A\} $.  Then $c_{1}\vee \ldots\vee c_{n} \notin k_{1}(g)$. So $k_{1}(g)\in O_{j}$, and $O_{j}\cap k_{1}(G)\neq\emptyset$. \\
%Thus one obtains that  $k_{1}(G)$ is dense in $(\mathcal{F}_{pr}(\mathcal{S}^{T}(\mathbb{K}^{T}), \mathcal{T})$, whence $k_{1}(G)=\mathcal{F}_{pr}(\mathcal{S}^{T}(\mathbb{K}^{T})$.\\
 (iii) Dually one can show that $k_{2}$ is a homeomorphism from $(M,\tau_{2})$ to $(\mathcal{I}_{pr}(\mathcal{S}^{T}(\mathbb{K}^{T})),\mathcal{J})$.
\end{proof}

\noindent To prove the main theorem (Theorem \ref{duality}), we need two more results.  Recall  the map $h$ giving  (Representation) Theorem \ref{RTDBA}, namely $h:D\rightarrow\mathfrak{S}^{T}(\mathbb{K}_{pr}^{T}(\textbf{D}))$  defined by $h(x):=(F_{\neg x},I_{x})$ for all $x\in D$. 
\begin{proposition}
\label{diagram comut}
{\rm Let $\textbf{D}_{1}$, $\textbf{D}_{2}$ be pure dBas. For a dBa isomorphism $f$ from $\textbf{D}_{1}$ to  $\textbf{D}_{2}$, the following diagram commutes
\begin{center}
\begin{tikzcd}
\textbf{D}_{1}\arrow[r, "f" ]\arrow[d, "h_{1}"]& \textbf{D}_{2}\arrow[d,"h_{2}"] \\\mathcal{S}^{T}(\mathbb{K}^{T}_{pr}(\textbf{D}_{1}))\arrow[r, "f_{\alpha_{f}\beta_{f}}" ]& \mathcal{S}^{T}(\mathbb{K}^{T}_{pr}(\textbf{D}_{2}))
\end{tikzcd}
\end{center}
that is, $h_{2}\circ f=f_{\alpha_{f}\beta_{f}}\circ h_{1}$, where $h_{1},h_2$ are as in Theorem \ref{RTDBA}. }
\end{proposition}
\begin{proof}
Let $x\in D_{1}$. Then $h_{2}\circ f(x)=h_{2}(f(x))=(F_{\neg f(x)},I_{f(x)})$ and $f_{\alpha_{f}\beta_{f}}\circ h_{1}(x)=f_{\alpha_{f}\beta_{f}}(h_1(x))=f_{\alpha_{f}\beta_{f}}(F_{\neg x},I_{x})=(\alpha^{-1}_{f}(F_{\neg x}),\beta^{-1}_{f}(I_{x}))=(F_{f(\neg x)},I_{f(x)})$ -- the last as we have shown in the proof of Proposition \ref{cntxmor} that $\alpha^{-1}_{f}(F_{x})=F_{f(x)}$ and $\beta^{-1}_{f}(I_{x})=I_{f(x)}$, for any $x$ in $ D_{1}$. Since $f$ is a dBa homomorphism, 
$f_{\alpha_{f}\beta_{f}}\circ h_{1}(x)=(F_{\neg f(x)},I_{f(x)})$.
\end{proof}
%Recall Definition \ref{topoisocxt}  giving the map $f_{\alpha\beta}$ corresponding to a CTSCR-isomorphism $(\alpha,\beta)$.
\begin{proposition}
\label{dBa-hom and ctx homeo}
{\rm Let $\mathbb{K}_{1}^{T}:=((G_{1},\tau_{1}),(M_{1},\rho_{1}),R_{1})$ and $\mathbb{K}_{2}^{T}:=((G_{2},\tau_{2}),(M_{2},\rho_{2}),R_{2})$ be Stone contexts, and $f_{1}:=(\alpha_{1},\beta_{1})$ and $f_{2}:=(\alpha_{2},\beta_{2})$ be CTSCR-homeomorphisms from  $\mathbb{K}_{1}^{T}$ to $\mathbb{K}_{2}^{T}$. If $f_{\alpha_{1}\beta_{1}}=f_{\alpha_{2}\beta_{2}}$ then $f_{1}=f_{2}.$}
\end{proposition}
\begin{proof}
%Let $f_{1}$ and $f_{2}$ be two homeomorphism of cts from $\mathbb{K}^{T}_{1}$ to $\mathbb{K}^{T}_{2}$. We assume that  
Let $f_{\alpha_{1}\beta_{1}}=f_{\alpha_{2}\beta_{2}}$. %Now we will show that $f_{1}=f_{2}$. 
If possible, suppose $f_{1}\neq f_{2}$. So either $\alpha_{1}\neq \alpha_{2}$ or $\beta_{1}\neq\beta_{2}$. Without loss of  generality, suppose $\alpha_{1}\neq\alpha_{2}$. Then there exists $a\in G_{1}$ such that $\alpha_{1}(a)\neq \alpha_{2}(a)$. There also exists a clopen set $A$ in $(G_{2},\tau_{2})$ such that $\alpha_{1}(a)\in A$, but $\alpha_{2}(a)\notin A$, as $(G_{2},\tau_{2})$ is a Stone space. Therefore $\alpha_{1}^{-1}(A)\neq\alpha_{2}^{-1}(A)$. Now consider the clopen object oriented semiconcept $x:=(A, A^{\blacksquare})$ in $\mathfrak{S}^{T}(\mathbb{K}^{T}_{2})$. Then $f_{\alpha_{1}\beta_{1}}((A,A^{\blacksquare}))=(\alpha^{-1}_{1}(A),\beta^{-1}_{1}(A^{\blacksquare}))$, and $f_{\alpha_{2}\beta_{2}}((A,A^{\blacksquare}))=(\alpha^{-1}_{2}(A),\beta^{-1}_{2}(A^{\blacksquare}))$. Therefore $f_{\alpha_{1}\beta_{1}}((A,A^{\blacksquare}))\neq f_{\alpha_{2}\beta_{2}}((A,A^{\blacksquare}))$, as $\alpha^{-1}_{1}(A)\neq \alpha^{-1}_{2}(A)$, which is a contradiction. This gives $f_{1}=f_{2}$. 
\end{proof}

We now obtain
\begin{theorem}
\label{duality}
{\rm \textbf{PDBA} is equivalent to $\textbf{Scxt}^{op}$. 
}
\end{theorem}
\begin{proof}
We must show that $F$ is (i) faithful, (ii)  full and (iii) essential surjective.\\
(i) %Next we show that $F$ is faithful. 
Let $\textbf{D}_{1},\textbf{D}_{2}\in Obj(\textbf{PDBA})$,
%Consider $F: Hom_{\textbf{PDBA}}(\textbf{D}_{1},\textbf{D}_{2})\rightarrow  Hom_{\textbf{Scxt}}(F(\textbf{D}_{1}),F(\textbf{D}_{2})),$ the restriction of the functor on $Hom_{\textbf{PDBA}}(\textbf{D}_{1},\textbf{D}_{2})$. This is  injective. Indeed, 
$f,g\in Hom_{\textbf{PDBA}}(\textbf{D}_{1},\textbf{D}_{2})$ and $F(f)=F(g)$. If possible, suppose $f\neq g$. Then there exists $x\in D_{1}$ such that $f(x)\neq g(x)$. By  Proposition \ref{order pure}, either $f(x)\sqcap f(x)\neq g(x)\sqcap g(x) $ or $f(x)\sqcup f(x)\neq g(x)\sqcup g(x)$. Let us assume that $f(x)\sqcap f(x)\neq g(x)\sqcap g(x)$. As $f(x)\sqcap f(x),g(x)\sqcap g(x)\in \textbf{D}_{2\sqcap}$,  either $f(x)\sqcap f(x)\not\sqsubseteq_{\sqcap} g(x)\sqcap g(x) $ or  $g(x)\sqcap g(x)\not\sqsubseteq_{\sqcap}  f(x)\sqcap f(x)$. Suppose $f(x)\sqcap f(x)\not\sqsubseteq_{\sqcap} g(x)\sqcap g(x).$ Then there exists a prime filter $F_{0}$ in $\textbf{D}_{2\sqcap} $ (a Boolean algebra)  such that $f(x)\sqcap f(x)\in F_{0}$ and $g(x)\sqcap g(x)\notin F_{0} $.
Due to Lemma \ref{lema1},  $F_{0}=F_{1}\cap D_{2\sqcap}$ where  $F_{1}=\{a\in D_{2}:y\sqsubseteq a~\mbox{for some} ~ y\in F_{0}\}$ is a filter in $\textbf{D}_{2}$. By Proposition \ref{comparison of two ideal} it follows that $F_{1}$ is a primary filter of $\textbf{D}_{2}$. $f(x)\in F_{1}$ as $f(x)\sqcap f(x)\sqsubseteq f(x)$ and $g(x)\notin F_{1}$ -- otherwise $g(x)\sqcap g(x)\in F_{1}\cap D_{2\sqcap}=F_{0}$. Therefore $x\in f^{-1}(F_{1})$ and $x\notin g^{-1}(F_{1})$. So $\alpha_{f}\neq \alpha_{g}.$
  
 \noindent If  $g(x)\sqcap g(x)\not\sqsubseteq_{\sqcap}  f(x)\sqcap f(x)$, then similarly we can find a primary filter $F_{2}$ of $\textbf{D}_{2}$ such that $x\in g^{-1}(F_{2})$ and $x\notin f^{-1}(F_{2})$. So $\alpha_{f}\neq\alpha_{g}$. Therefore in both cases,  $F(f)\neq F(g)$.
 
 \noindent  If $f(x)\sqcup f(x)\neq g(x)\sqcup g(x)$, then dually we can show that $\beta_{f}\neq \beta_{g}$. So $F(f)\neq F(g)$ in this case as well. 
 %Therefore our assumption was wrong and 
 Hence $f=g$, and $F$ is faithful.

\noindent (ii) 
Let $h:=(\alpha,\beta)\in Hom_{\small\textbf{Scxt}}(G(\textbf{D}_{2}), G(\textbf{D}_{1}))$, where $\textbf{D}_{1}$ and $\textbf{D}_{2}$ are pure dBas. By Theorem \ref{injecmor}, $f_{\alpha\beta}$ is a dBa isomorphism from $\mathcal{S}^{T}(G(\textbf{D}_{1}))$ to $\mathcal{S}^{T}(G(\textbf{D}_{2}))$. Now let $l:=h_{2}^{-1}\circ f_{\alpha\beta}\circ h_{1}$, where for $i=1,2$, $h_{i}$ are the dBa isomorphisms from $\textbf{D}_{i}$ to $\mathcal{S}^{T}(G(\textbf{D}_{i}))$ defined as in Theorem \ref{RTDBA}. Then $l$ is a dBa isomorphism  from $\textbf{D}_{1}$ to $\textbf{D}_{2}$. We show that $G(l)=h$, that is, $(\alpha_{l},\beta_{l})=(\alpha,\beta)$. Indeed, by Proposition \ref{diagram comut}, $l=h_{2}^{-1}\circ f_{\alpha_{l}\beta_{l}}\circ h_{1}$. Therefore $f_{\alpha\beta}=f_{\alpha_{l}\beta_{l}}$, whence by Proposition \ref{dBa-hom and ctx homeo},  $(\alpha_{l},\beta_{l})=(\alpha,\beta)$.

\noindent (iii) For each $\mathbb{K}^{T}\in Obj(\textbf{Scxt})$, $\mathcal{S}^{T}(\mathbb{K}^{T})\in Obj(\textbf{PDBA})$ by Theorem \ref{clopen dba}. By Theorem \ref{isom for boolean ctx}, $\mathbb{K}^{T}$ is homeomorphic to $\mathbb{K}^{T}_{pr}(\mathcal{S}^{T}(\mathbb{K}^{T}))= F(\mathcal{S}^{T}(\mathbb{K}^{T}))$. Hence $F$ is essential surjective.
\end{proof}

Theorems \ref{algebricdualty} and \ref{duality} give 
\begin{theorem}
\label{fcdbaduallty}
{\rm $\textbf{FCDBA}$ is dually equivalent to $\textbf{Scxt}$.}
\end{theorem}

\section{Conclusions}
\label{conclusion}
In order to give topological representation results for dBas, this work adds topologies to the sets of all primary filters and ideals of dBas and introduces an enhanced version $\mathbb{K}_{pr}^{T}(\textbf{D})$  of the standard context defined by Wille. Contexts with topological spaces, CTSCR, clopen object oriented semiconcepts and protoconcepts   are defined. For every dBa \textbf{D}, it is proved that  $\mathbb{K}_{pr}^{T}(\textbf{D}):=((\mathcal{F}_{pr}(\textbf{D}), \mathcal{T}), (\mathcal{I}_{pr}(\textbf{D}), \mathcal{J}), \nabla)$  is a CTSCR.  The representation results  obtained for dBas are as follows.
Any dBa $\textbf{D}$ is quasi-embeddable into the algebra of clopen object oriented protoconcepts of the CTSCR $\mathbb{K}_{pr}^{T}(\textbf{D})$. The largest pure subalgebra $\textbf{D}_{p}$  of $\textbf{D}$ is isomorphic to the algebra of clopen object oriented semiconcepts of the CTSCR $\mathbb{K}_{pr}^{T}(\textbf{D})$, as a consequence of which any pure dBa $\textbf{D}$ is isomorphic to the algebra of clopen object oriented semiconcepts of  $\mathbb{K}_{pr}^{T}(\textbf{D})$. For a contextual dBa $\textbf{D}$ the quasi-embedding becomes an embedding, and in case $\textbf{D}$ is   fully contextual, it is an isomorphism. When $\textbf{D}$ is a finite dBa, 
it is observed that a representation result  can be  obtained in terms of object oriented protoconcepts and semiconcepts. This representation is also obtained as a special case from the above-mentioned quasi-embedding theorem for dBas. 
%the above-mentioned representation  result   yields the
%existing one for finite dBas obtained by Wille \cite{wille} re-stated in the setting of object oriented protoconcepts and semiconcepts.  

Some observations on Boolean algebras are also made. Trivially, Boolean algebras provide examples of dBas that are both fully contextual and pure, where the Boolean negation serves as both the negations defining a dBa.
%, and the law of double negation holds. 
It is shown here that, on the other hand, if in any dBa the two negations defining it coincide and  the law of double negation holds, it becomes a Boolean algebra. In case of a Boolean algebra $\textbf{D}$, $\nabla$ in the CTSCR $\mathbb{K}_{pr}^{T}(\textbf{D}):=((\mathcal{F}_{pr}(\textbf{D}), \mathcal{T}), (\mathcal{I}_{pr}(\textbf{D}), \mathcal{J}), \nabla)$  is,  in fact, a homeomorphism. The isomorphism theorems for fully contextual and pure dBas yield a representation theorem  for Boolean algebras as well.
%So we can conclude that dBas are one kind of generalization of Boolean algebras by two negations.*

The definition of a Stone context is obtained on  abstraction of  properties  of the  CTSCR  $\mathbb{K}_{pr}^{T}(\textbf{D})$. Categories $\textbf{FCDBA}$,  $\textbf{PDBA}$ and $\textbf{Scxt}$ are defined. Covariant and contravariant functors $G$, $F$ are obtained from $\textbf{FCDBA}$ into $\textbf{PDBA}$ and from $\textbf{PDBA}$ into $\textbf{Scxt}$ respectively. $G$ is shown to be an equivalence, while $F$ is a dual equivalence. The relationships between the categories are summarized in the following diagram.

\begin{figure}[h!]
\centering
\caption{Categorical Relations}
\begin{tikzpicture}
\node (1) at (1.5,0) {\textbf{PDBA}};
\node (2) at (-1.7,0) {\textbf{FCDBA}};
\node (3) at (0,-4.7) {$\textbf{Scxt}^{op}$};
\node(4) at (2.2,-2) {\textbf{$F$}};
\node (5) at (0,1.4) {\textbf{$G$}};
\node(6) at (-2.7,-2) {\textbf{$F\circ G$}};
\draw (1,0) ellipse (2cm and 1cm);
\draw (-1,0) ellipse (2cm and 1cm);
\draw (0,-4)  ellipse (2cm and 1cm);
\draw [->] (1) to [bend left=45] (3);
\draw [->] (2) to [bend right=45] (3);
\draw [->] (2) to [bend left=90] (1);
\end{tikzpicture}

\end{figure}
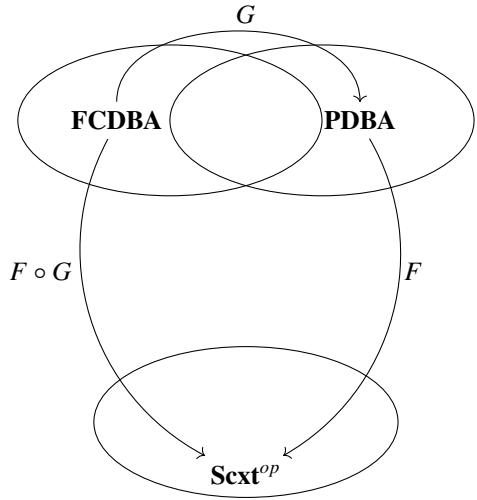

The isomorphism theorem implies that  every pure dBa is a subalgebra of a fully contextual dBa, but it may not be fully contextual itself -- as observed through an example. So the class of fully contextual dBas does not form a variety. On the other hand, fully contextual dBas may not be pure -- observed through the same example. A characterisation of (that is, an isomorphism theorem for) dBas  that are neither fully contextual nor pure, remains an open question. 
 \section{Acknowledgements}
 This work is supported by the \emph{Council of Scientific and Industrial Research} (CSIR) India - Research Grant No. 09/092(0950)/2016-EMR-I.



\end{document}